\documentclass[10pt,a4paper]{amsart}

\RequirePackage{doi}
\usepackage{hyperref}

\usepackage[marginpar=2cm,ignoremp,margin=3cm]{geometry}
\usepackage{changepage}

\usepackage{tikz}
\usepackage{tikz}
\usetikzlibrary{matrix}
\usetikzlibrary{calc}
\usetikzlibrary{decorations.pathreplacing}

\tikzstyle{overbrace text style}=[font=\tiny, above, pos=.5, yshift=3mm]
\tikzstyle{overbrace style}=[decorate,decoration={brace,raise=2mm,amplitude=3pt}]
\tikzstyle{underbrace style}=[decorate,decoration={brace,raise=2mm,amplitude=5pt,mirror},color=gray]
\tikzstyle{underbrace text style}=[font=\tiny, below, pos=.5, yshift=-3mm]

\tikzset{
	bracket/.default=0.1cm,
	bracket/.style={
		to path={
			(\tikztostart) -- ([yshift=#1]\tikztostart) -- ([yshift=#1]\tikztotarget) \tikztonodes -- (\tikztotarget)
		}
	},
	ubracket/.default=0.1cm,
	ubracket/.style={
		to path={
			(\tikztostart) -- ([yshift=-#1]\tikztostart) -- ([yshift=-#1]\tikztotarget) \tikztonodes -- (\tikztotarget)
		}
	}
}

\newcommand\lowt{t}

\DeclareMathOperator{\id}{id}
\DeclareMathOperator{\gr}{gr}

\newcommand{\mot}{\mathfrak{m}}
\newcommand{\umot}{\mathfrak{u}}
\newcommand{\lmot}{\mathfrak{l}}

\usepackage[normalem]{ulem}
\renewcommand{\vec}[1]{\uline{\boldsymbol{\mathbf{#1}}}}

\newcommand{\ttw}{\widetilde{\,t\,}}
\DeclareMathOperator{\Li}{Li}
\DeclareMathOperator{\Ti}{Ti}
\DeclareMathOperator{\LewinTi}{\widetilde{Ti}}

\usepackage{textcomp}
\usepackage{shuffle}

\usepackage{amssymb}
\allowdisplaybreaks

\usepackage{bbm}
\usepackage{mathrsfs}

\usepackage{mathtools}
\usepackage{thmtools}

\declaretheorem[
style=plain,
name=Theorem,
numberwithin=section,
refname={Theorem,Theorems},
Refname={Theorem,Theorems}
]{Thm}
\declaretheorem[
style=plain,
name=Proposition,
numberlike=Thm,
refname={Proposition,Propositions},
Refname={Proposition,Propositions}
]{Prop}
\declaretheorem[
style=plain,
name=Lemma,
numberlike=Thm,
refname={Lemma,Lemmas},
Refname={Lemma,Lemmas}
]{Lem}
\declaretheorem[
style=plain,
name=Corollary,
numberlike=Thm,
refname={Corollary,Corollaries},
Refname={Corollary,Corollaries}
]{Cor}
\declaretheorem[
style=definition,
name=Definition,
numberlike=Thm,
refname={Definition,Definitions},
Refname={Definition,Definitions},
]{Def}
\declaretheorem[
style=definition,
name=Example,
numberlike=Thm,
refname={Example,Examples},
Refname={Example,Examples},
]{Eg}
\declaretheorem[
style=plain,
name=Conjecture,
numberlike=Thm,
refname={Conjecture,Conjectures},
Refname={Conjecture,Conjectures},
]{Conj}
\declaretheorem[
style=definition,
name=Remark,
numberlike=Thm,
refname={Remark,Remarks},
Refname={Remark,Remarks},
]{Rem}

\newcommand{\Z}{\mathbb{Z}}
\newcommand{\Q}{\mathbb{Q}}
\newcommand{\C}{\mathbb{C}}

\newcommand{\ii}{\mathrm{i}}

\newcommand{\abs}[1]{\left\lvert #1 \right\rvert}

\renewcommand{\Re}{\operatorname{Re}}
\renewcommand{\Im}{\operatorname{Im}}

\newcommand\poch[2]{\left\{ #1 \right\}_{#2}}

\newmuskip\pFqmuskip
\newcommand*\pFq[6][8]{%
	\begingroup 
	\pFqmuskip=#1mu\relax
	\mathchardef\normalcomma=\mathcode`,
	\mathcode`\,=\string"8000
	\begingroup\lccode`\~=`\,
	\lowercase{\endgroup\let~}\pFqcomma
	{}_{#2}F_{#3}{\left[\genfrac..{0pt}{}{#4}{#5};#6\right]}%
	\endgroup
}
\newcommand{\pFqcomma}{{\normalcomma}\mskip\pFqmuskip}

\renewcommand{\epsilon}{\varepsilon}

\newcommand{\sgnarg}[2]{
	\left( \genfrac{}{}{0pt}{0}{#1}{#2} \right)
}
\newcommand{\sgnargsm}[2]{
	\left( \textstyle\genfrac{}{}{0pt}{1}{#1}{#2} \right)
}

\DeclareMathOperator{\reg}{reg}
\DeclareMathOperator{\per}{per}

\let\overlineO\overline
\renewcommand{\overline}[1]{\overlineO{\mathclap{\phantom{I}}#1}}

\makeatletter
\let\@@pmod\pmod
\DeclareRobustCommand{\pmod}{\@ifstar\@pmods\@@pmod}
\def\@pmods#1{\mkern4mu({\operator@font mod}\mkern 6mu#1)}
\makeatother

\makeatletter
\@namedef{subjclassname@2020}{%
	\textup{2020} Mathematics Subject Classification}
\makeatother

\begin{document}
	
	\title[On motivic M{\protect\lowt}V's, Saha's basis conjecture, and generators of alternating MZV's]{On motivic multiple \( t \) values, Saha's basis conjecture, \\ and generators of alternating MZV's}
	\author{Steven Charlton}
	\date{29 December 2021}
	
	\address{Fachbereich Mathematik (AZ), Universit\"at Hamburg, Bundesstra\textup{\ss}e 55, 20146 Hamburg, Germany}
	\email{steven.charlton@uni-hamburg.de}
	
	\begin{abstract}
		We give an evaluation for the stuffle-regularised \( t^{\ast,V}(\{2\}^a,1,\{2\}^b) \) as a polynomial in single-zeta values, \( \log(2) \) and \( V \).  We then apply this to establish some linear independence results of certain sets of motivic multiple \( t \) values.  In particular, we prove the elements of Saha's conjectural basis are linearly independent, on the motivic level, and that the (suitably regularised) elements \( t^\mot(\{1,2\}^\times) \) form a basis for both the (extended) motivic MtV's and the alternating MZV's.
	\end{abstract}

	\maketitle

	\setcounter{tocdepth}{1}
	\tableofcontents
	\setcounter{tocdepth}{2}
	
	\section{Introduction}
	
	The \emph{multiple zeta value} (MZV) with indices \( k_1,\ldots,k_d \in \Z_{\geq1} \), and \( k_d \geq 2 \) for convergence reasons, is defined by
	\[
		\zeta(k_1,\ldots,k_d) \coloneqq \sum_{0 < n_1 < \cdots < n_d} \frac{1}{n_1^{k_1} \cdots n_d^{k_d}} \,.
	\]
	As is common, we call \( d \) the \emph{depth} of the MZV and \( k_1 + \cdots + k_d \) the \emph{weight}.  Although the cases  of depth \( d = 1, 2 \) were already studied by Euler, the research into the case of general depth \( d \) only started in the early 1990's with work of Hoffman \cite{hoffman92} and Zagier \cite{zagier94}, with many theorems and identities being proven and many conjectures formulated since then.  These values have also earned a prominent place in high energy physics, as part of the calculation (in special cases) of Feynman integrals and scattering amplitudes (see \cite{broadhurstKreimer97} as a starting point).
	
	In \cite{hoffman19}, Hoffman studied the \emph{multiple \( t \) values} (MtV's) defined by restricting to MZV-like sums with odd denominators,
	\[
		t(k_1,\ldots,k_d) = \sum_{0 < n_1 < \cdots < n_d} \frac{1}{(2n_1-1)^{k_1} \cdots (2n_d-1)^{k_d}} \,,
	\]
	with again \( k_i \in \Z_{\geq1} \) and \( k_d > 1 \) for convergence, and the same notion of weight and depth.  (Be aware that Hoffman uses the other convention on MZV's, with summation indices given by \( n_1 > \cdots > n_d > 0 \), which has the effect of reversing argument strings.)  Therein Hoffman compared and contrasted the algebraic and combinatorial properties of MtV's and MZV's, establishing that MtV's have many similarities with MZV's, but  some distinct differences of there own.  (In particular, both MtV's and MZV's have a stuffle-product, and symmetric sum formulae \cite[Section 3]{hoffman19}.  MZV's admit a duality relation such as \( \zeta(1,2) = \zeta(3) \), but MtV's appear to have no such identities: already \cite[Appendix A]{hoffman19} shows that \( t(3) \) and \( t(1,2) \) are unrelated.  However MtV's conjecturally admit a derivation \cite[Conjecture 2.1]{hoffman19} which is realised in Appendix A therein as a formal differentiation with respect to \( \log(2) \).  We refer to \autoref{rem:diff:t} below for an interpretation and explanation of this derivation as the action of \( D_1 \) on the motivic level.)
	
	By writing
	\[
		t(k_1,\ldots,k_d) \coloneqq  \sum_{0 < n_1 < \cdots < n_d} \frac{(1 - (-1)^{n_1})}{2 \, n_1^{k_1}} \cdots \frac{(1 - (-1)^{n_d})}{2 \, n_d^{k_d}} \,,
	\]
	one obtains a formula (cf. \cite[Corollary 4.1]{hoffman19}) for the MtV's in terms of so-called alternating MZV's.  Namely
	\begin{equation}\label{eqn:tasz}
		t(k_1,\ldots,k_d) = \frac{1}{2^d} \sum_{\epsilon_i \in \{ \pm 1 \}} \epsilon_1 \cdots \epsilon_d \zeta\sgnarg{\epsilon_1,\ldots,\epsilon_d}{k_1,\ldots,k_d} \,,
	\end{equation}
	where
	\[
		\zeta\sgnarg{\epsilon_1,\ldots,\epsilon_d}{k_1,\ldots,k_d} \coloneqq \sum_{0 < n_1 < \cdots < n_d} \frac{\epsilon_1^{n_1} \cdots \epsilon_d^{n_d}}{n_1^{k_1} \cdots n_d^{k_d}}
	\]
	is the \emph{alternating MZV} with signs \( \epsilon_i \) and arguments \( k_i \) (also called the coloured MZV of level \( N = 2 \); level here referring to the order of the roots of unity involved).  Often, if all \( \epsilon_i \in \{ \pm 1 \} \), one denotes arguments \( k_i \) which have associated sign \( \epsilon_i = -1 \) by \( \overline{k_i} \).  There are notions of \emph{regularisation} which we cover more fully in \autoref{sec:reg}, which allow the divergent MZV's and MtV's to be assigned consistent finite values. \medskip
	
	Our first main result, in \autoref{sec:t2212ev}, gives an evaluation for the stuffle-regularised multiple \( t \) value \( t^{\ast,V}(\{2\}^a,1,\{2\}^b) \), analogous to the evaluation for \( \zeta(\{2\}^a, 3, \{2\}^b) \) established by Zagier \cite{zagier2232}, and the evaluation for \( t(\{2\}^a, 3, \{2\}^b) \) established by Murakami \cite{murakami21}, both of which were used to establish the linear independence and basis properties of certain motivic MZV's and MtV's respectively (\cite{brown12} in the zeta-value case, and \cite{murakami21} for the $t$-value case).	
	\begin{Thm}[\autoref{thm:t2212ev} below]\label{thm:intro:t2212ev}
		The following evaluation holds for any \( a, b \in \Z_{\geq0} \), for \( t^{\ast,V} \) the stuffle-regularised MtV's with \( t^{\ast,V}(1) = V \).
		\begin{align*}
			t^{\ast,V}(\{2\}^a,1,\{2\}^b) = {} & -\sum_{r=1}^{a+b} (-1)^r 2^{-2r} \bigg[ \binom{2r}{2a} + \frac{2^{2r}}{2^{2r} - 1} \binom{2r}{2b} \bigg] \zeta(\overline{2r+1}) t(\{2\}^{a + b - r}) \\
		& + \delta_{a=0} \log(2) t(\{2\}^b) + \delta_{b=0} (V - \log(2)) t(\{2\}^a) \,,
		\end{align*}
		where we write \( \{k\}^n = \overbrace{k, \ldots, k}^n \) for the argument \( k \) repeated \( n \) times, and \( \delta_\bullet \) is the Kronecker delta symbol, equal to 1 if the condition \( \bullet \) holds, and 0 otherwise.
	\end{Thm}
	Because a subset of the MtV's which appear in this evaluation series are divergent, we must understand the asymptotics of certain \( {}_3F_2 \) hypergeometric series  (and thus the \( {}_4F_3 \) series from which they originate), via results from the Evans-Stanton/Ramanujan asymptotic \cite{evans84}, in order to extract the evaluation with a generating series approach.  By extracting the special case \( a = 0, b = n \) of \autoref{thm:intro:t2212ev} in \autoref{sec:t1222}, we also answer a question posed in \cite{chavan21}.  \medskip

	We then utilise the arithmetic properties of the coefficients to establish some linear independence properties of certain sets of motivic MtV's. In \autoref{sec:mot} we state the main definitions, properties and theorems pertaining to the framework of motivic MZV's.    One can often think of the motivic MZV's and MtV's as some algebraically defined `formal' analogue to their analytic counterparts, which have more rigid structure and better properties.  Conjecturally the motivic versions reflect all of the relations between the real-valued versions, but they certainly do not introduce new relations.  In \autoref{sec:dt} we discuss the regularised distribution relations, and extend the formulae given by Murakami \cite{murakami21} for \( D_{r} \), to all motivic MtV's.  In \autoref{sec:mott2212} we use this to lift \autoref{thm:intro:t2212ev} to a motivic version.\medskip
	
	Then in \autoref{sec:saha}, we establish that the elements conjectured by Saha \cite{saha17} to be a basis of (convergent) MtV's are, at least, linearly independent.  This establishes a lower bound of \( F_N \) on the dimension of the space of (convergent) motivic MtV's of weight \( N \), where \( F_n = F_{n-1} + F_{n-2} \) is the \( n \)-th Fibonacci number, with \( F_1 = 1, F_2 = 1 \).
	\begin{Thm}[\autoref{cor:saha:indep} below]
		Let
		\[
			S = \{ t^\mot(k_1, \ldots, k_{r-1}, k_r + 1) \mid k_i \in \{ 1,2 \} \, \}
		\]
		be the set of elements in Saha's basis conjecture.  Then the elements of \( S \) are linearly independent.
	\end{Thm}
	For example, this establishes that \( t^\mot(1,2) \) and \( t^\mot(3) \) are linearly independent in weight 3, as are \( t^\mot(1,1,2), t^\mot(2,2) \) and \( t^\mot(1,3) \) in weight 4. \medskip

	Finally in \autoref{sec:hoff} we introduce the Hoffman \( t \) one-two elements -- those MtV's with arguments exactly 1 or 2 -- in analogy with the Hoffman elements \( \zeta(k_1,\ldots,k_r), k_i \in \{ 2,3 \} \) for classical MZV's.  We show the one-two elements form a basis of both (extended) motivic MtV's and alternating motivic MZV's, under a certain shuffle regularisation and certain stuffle regularisation.
	
	\begin{Thm}[\autoref{cor:hoff:indep}, and \autoref{cor:hoff:st:indep}, below]
		Let
		\[
			H^\bullet = \{ t^{\bullet,\mot}(k_1,\ldots,k_r) \mid k_i \in \{ 1,2 \} \, \}
		\]
		be the Hoffman one-two elements, for \( \bullet = \ast \) or \( \bullet = \shuffle \), where the shuffle regularisation arises from \( \zeta^{\shuffle,\mot}(1) = 0 \) and the stuffle regularisation has \( t^{\mot,\ast}(1) = \lambda \log^\mot(2) \), \( \lambda \) of the form \( \frac{2a+1}{b}\in \Q \) with \( a, b \in \Z \).  (In particular, \( \lambda = \frac{1}{2}, 1 \) are allowed.) Then the elements in \( H \) are linearly independent, and span the space of both (extended) motivic MtV's and motivic alternating MZV's.
	\end{Thm}

	As a corollary, we see the space of alternating MZV's of weight \( N \) (under shuffle-regularisation with \( \zeta^{\shuffle,0}(1) = 0 \) or stuffle-regularisation with \( \zeta^{\ast,0}(1) = 0 \) and extended MtV's (under shuffle-regularisation induced by \( \zeta^{\shuffle,0}(1) = 0 \), or stuffle-regularisation with \( t^{\ast,V}(1) = V \) for \( V = \lambda \log(2) \), \( \lambda = \frac{2a+1}{b} \in \Q \), \( a, b \in \Z \)) coincide.  In particular they have dimension \( F_{N+1} \). We also indicate how badly this can fail for certain `singular' regularisation parameters in \autoref{rem:singular:reg}.\medskip
	
	We also include some further examples of the motivic Galois descent (to motivic MZV's) for \( t^\mot(k_1,\dots,k_d) \) for particular families which include arguments \( k_i = 1 \).  For example, in \autoref{prop:t21232} we show
	\[
		t^\mot(\{2\}^a, 1, \{2\}^b, 3, \{2\}^c) \,,
	\]
	is a linear combination of motivic MZV's whenever \( a \geq 1 \).  This shows that Murakami's motivic Galois descent of \( t^\mot(k_1, \ldots, k_d) \), all \( k_i \geq 2 \) \cite[Theorem 8]{murakami21} is not exhaustive, and raises the question of more generally characterising when such a motivic MtV can be written as a linear combination of motivic MZV's.
	
	\subsection*{Acknowledgements} I am grateful to Michael Hoffman for frequent discussions on MtV's, MZV's and various related subjects during his extended stay at the Max-Planck Institute f\"ur Mathematik in Bonn, through spring and summer 2020.  I am also grateful to the MPIM for extended support, which facilitated these discussions with Michael Hoffman and precipitated the start of this work.  I am grateful to both Adam Keilthy and Danylo Radchenko for some discussions and suggestions on how to tackle the hypergeometric series which arise during the evaluation of the \( t(\{2\}^a,1,\{2\}^b) \) generating series.  I was supported by DFG Eigene Stelle grant CH 2561/1-1, for Projektnummer 442093436.

	\section{\texorpdfstring{Relating regularisations of multiple zeta values and multiple \( t \) values}
	{Relating regularisations of multiple zeta values and multiple t values}}
	\label{sec:reg}
	
	In this section we recall, compare and contrast the different notions of regularisation which apply to MtV's.  In particular, we need to understand how the stuffle regularisation of \( t \) values with \( t^{\ast,V}(1) = V \) relates to the regularisation of \( t \) values induced by the stuffle regularisation of the underlying zeta values at \( \zeta^{\ast,U}(1) = U \), or the shuffle regularisation of the zeta values with \( \zeta^{\shuffle,W}(1) = W \).
	
	\subsection{Stuffle and shuffle regularisation for MZV's and MtV's}\label{sec:mzvreg}
	
	As already noted in the introduction, MZV's and MtV's are only defined when the last argument \( k_d \geq 2 \), otherwise the series is divergent.  However, one can use the stuffle product structure to give a consistent definition to any MZV or MtV with trailing 1's, in terms of a single parameter assigned to \( \zeta(1) \coloneqq \zeta^{\ast,U}(1) = U \) or \( t(1) \coloneqq t^{\ast,V}(1) = V \) respectively.  Alternatively one can utilise the iterated integral representation to define a(nother) regularisation.  We briefly recall the details here, see \cite{ikz06} for full details. \medskip
	
	\paragraph{\bf Stuffle-regularisation of MZV's:} One considers the truncated MZV
	\[
		\zeta_M(k_1,\ldots,k_d) = \sum_{0 < n_1 < \cdots < n_d \leq M} \frac{1}{n_1^{s_1} \cdots n_d^{s_d}} \,.
	\]
	It is well-known that as \( M \to \infty \)
	\[
		\zeta_M(1) = \sum_{n=1}^M \frac{1}{n} = \log(M) + \gamma + O\bigg( \frac{1}{M} \bigg) \,,
	\]
	where \( \gamma = 0.577\ldots \) is the Euler-Mascheroni constant.  The stuffle-product formula (which both truncated and non-truncated MZV's satisfy, including alternating versions by multiplying the signs, as well as MtV's and truncated versions thereof) corresponds to interleaving the summation indices with equality allowed, for example
	\begin{align*}
		\zeta_M(k_1,k_2) \zeta_M(\ell_1) &= \sum_{0 < n_1 < n_2 < M} \sum_{0 < m_1 < M} \frac{1}{n_1^{k_1} n_2^{k_2}} \cdot \frac{1}{m_1^{\ell_1}} \\
		&= \begin{aligned}[t] 
			\bigg( \sum_{0 < n_1 < n_2 < m_1 < M} \! + \! & \sum_{0 < n_1 < m_1 < n_2 < M} \! + \!  \sum_{0 < m_1 < n_1 < n_2 < M} \!  \\[-0.5ex]
			& {} + \! \sum_{0 < n_1 < n_2 = m_1 < M} \! + \!  \sum_{0 < n_1 = m_1 < n_2 < M}\bigg)\frac{1}{n_1^{k_1} n_2^{k_2}} \cdot \frac{1}{m_1^{\ell_1}} \end{aligned} \\
		&= \zeta(k_1,k_2,\ell_1) + \zeta(k_1,\ell_1,k_2) + \zeta(\ell_1,k_1,k_2) + \zeta(k_1,k_2+\ell_1) + \zeta(k_1 + \ell_1, k_2) \,.
	\end{align*}  By applying this stuffle-product one sees by induction that for any indices \( \vec{k} = (k_1,\ldots,k_d) \)
	\[
		\zeta_M(\vec{k}) = Z^{\ast}(\vec{k}; \log(M) + \gamma) + O\bigg(\frac{\log^J{M}}{M}\bigg)
	\]
	for some polynomial \( Z^{\ast}(\vec{k}; T) \) with convergent MZV coefficients, and some \( J \).  Then the stuffle-regularised MZV with parameter \( \zeta^{\ast,U}(1) = U \) is defined to be this polynomial
	\[
		\zeta^{\ast,U}(\vec{k}) = Z^\ast(\vec{k}; U) \,.
	\]
	More precisely, once recursively computes this polynomial \( Z^\ast(\vec{k}; T) \) by extending the stuffle product formally to allow trailing 1's, and considering
	\[
		\zeta^{\ast,U}(k_1,\ldots,k_d, \{1\}^\alpha) - \tfrac{1}{\alpha} \zeta^{\ast,U}(k_1,\ldots,k_d,\{1\}^{\alpha-1}) \zeta^{\ast,U}(1)
	\]
	which has strictly fewer trailing 1's after expanding out.  This recursively writes every MZV as a polynomial in \( \zeta^{\ast,U}(1) = U \) with convergent MZV coefficients, with stuffle-regularised MZV's still obeying the stuffle-product formula. 
	
	For example, for \( a \neq 1 \):
	\begin{equation}\label{eqn:reg:za1}
	\begin{aligned}
		\zeta^{\ast,U}(a, 1) & = \zeta^{\ast,U}(a) \zeta^{\ast,U}(1) - \zeta^{\ast,U}(1,a) - \zeta^{\ast,U}(a+1) \\
		 & = \zeta(a) U - \zeta(1,a) - \zeta(a+1)
	\end{aligned}
	\end{equation}
	Then
	\begin{equation}\label{eqn:reg:za11}
	\begin{aligned}
	\zeta^{\ast,U}(a, 1, 1) & = \tfrac{1}{2} \zeta^{\ast,U}(a,1) \zeta^{\ast,U}(1) - \zeta^{\ast,U}(1,a,1) - \zeta^{\ast,U}(a,2) - \zeta^{\ast,U}(a+1,1) \\
		&= \begin{aligned}[t]
			\tfrac{1}{2} U^2 \zeta(a) - U \zeta(a+1)  - U \zeta(1,a) + \frac{1}{2} \zeta(a+2) \\
			+ \zeta(1, a+1) + \tfrac{1}{2} \zeta(2,a) - \tfrac{1}{2} \zeta(a,2) + \zeta(1,1,a) \,.
		\end{aligned}
	\end{aligned}
	\end{equation}
	\medskip
	
	\paragraph{\bf Stuffle-regularisation of MtV's:} The same stuffle-product structure is be used to define the stuffle-regularised MtV's \( t^{\ast,V}(\vec{k}) \) with regularisation parameter \( t^{\ast,V}(1) = V \).  That is to say, both \eqref{eqn:reg:za1} and \eqref{eqn:reg:za11} hold for \( \zeta \) replaced everywhere by \( t \).
	
	Moreover, one can also argue for the regularisation \( t^{\ast,T}(1)  = \log(2) \) via the asymptotic behaviour of the truncated
	\[
	t_M(1) = \sum_{0 \leq k < M} \frac{1}{2k+1} \,.
	\]
	The known asymptotics of
	\[
	\zeta_M(1) = \sum_{1 \leq k < M} \frac{1}{k} = \log(M) + \gamma + O\Big(\frac{1}{M}\Big)
	\]
	lead one to define a natural regularisation of \( \zeta^{\ast,T}(1) \) as the constant term in the above, when viewed as a polynomial in \( \log(M) + \gamma \).  I.e. the regularised version of \( \zeta^{\ast,T}(1) \) is most naturally given via \( \zeta^{\ast,T}(1) = 0 \).  Applying the same to \( t_M(1) \) gives
	\begin{align*}
	t_M(1) = \sum_{0 \leq k < M} \frac{1}{2k+1} &= \sum_{1 \leq k < 2M} \frac{1}{k} - \sum_{1 \leq k < M} \frac{1}{2k} \\
	&= \Big(\log(2M) + \gamma + O\Big(\frac{1}{2M}\Big)\Big) - \frac{1}{2} \Big( \log(M) + \gamma + O\Big(\frac{1}{M}\Big) \Big) \\
	&= \log(2) + \frac{1}{2} \big( \log(M) + \gamma \big) + O\Big( \frac{1}{M} \Big)
	\end{align*}
	So one can naturally define 
	\begin{equation}
	\label{eqn:t1reg}
	 t^{\ast,T}(1) = \log(2) 
	 \end{equation}
	 via the constant term of the above, viewed as a polynomial in \( \log(M) + \gamma \).
	
	Since the space of weight 1 MtV's \( t(k_1,\ldots,k_d) \) with \( k_1 + \ldots + k_d = 1 \) and \( k_d > 1 \) is 0 dimensional, we already see that defining \( t^{V,\ast}(1) = \log(2) \) extends the space of MtV's.  We make the following definition, for clarity, with regard to convergent and regularised MtV's.
	\begin{Def}[Convergent and extended MtV's]\label{def:mtvconv}
		We call a multiple \( t \) value \( t(k_1,\ldots,k_d) \)
		\begin{itemize}
			\item[i)] \emph{convergent} if \( k_d \geq 2 \), and
			\item[ii)] \emph{extended} if \( k_d \geq 1 \)
		\end{itemize}
		The \emph{space of convergent MtV's} refers the space generated by all MtV's \( t(k_1,\ldots,k_d) \) which have \( k_d \geq 1 \), whereas the \emph{space of extended MtV's} denotes that generated by all MtV's \( t(k_1,\ldots,k_d) \) allowing \( k_d > 1 \) or \( k_d = 1 \), under some particular (specified) regularisation.
	\end{Def}

	\paragraph{\bf Shuffle regularisation of MZV's:}
	
	Likewise, it is well know how to write MZV's (or more generally multiple polylogarithms) as iterated integrals.  Namely
	\begin{equation}
	\label{eqn:zasi}
		\zeta\sgnarg{\epsilon_1,\ldots,\epsilon_d}{k_1,\ldots,k_d} = 
		 (-1)^d I(0; \eta_1, \{0\}^{k_1-1}, \eta_2, \{0\}^{k_2-1}, \ldots, \eta_d, \{0\}^{k_d-1}; 1)
	\end{equation}
	where the notation \( \{k\}^n = \overbrace{k, \ldots, k}^n \) denotes \( k \) repeated \( n \) times, \( \eta_i = \prod_{j=1}^d \epsilon_j \), and
	\begin{equation}\label{eqn:itint}
		I(x_0; x_1,\ldots,x_N; x_{N+1}) \coloneqq \int_{x_0 < s_1 < \cdots < s_N < x_{n+1}} \omega(x_1; s_1) \wedge \cdots \wedge \omega(x_N; s_N) \,,
	\end{equation}
	is the iterated integral of the family of differential forms \( \omega(x; s) = \frac{\mathrm{d}s}{s - x} \).
	
	Integrals of this form are multiplied by the shuffle product, corresponding to interleaving the integration indices, although here equality of indices gives sets of measure zero, and so no contribution to the result.  For example
	\begin{align*}
		\zeta(2) \zeta(2) & = I(0; 1, 0; 1) I(0; 1, 0; 1) \\
		&= \int_{0 < s_1 < s_2 < 1} \frac{\mathrm{d}s_1}{s_1 - 1} \wedge  \frac{\mathrm{d}s_2}{s_2} \int_{0 < s_3 < s_4 < 1} \frac{\mathrm{d}s_3}{s_3-1} \wedge  \frac{\mathrm{d}s_4}{s_4} \\
		&= \begin{aligned}[t] 
			& \int_{0 < s_1 < s_2 < r_1 < r_2} \frac{\mathrm{d}s_1}{s_1 - 1} \wedge  \frac{\mathrm{d}s_2}{s_2} \wedge  \frac{\mathrm{d}r_1}{r_1 - 1} \wedge  \frac{\mathrm{d}r_2}{r_2} \\
			& \quad + \int_{0 < s_1  < r_1 < s_2< r_2} \frac{\mathrm{d}s_1}{s_1 - 1} \wedge  \frac{\mathrm{d}r_1}{r_1 - 1} \wedge  \frac{\mathrm{d}s_2}{s_2} \wedge  \frac{\mathrm{d}r_2}{r_2} + \text{4 more terms} \end{aligned} \\
		&= \zeta(2,2) + \zeta(1,3) + \text{4 more terms} \\
		&= 2\zeta(2,2) + 4\zeta(1,3)
	\end{align*}
	
	The integral representation only converges when \( x_0 \neq x_1 \) and \( x_n \neq x_{n+1} \), however by formally extending the shuffle product to these cases, one can write any \( I(0; x_1, \ldots, x_n; 1) \) as a polynomial in \( I(0; 0; 1) \) and \( I(0; 1; 1) \), with convergent MZV coefficients.  Recursively, one has that if \( x_i \neq 1 \), then
	\[
		I(0; x_1,\ldots,x_i, \{1\}^\alpha; 1) - \tfrac{1}{\alpha} I(0; x_1,\ldots,x_i, \{1\}^{\alpha-1}; 1) I(0; 1; 1) 
	\]
	ends in fewer trailing 1's, so eventually such an expression with no trailing 1's is obtained.  Likewise, if \( x_i \neq 0 \), then
	\[
		I(0; \{0\}^\beta, x_1,\ldots,x_i; 1) - \tfrac{1}{\beta} I(0; \{0\}^{\beta-1}, x_1,\ldots,x_i; 1) I(0; 0; 1) 
	\]
	ends in fewer starting 0's, so eventually one obtains an expression with no starting 0's, and also no trailing 1's.  On account of the duality relation, setting \( t_i \mapsto 1-t_i \) in the integrated integral expression, one typically imposes that \( I(0; 0; 1) = I(0; 1;1) \) in the regularisation procedure.  One then converts this expression back to MZV's, in order to obtain the shuffle regularisation polynomial for \( \zeta^{\shuffle,W}(\vec{k}) \), with parameter \( \zeta^{\shuffle,W}(1) = W = -I(0; 0; 1) = -I(0; 1; 1) \)
	
	One also extends the notion of MZV to allow initial 0's in the integral with the notation
	\[
		\zeta_\ell\sgnarg{\epsilon_1,\ldots,\epsilon_d}{k_1,\ldots,k_d} = (-1)^d I(0; \{0\}^{\ell}, \eta_1, \{0\}^{k_1-1}, \eta_2, \{0\}^{k_2-1}, \ldots, \eta_d, \{0\}^{k_d-1}; 1) \,,
	\]
	which can be simplified to \( \zeta_\ell(k_1,\ldots,k_d) \) if all \( \epsilon_i = 1 \), or one writes \( \overline{k_i} \) to denote the signs \( \epsilon_i = -1 \) in position \( i \).
	
	For example, we compute the regularisation polynomial of the following MZV's to be as given
	\begin{align*}
		\zeta^{\shuffle,W}(2,1,1) &= \frac{1}{2} W^2\zeta(2) - 2 W \zeta(1,2)  + 3 \zeta(1,1,2) \,, \\
		\zeta_1^{\shuffle,W}(2,1,1) &= \begin{aligned}[t] 
		& -\frac{1}{2} W^3 \zeta(2) - W^2 \zeta(3) + 2 W^2 \zeta(1,2)  + 4 W \zeta(1,3) + W \zeta(2,2) - 3 W \zeta(1,1,2) \\
		& - 6 \zeta(1,1,3) - 2 \zeta(1,2,2)- \zeta(2, 1, 2) - \zeta(2, 1, 2)
		\end{aligned}
	\end{align*}
	We already remark here that \( \zeta^{\shuffle,W}(\vec{k}) \) and \( \zeta^{\ast,T}(\vec{k}) \) are not the same even for \( W = T \), but they are closely related, as we recall in \autoref{lem:regshtost} below.
	
	In the case \( W = 0 \), one has the following explicit formula to unshuffle and regularise the initial 0's.  (Compare property I2 in Section 2.4 of \cite{brown12}) 
	\begin{Lem}[Unshuffling of initial 0's]\label{lem:unshufflestart0}
		For any \( \ell \geq 0 \), any \( \vec{k} = (k_1,\ldots,k_d) \in \Z_{\geq1} \) and any signs \( \epsilon_i \), we have
		\begin{equation}\label{eqn:unshufflestart0}
			\zeta_\ell^{\shuffle,0}\sgnarg{\vec{\epsilon}}{\vec{k}} = (-1)^\ell \sum_{i_1 + \cdots + i_d = \ell} \binom{k_1 + i_1 - 1}{i_1} \cdots \binom{k_d + i_d - 1}{i_d} \zeta^{\shuffle,0}\sgnarg{\vec{\epsilon}}{k_1 + i_1, \ldots, k_d + i_d} \,.
		\end{equation}		
		
		\begin{proof}
			This is a straightforward proof by induction on \( \ell \), the case \( \ell = 0 \) is clearly true.  Since \( \zeta^{\shuffle,0}(1) = \zeta_1^{\shuffle,0}(\emptyset) = I(0; 0; 1) = I(0; 1; 1) = 0 \), the product terms can be neglected after expanding out 
			\[
				I(0; \{0\}^\ell, x_1,\ldots,x_i; 1) - \tfrac{1}{\ell} I(0; \{0\}^{\ell-1}, x_1,\ldots,x_i; 1) I(0; 0; 1) \,.
			\]
			to reduce to fewer initial 0's.
		\end{proof}
	\end{Lem}

	\paragraph{\bf Shuffle-regularisation of MtV's:} In contrast to MZV's, MtV's have no good integral representation, at least not one which endows them with a nice shuffle-product structure.	For example, one sees that in depth 1, we have
	\[
		t(2) = - \int_{0 < s_1 < s_2 < 1} \frac{\mathrm{d}s_1}{s_1^2 - 1} \wedge \frac{\mathrm{d}s_2}{s_2} \,,
	\]
	whereas in depth 2
	\[
		t(2,2) = \int_{0 < s_1 < s_2 < s_3 < s_4 < 1} \frac{\mathrm{d}s_1}{s_1^2 - 1} \wedge \frac{\mathrm{d}s_2}{s_2} \wedge \frac{s_3 \mathrm{d}s_3}{s_3^2 - 1} \wedge \frac{\mathrm{d}s_4}{s_4} \,.
	\]
	In particular, the first differential form is of a unique type, and any further forms corresponding to higher depth and more arguments have an extra \( s_i \) in the numerator
	\[
		\frac{s_i \mathrm{d}s_i}{s_i^2 - 1}  \text{ verses } \frac{\mathrm{d}s_1}{s_1^2 - 1} \,.
	\]
	Therefore, after taking the shuffle product of two copies this representation of \( t(2) \), one has to disentangle cases the forms when
	\[
		\frac{\mathrm{d}s_j}{s_j^2 - 1} 
	\]
	appears after the first position, and when
	\[
		\frac{s_1 \mathrm{d}s_1}{s_1^2 - 1} 
	\]
	appears in the first position.
	
	To cut directly to the point, we take the expression for \( t \) in terms of \( \zeta \) from \eqref{eqn:tasz}, and use this as the basis for defining the shuffle regularised version of \( t \) by shuffle regularising the MZV's therein
	\begin{equation}\label{eqn:tshuffle}
	t^{\shuffle,W}(k_1,\ldots,k_d) = \frac{1}{2^d} \sum_{\epsilon_i \in \{ \pm 1 \}} \epsilon_1 \cdots \epsilon_d \zeta^{\shuffle,W}\sgnarg{\epsilon_1,\ldots,\epsilon_d}{k_1,\ldots,k_d} \,,
	\end{equation}
	
	\subsection{\texorpdfstring{Compatibility of the \( t \) and zeta stuffle-regularisations}
	{Compatibility of the t and zeta stuffle regularisations}}
	
	Firstly, we note that the stuffle product of \( t \) values and of the underlying zeta values is compatible in the following sense.
	
	\begin{Lem}
		Let \( \vec{r} , \vec{s} \) be two sets of indices.  Then
		\[
			t(\vec{r} \ast_t \vec{s}) = t(\vec{r}) \ast_\zeta t(\vec{s}) \,,
		\]
		where \( \ast_t \) on the left denotes the stuffle product of MtV's, and \( \ast_\zeta \) on the right denotes the stuffle product of the MZV's in the MZV expression for \( t(\vec{r}) \) and \( t(\vec{s}) \).
		
		\begin{proof}
			This is essentially a tautology on the level of the series definition of \( t(\vec{r}) \), after inserting the factors \( \frac{1}{2} (1 - (-1)^{n_i}) \), \( i = 1, \ldots, d \) into the numerator, so that the summation range can be extended to all integers.  Stuffle product on the level of series corresponds to interleaving the two sets of summation induces \( n_1 < \cdots < n_d \) and \( m_1 < \cdots < m_{d'} \) in compatible ways.  The conversion from MtV to MZV form, and vice versa, corresponds by taking \( n_i, m_j \) odd positive integers or \( n_i, m_j \) any positive integers, which does not affect the form of any of the compatible interleavings. We also refer to \cite[Section 4]{hoffman19}, for a more precise formulation.  \medskip
			
			Nevertheless, writing out the case \( t((a) \ast_t (b,c)) \) verses \( t(a) \ast_\zeta t(b,c) \) is illuminating.  Directly from the ($t$-)stuffle product of \( t \) values, we have
			\[
				t((a) \ast_t (b,c)) = t(a,b,c) + t(b,a,c) + t(b,c,a) + t(a+b,c) + t(b,a+c) \,.
			\]
			
			On the other hand, by the expansion of \( t \) in terms of \( \zeta \) from \eqref{eqn:tasz}, we have
			\begin{align*}
				t(a) &= \frac{1}{2} \big(\zeta(a) - \zeta(\overline{a}) \big)  \\
				t(b,c) &= \frac{1}{4} \big(\zeta(b,c) - \zeta(\overline{b},c) - \zeta(b,\overline{c}) + \zeta(\overline{b},\overline{c}) \big) 
			\end{align*}
			Computing the (zeta-)stuffle product of these combinations gives
			\begin{align*}
				& \frac{1}{2} \big(\zeta(a) - \zeta(\overline{a}) \big) \ast \frac{1}{4} \big(\zeta(b,c) - \zeta(\overline{b},c) - \zeta(b,\overline{c}) + \zeta(\overline{b},\overline{c}) \big)  \\
				& = \frac{1}{8} \big(
					\begin{aligned}[t]
						& (\zeta(a,b,c) + \zeta(b,a,c) + \zeta(b,c,a) + \zeta(a+b,c) + \zeta(b,a+c)) \\
						& - (\zeta(a,\overline{b},c) + \zeta(\overline{b},a,c) + \zeta(\overline{b},c,a) + \zeta(\overline{a+b},c) + \zeta(\overline{b},a+c)) \\
						& - (\zeta(a,b,\overline{c}) + \zeta(b,a,\overline{c}) + \zeta(b,\overline{c},a) + \zeta(a+b,\overline{c}) + \zeta(b,\overline{a+c})) \\
						& + (\zeta(a,\overline{b},\overline{c}) + \zeta(\overline{b},a,\overline{c}) + \zeta(\overline{b},\overline{c},a) + \zeta(\overline{a+b},\overline{c}) + \zeta(\overline{b},\overline{a+c})) \\[1ex]
						& - (\zeta(\overline{a},b,c) + \zeta(b,\overline{a},c) + \zeta(b,c,\overline{a}) + \zeta(\overline{a+b},c) + \zeta(b,\overline{a+c})) \\
						& + (\zeta(\overline{a},\overline{b},c) + \zeta(\overline{b},\overline{a},c) + \zeta(\overline{b},c,\overline{a}) + \zeta(a+b,c) + \zeta(\overline{b},\overline{a+c})) \\
						& + (\zeta(\overline{a},b,\overline{c}) + \zeta(b,\overline{a},\overline{c}) + \zeta(b,\overline{c},\overline{a}) + \zeta(\overline{a+b},\overline{c}) + \zeta(b,a+c)) \\
						& - (\zeta(\overline{a},\overline{b},\overline{c}) + \zeta(\overline{b},\overline{a},\overline{c}) + \zeta(\overline{b},\overline{c},\overline{a}) + \zeta(a+b,\overline{c}) + \zeta(\overline{b},a+c))
				\big)
			\end{aligned}
			\end{align*}
			One sees immediately that each depth 3 MZV occurs with all 8 choices of bars, a different one in each line, whereas each depth 2 MZV occurs with all 4 choices of bars, but each is repeated twice.  In particular, we find this combination is equal to (reading down the columns)
			\[
				= t(a,b,c) + t(b,a,c) + t(b,c,a) + t(a+b,c) + t(b,a+c) \,,
			\]
			as expected.  (The multiplicity of 2 in depth 2 terms cancels with a left over \( \frac{1}{2} \) from the leading \( \frac{1}{8} = 2^{-3} \) coefficient, since only \( \frac{1}{4} = 2^{-2}\) is used in the conversion of depth 2 MtV's to MZV's.)
		\end{proof}
	\end{Lem}
	
	\begin{Cor}\label{cor:tregtozreg}
		Stuffle regularisation with parameter \( t^{\ast,V}(1) = V \) on the level of MtV's corresponds to stuffle regularisation of the underlying MZV's with parameter \( \zeta^{\ast,U}(1) = U \), for \( U = 2V - \log(2) \).
		
		\begin{proof}
			If the MZV's are regularised with parameter \( \zeta^{\ast,U}(1) = U \), then for agreement with the MtV regularisation with parameter \( t^{\ast,V}(1) = V \) we must have
			\[
				V = t^{\ast,V}(1) = \frac{1}{2} ( \zeta^{\ast,U}(1) - \zeta(\overline{1}) ) = \frac{1}{2} ( U + \log(2) )  \,.
			\]
			Hence the relation \( U = 2V - \log(2) \) follows.
		\end{proof}
	\end{Cor}

	In some sense, this means the most natural regularisation for multiple \( t \) values, when defined formally as a sum of alternating MZV's via \eqref{eqn:tasz}, has \( t^{\ast,T}(1) = \frac{1}{2} \log(2) \).  We already saw in \eqref{eqn:t1reg} above that \( t^{\ast,T}(1) = \log(2) \) is another very natural regularisation for MtV's, and so these will be the two cases of most interest.

	\subsection{Relations between regularisations of alternating MZV's}
	
	We establish (or recall) some relationships between regularisations with different parameters, and between the shuffle and stuffle regularisations of \emph{alternating} MZV's, which will be useful when applied to MtV's.  In the Lemmas the follow, let \( \vec{k} = (k_1,\ldots,k_d) \) be an index (possibly with barred entries), such that \( k_d \neq 1 \).
	
	\begin{Lem}\label{lem:regastTtoastS}
		The stuffle regularisations with parameter \( \zeta^{\ast,T}(1) = T \) and \( \zeta^{\ast,S}(1) = S \) are related as follows, 
		\[
			\zeta^{\ast,T}(\vec{k}, \{1\}^\alpha) = \sum_{i=0}^\alpha \zeta^{\ast,S}(\vec{k}, \{1\}^{\alpha - i}) \frac{(T-S)^i}{i!} \,.
		\]
		
		\begin{proof}
			We actually prove a stronger statement, which claims that this regularisation formula holds for MZV's at arbitrary roots of unity.  Consider the alphabet  \( Z = \{ z_{\theta,n} \mid n \geq 1 \in \Z, \theta \in \C, \abs{\theta} = 1 \} \), with letter product \( z_{\theta,n} \diamond z_{\phi,m} = z_{\theta\phi,n+m} \) on \( \Q Z \), and the induced stuffle product on \( \Q\langle Z \rangle \) given by 
			\[
			(z_1 w) \ast (z_2 v) = (z_1 \diamond z_2)(w \ast v) + z_1 (w \ast z_2 v) + z_2 (z_1 w \ast v) \,.
			\]
			Then \( \ast \) describes the product of multiple zeta values (at arbitrary roots of unity) under the map \( z_{\theta_1,n_1}\cdots z_{\theta_d,n_d} \mapsto \zeta\sgnargsm{\theta_1,\ldots,\theta_d}{n_1,\ldots,n_d} \).
			By expanding out with the stuffle product we see the following.  For any convergent word \( w_0 = w_0' z_{\theta,n} \) with \( (\theta,n) \neq (1,1) \),
			\[
			\sum_{i=0}^\alpha \frac{(-1)^i}{i!} z_{1,1}^{\ast i} \ast (w_0 z_{1,1}^{\alpha-i}) .
			\]
			is a sum of purely convergent words; there is a pairwise cancellation of any words ending in \( z_{1,1} \).  On the other hand, this expression is a stuffle-polynomial in \( z_{1,1} \), whose constant term is the word \( w_0 z_{1,1}^\alpha \).  In the regularisation where \( z_{1,1} \mapsto 0 \), only the constant term of this polynomial is left, and we see
			\[
			\reg_{0}^\ast (w_0 z_{1,1}^\alpha) = \sum_{i=0}^\alpha \frac{(-1)^i}{i!} z_{1,1}^{\ast i} \ast w_0 z_{1,1}^{\alpha-i} \,.
			\]
			By substituting the above expression for the case \( \reg_{0}^\ast w_0 z_{1,1}^{\alpha-i} \) into the following, and switching the order of summation, we see
			\[
			\sum_{i=0}^\alpha \frac{1}{i!} \reg_0^\ast( w_0 z_{1,1}^{\alpha-i} ) \ast z_{1,1}^{\ast i} = w_0 z_{1,1}^\alpha \,.
			\]
			Since the left hand side is now a polynomial in \( z_{1,1} \) with convergent coefficients (by virtue of being a regularised expression already), we can apply the regularisation map with \( z_{1,1} \mapsto T \) to obtain (in zeta notation already)
			\[
			\zeta^{\ast,T}\sgnarg{\vec{\phi},\{1\}^\alpha}{\vec{n},\{1\}^\alpha} = \sum_{i=0}^\alpha \zeta^{\ast,0}\sgnarg{\vec{\phi},\{1\}^{\alpha-i}}{\vec{n},\{1\}^{\alpha-i}} \frac{T^i}{i!} \,,
			\]
			where the letters in \( w_0 = z_{\phi_1,n_1} \cdots z_{\phi_d,n_d} \) induce the arguments \( \vec{n} \) with signs \( \vec{\phi} \) in the MZV's.  (Note that this formula is already established for classical MZV's in \cite[Proposition 10, Equation (5.10), and Corollary 5]{ikz06}.)  
			
			Now multiply both sides of the preceding equation by \( u^\alpha \) and sum on \( \alpha \) to obtain
			\begin{align*}
			\sum_{\alpha=0}^\infty \zeta^{\ast,T}\sgnarg{\vec{\phi},\{1\}^\alpha}{\vec{n},\{1\}^\alpha} \, u^\alpha
			&= \sum_{\alpha=0}^\infty \sum_{i=0}^\alpha \zeta^{\ast,0}\sgnarg{\vec{\phi},\{1\}^{\alpha-i}}{\vec{n},\{1\}^{\alpha-i}} \frac{u^\alpha T^i}{i!} \\
			&= \sum_{\alpha=0}^\infty \zeta^{\ast,0}\sgnarg{\vec{\phi},\{1\}^{\alpha}}{\vec{n},\{1\}^{\alpha}} u^\alpha \cdot \sum_{\alpha=0}^\infty \frac{u^i T^i}{i!}\\
			&= \sum_{\alpha=0}^\infty \zeta^{\ast,0}\sgnarg{\vec{\phi},\{1\}^{\alpha}}{\vec{n},\{1\}^{\alpha}} u^\alpha \cdot \exp(T u) .
			\end{align*}
			From this we see 
			\[
			\sum_{\alpha=0}^\infty \zeta^{\ast,T}\sgnarg{\vec{\phi},\{1\}^\alpha}{\vec{n},\{1\}^\alpha} \, u^\alpha \cdot \exp(-T u)
			\]
			is independent of \( T \), so by equating the \( T \) and the \( S \) regularisation we obtain
			\[
			\sum_{\alpha=0}^\infty \zeta^{\ast,T}\sgnarg{\vec{\phi},\{1\}^\alpha}{\vec{n},\{1\}^\alpha} \, u^\alpha
			= 	\sum_{\alpha=0}^\infty \zeta^{\ast,S}\sgnarg{\vec{\phi},\{1\}^\alpha}{\vec{n},\{1\}^\alpha} \, u^\alpha \cdot \exp((T-S) u) \,.
			\]
			Comparing the coefficient of \( u^\alpha \) establishes the claim for MZV's at arbitrary roots of unit; when \( \vec{\phi} \in \{ \pm 1 \} \), one reduces to the case of alternating MZV's as stated in the lemma.
		\end{proof}
	\end{Lem}

	\begin{Lem}\label{lem:regshTtoshS}
	The shuffle regularisations with parameter \( \zeta^{\ast,T}(1) = T \) and \( \zeta^{\ast,S}(1) = S \) are related as follows, 
	\[
	\zeta^{\shuffle,T}(\vec{k}, \{1\}^\alpha) = \sum_{i=0}^\alpha \zeta^{\shuffle,S}(\vec{k}, \{1\}^{\alpha - i}) \frac{(T-S)^i}{i!} \,.
	\]
	
		\begin{proof}
			The proof of this is analogous to the above proof for the stuffle regularisation (and is also shown in the case of MZV's at arbitrary roots of unity).  Namely consider the alphabet \( Y = \{ e_0 \} \cup \{ e_\eta \mid \eta \in \C, \abs{\eta} = 1 \} \).  Then under the induced shuffle product
			\[
				(e_a w) \shuffle (e_b v) = e_a (w \shuffle e_b v) + e_a (e_b w \shuffle v)
			\]
			the algebra \( \Q\langle Y \rangle \) encodes the shuffle product of MZV's (at arbitrary roots of unity) under the map
			\begin{align*}
				e_0^{k}e_{\eta_1}e_0^{n_1-1} \cdots e_{\eta_d}e_0^{n_d-1} \mapsto & (-1)^d I(0; \{e_0\}^{k}, e_{\eta_1}, \{e_0\}^{n_1-1}, \ldots, e_{\eta_d}, \{0\}^{n_d-1}; 1) \\
				& = \zeta_k\sgnarg{\eta_2/\eta_1, \eta_3/\eta_2, \ldots, 1/\eta_d}{\mathrlap{\,n_1}\phantom{\eta_2/\eta_1},\mathrlap{\,n_2}\phantom{\eta_3/\eta_2},\ldots,\mathrlap{\,n_d}\phantom{1/\eta_d}} \,,
			\end{align*}
			(where the 1 in \( 1/\eta_d \) in the last sign comes from the upper bound of the integral).
			
			For any convergent word \( w_0 = e_a w_0' e_b \) with \( e_a \neq e_0 \) and \( e_b \neq e_1 \),
			\[
			\sum_{i=0}^\alpha \frac{(-1)^i}{i!} e_1^{\shuffle i} \shuffle (w_0 e_1^{\alpha-i}) .
			\]
			is a sum of purely convergent words; there is a pairwise cancellation of any words ending in \( e_1 \).  On the other hand, this expression is a shuffle-polynomial in \( e_1 \), whose constant term is the word \( w_0 e_1^\alpha \).  In the regularisation where \( e_1 \mapsto 0 \), only the constant term of this polynomial is left, and we see
			\[
			\reg_{0}^\shuffle (w_0 e_!^\alpha) = \sum_{i=0}^\alpha \frac{(-1)^i}{i!} e_1^{\shuffle i} \shuffle w_0 e_1^{\alpha-i} \,.
			\]
			By substituting the above expression for the case \( \reg_{0}^\shuffle w_0 e_1^{\alpha-i} \) into the following, and switching the order of summation, we see
			\[
			\sum_{i=0}^\alpha \frac{1}{i!} \reg_0^\shuffle( w_0 e_1^{\alpha-i} ) \shuffle e_1^{\shuffle i} = w_0 e_1^\alpha \,.
			\]
			Since the left hand side is now a polynomial in \( e_1 \) with convergent coefficients (by virtue of being a regularised expression already), we can apply the regularisation map with \( e_1 \mapsto T \) to obtain (in zeta notation already)
			\[
			\zeta^{\shuffle,T}\sgnarg{\vec{\phi},\{1\}^\alpha}{\vec{n},\{1\}^\alpha} = \sum_{i=0}^\alpha \zeta^{\shuffle,0}\sgnarg{\vec{\phi},\{1\}^{\alpha-i}}{\vec{n},\{1\}^{\alpha-i}} \frac{T^i}{i!} \,,
			\]
			where the letters in \( w_0 = e_{\eta_1}e_0^{n_1-1} \cdots e_{\eta_d}e_0^{n_d-1}  \) induce the arguments \( \vec{n} \) with signs \( \vec{\phi} \) in the MZV's.  (Note that this formula is already established for classical MZV's in \cite[Proposition 10, Equation (5.9), and Corollary 5]{ikz06}.)  
			
			Now multiplying both sides of the preceding equation by \( u^\alpha \), and summing on \( \alpha \) to form the generating series shows that
			\[
			\sum_{\alpha=0}^\infty \zeta^{\shuffle,T}\sgnarg{\vec{\phi},\{1\}^\alpha}{\vec{n},\{1\}^\alpha} \, u^\alpha \cdot \exp(-T u)
			\]
			is independent of \( T \), so by equating the \( T \) and the \( S \) regularisation we obtain the claim.
		\end{proof}
	\end{Lem}
		
		The following is first proven in  \cite{ikz06} for classical MZV's.  It is convenient however to recall the details for application to later lemmas; moreover we need a version which holds also for alternating MZV's \cite[Theorem 13.3.9]{zhaoBook}.  
	
	\begin{Def}[Linear map \( \rho \)]\label{def:rho}
		Define an \( \mathbb{R} \) linear map \( \rho \colon \mathbb{R}[T] \to \mathbb{R}[T] \)  by
	\[
		\rho(e^{Tu}) = \exp\Big( \sum_{n=2}^\infty \frac{(-1)^n}{n} \zeta(n) u^n \Big) e^{Tu} \,, \abs{u} < 1 \,.
	\]
	So \( \rho(1) = 1, \rho(T) = T, \rho(T^2) = T^2 + \zeta(2) \) and \( \rho(T^3) = T^3 + 3 \zeta(2) T - 2 \zeta(3) \) are the initial few values.
	\end{Def}

	Then the map \( \rho \) gives us the translation between shuffle and stuffle regularisation, as follows.
	
	\begin{Lem}[Theorem 1, \cite{ikz06}, generalised in Theorem 13.3.9, \cite{zhaoBook}]\label{lem:regshtost}
		The shuffle regularisation with parameter \( \zeta^{\shuffle,T}(1) = T \) and the stuffle regularisation with the same parameter \( \zeta^{\ast,T}(1) = T \) are related as follows.
		\[
			\zeta^{\shuffle,T}(\vec{k}) = \rho\big( \zeta^{\ast,T}(\vec{k}) \big) \,.
		\]
	\end{Lem}

	At this point it is instructive to notice that
	\[
		\exp\Big(\sum_{n=2}^\infty \frac{(-1)^n}{n} \zeta(n) u^n\Big) = \Big( 1 + \sum_{n=1}^\infty \zeta^{\ast,0}(\{1\}^n) u^n \Big)^{-1} \,.
	\]
	This can be seen via Corollary 2 in \cite{ikz06}, or rather via Corollary 1 upon applying the regularisation-evaluation map \( Z \circ \reg^\ast_{0} \).  Directly, one sees that the regularisation has parameter \( \zeta^{\ast,0}(1) = 0 \) since the \( \zeta(1) \) term on the left hand side is not present, and so has been regularised to 0.  More generally, applying \( Z \circ \reg^\ast_{T} \) one has:
	\[
		\exp\Big({-}Tu + \sum_{n=2}^\infty \frac{(-1)^n}{n} \zeta(n) u^n\Big) = \Big( 1 + \sum_{n=1}^\infty \zeta^{\ast,T}(\{1\}^n) u^n \Big)^{-1} \,.
	\]

	\begin{Lem}\label{lem:sttosh111}
		The stuffle regularisation with parameter \( \zeta^{\ast,T}(1) = T \) may be expressed via the shuffle regularisation with parameter \( \zeta^{\shuffle,0}(1) = 0 \) and the `periodic' MZV \( \zeta^{\ast,T}(\{1\}^i) \), as follows.
		\[
			\zeta^{\ast,T}(\vec{k}, \{1\}^\alpha) = \sum_{i=0}^\alpha \zeta^{\shuffle,0}(\vec{k},\{1\}^{\alpha-i}) \zeta^{\ast,T}(\{1\}^{i})
		\]
		
		\begin{proof}
			We apply \autoref{lem:regastTtoastS} in the case \( S = 0 \) to write
			\[
				\zeta^{\ast,T}(\vec{k}, \{1\}^\alpha) = \sum_{i=0}^\alpha \zeta^{\ast,0}(\vec{k}, \{1\}^{\alpha-i}) \frac{T^i}{i!} \,.
			\]
			Note now \( \zeta^{\ast,0}(\vec{k}, \{1\}^{\alpha-i}) \) is a combination of convergent MZV's (after being regularised with parameter \( \zeta^{\ast,0}(1) = 0\)).  Since \( \rho \) is \( \mathbb{R} \)-linear, application of \( \rho \) to convert to the shuffle product only applies to the \( \frac{T^i}{i!} \) part of the summand.  That is to say, we have
			\[
				\zeta^{\shuffle,T}(\vec{k}, \{1\}^\alpha) = \rho(\zeta^{\ast,T}(\vec{k}, \{1\}^\alpha) = \sum_{i=0}^\alpha \zeta^{\ast,0}(\vec{k}, \{1\}^{\alpha-i}) \rho\Big( \frac{T^i}{i!} \Big)
			\]
			Multiply both sides by \( u^\alpha \), and sum on \( \alpha \) to form the generating series
			\[
				\sum_{\alpha=0}^\infty \zeta^{\shuffle,T}(\vec{k}, \{1\}^\alpha) u^\alpha = \sum_{\alpha=0}^\infty \sum_{i=0}^\alpha \zeta^{\ast,0}(\vec{k}, \{1\}^{\alpha-i}) \rho\Big( \frac{T^i}{i!} \Big) u^\alpha \,.
			\]
			Interchange the summation order, then set \( \alpha \to \alpha + i \)
			\begin{align*}
			 = \sum_{i=0}^\infty \sum_{\alpha=i}^\infty \zeta^{\ast,0}(\vec{k}, \{1\}^{\alpha-i}) \rho\Big( \frac{T^i}{i!} \Big) u^\alpha 
 			 & = \sum_{i=0}^\infty \sum_{\alpha=0}^\infty \zeta^{\ast,0}(\vec{k}, \{1\}^{\alpha}) \rho\Big( \frac{T^i}{i!} \Big) u^{i+\alpha}\\
 			 & = \sum_{\alpha=0}^\infty \zeta^{\ast,0}(\vec{k}, \{1\}^{\alpha}) u^\alpha \cdot \sum_{i=0}^\infty \rho\Big( \frac{T^i}{i!} \Big) u^{i} \,.
			\end{align*}
			One sees that the sum involving \( \rho \) is simply \( \rho(e^{Tu}) \), which may be replaced by the expression in \autoref{def:rho}, to give
			\begin{align*}
 			 & = \sum_{\alpha=0}^\infty \zeta^{\ast,0}(\vec{k}, \{1\}^{\alpha}) u^\alpha \cdot \exp\Big(\sum_{n=2}^\infty \frac{(-1)^n}{n} \zeta(n) u^n \Big) e^{Tu}
			\end{align*}
			We therefore have
			\[
				\sum_{\alpha=0}^\infty \zeta^{\shuffle,T}(\vec{k},\{1\}^\alpha) u^\alpha \cdot e^{-Tu} \cdot \exp\Big(Tu-\sum_{n=2}^\infty \frac{(-1)^n}{n} \zeta(n) u^n\Big) = \sum_{\alpha=0}^\infty \zeta^{\ast,0}(\vec{k},\{1\}^\alpha) u^\alpha \cdot e^{Tu} \,.
			\]
			By \autoref{lem:regshTtoshS} and \autoref{lem:regastTtoastS}, respectively
			\begin{align*}
				\sum_{\alpha=0}^\infty \zeta^{\shuffle,T}(\vec{k},\{1\}^\alpha) u^\alpha \cdot e^{-Tu} &= \sum_{\alpha=0}^\infty \zeta^{\shuffle,0}(\vec{k},\{1\}^\alpha) \\
				\sum_{\alpha=0}^\infty \zeta^{\ast,0}(\vec{k},\{1\}^\alpha) u^\alpha \cdot e^{Tu} &= \sum_{\alpha=0}^\infty \zeta^{\ast,T}(\vec{k},\{1\}^\alpha) u^\alpha \,.
			\end{align*}
			Finally, it follows from the observation above that
			\[
				\exp\Big(Tu-\sum_{n=2}^\infty \frac{(-1)^n}{n} \zeta(n) u^n\Big) = 1 + \sum_{n=1}^\infty \zeta^{\ast,T}(\{1\}^n) u^n \,.
			\]
			Making these substitutions, and extracting the coefficient of \( u^\alpha \) establishes the claim.
		\end{proof}
	\end{Lem}

	\subsection{Relating stuffle and shuffle regularised MtV's}  Finally, we give a concrete and explicit relationship between the stuffle and shuffle regularised MtV's for an arbitrary choice of parameters.
	
	\begin{Prop}\label{prop:tstVtotshandz111}
		Let \( \vec{k} = (k_1,\ldots,k_d) \), such that \( k_d \neq 1 \).  Then the stuffle regularisation of MtV's at parameter \( t^{\ast,V} \) and the shuffle regularisation \( t^{\shuffle,0} \) induced by the representation of MtV's as alternating MZV's, with \( \zeta^{\shuffle,0}(1) = 0 \), are related as follows.
		\begin{equation}
			t^{\ast,V}(\vec{k}, \{1\}^\alpha) = \sum_{i=0}^\alpha t^{\shuffle,0}(\vec{k},\{1\}^{\alpha-i}) \cdot \frac{1}{2^{i}} \zeta^{\ast,2V-\log(2)}(\{1\}^{i}) \,.
		\end{equation}
		
		\begin{proof}
			This is clearly true when \( \alpha = 0 \), and no regularisation is necessary, so we assume \( \alpha > 0 \).  Now apply the expression for \( t \) in terms of alternating MZV's in \eqref{eqn:tasz}, and \autoref{cor:tregtozreg} to write
			\[
				t^{\ast,V}(\vec{k},\{1\}^\alpha) = \frac{1}{2^{d+\alpha}} \sum_{\substack{\vec{\epsilon} = (\epsilon_1,\ldots,\epsilon_d), \\ \epsilon_k, \delta_1,\ldots,\delta_\alpha \in \{\pm1\}}} \epsilon_1\cdots\epsilon_d \cdot \delta_1\ldots\delta_\alpha \zeta^{\ast,2V-\log(2)}\sgnarg{\vec{\epsilon},\delta_1,\ldots,\delta_\alpha}{\vec{k},1,\ldots,1}
			\]
			For notation simplicity, we shall always write \( \vec{\epsilon} = (\epsilon_1,\ldots,\epsilon_d) \), and drop the explicit reference to \( \in \{ \pm 1 \} \) from the summation; this should be taken as implied for whatever selection of signs we specify in the sum.  Now gather the terms in this sum by the number of trailing \( \delta_i = 1 \) signs.  One has
			\[
				= \frac{1}{2^{d+\alpha}} \sum_{j=0}^\alpha \sum_{\substack{ \vec{\epsilon}, \\ \delta_1,\ldots,\delta_{\alpha-1-j}, \\ \delta_{\alpha-j}=-1}} \epsilon_1\cdots\epsilon_d \cdot \delta_1\ldots\delta_{\alpha-j} \zeta^{\ast,2V-\log(2)}\sgnarg{\vec{\epsilon},\delta_1,\ldots,\delta_{\alpha-j},\{1\}^j}{\vec{k},1,\ldots,1,\{1\}^j}
			\]
			Application of \autoref{lem:sttosh111} allows us to convert the \( \zeta^{\ast,T} \) regularisation to \( \zeta^{\shuffle,0} \) corrected by \( \zeta^{\ast,T}(\{1\}^n) \), which gives
			\[
				= \frac{1}{2^{d+\alpha}} \sum_{j=0}^\alpha \sum_{\substack{ \vec{\epsilon}, \\ \delta_1,\ldots,\delta_{\alpha-1-j}, \\ \delta_{\alpha-j}=-1}} \epsilon_1\cdots\epsilon_d \cdot \delta_1\ldots\delta_{\alpha-j} \sum_{i=0}^j \zeta^{\shuffle,0}\sgnarg{\vec{\epsilon},\delta_1,\ldots,\delta_{\alpha-j},\{1\}^{j-i}}{\vec{k},1,\ldots,1,\{1\}^{j-i}} \zeta^{\ast,2V-\log(2)}(\{1\}^{i})
			\]
			Moving the sum over \( i \) outside the sum over signs (of which it is independent), and then interchanging the \( j \) and \( i \) summation order gives
			\begin{align*}
			& = \frac{1}{2^{d+\alpha}} \sum_{i=0}^\alpha \sum_{j=i}^\alpha \sum_{\substack{ \vec{\epsilon}, \\ \delta_1,\ldots,\delta_{\alpha-1-j}, \\ \delta_{\alpha-j}=-1}} \!\!\!  \epsilon_1\cdots\epsilon_d \cdot \delta_1\ldots\delta_{\alpha-j} \zeta^{\shuffle,0}\sgnarg{\vec{\epsilon},\delta_1,\ldots,\delta_{\alpha-j},\{1\}^{j-i}}{\vec{k},1,\ldots,1,\{1\}^{j-i}} \zeta^{\ast,2V-\log(2)}(\{1\}^{i}) \\
			& = \frac{1}{2^{d+\alpha}} \sum_{i=0}^\alpha \sum_{j=0}^{\alpha-i} \sum_{\substack{ \vec{\epsilon}, \\ \delta_1,\ldots,\delta_{\alpha-i-1-j}, \\ \delta_{\alpha-i-j}=-1}} \!\!\! \epsilon_1\cdots\epsilon_d \cdot \delta_1\ldots\delta_{\alpha-i-j} \zeta^{\shuffle,0}\sgnarg{\vec{\epsilon},\delta_1,\ldots,\delta_{\alpha-i-j},\{1\}^j}{\vec{k},1,\ldots,1,\{1\}^j} \zeta^{\ast,2V-\log(2)}(\{1\}^{i})
			\end{align*}
			One now recognises that the sum over \( j \) and the sum over signs with \( \delta_{\alpha-i-j} = -1 \) is just the expression for the sum over all signs, gathered by the number of trailing 1's.  So we can rewrite this to be
			\[
			= \frac{1}{2^{d+\alpha}} \sum_{i=0}^\alpha \sum_{\substack{ \vec{\epsilon}, \\ \delta_1,\ldots,\delta_{\alpha-i}}} \epsilon_1\cdots\epsilon_d \cdot \delta_1\ldots\delta_{\alpha-i} \zeta^{\shuffle,0}\sgnarg{\vec{\epsilon},\delta_1,\ldots,\delta_{\alpha-i}}{\vec{k},1,\ldots,1} \zeta^{\ast,2V-\log(2)}(\{1\}^{i})
			\]
			Lastly, we recognise the sum over signs to be \( 2^{d+\alpha-i} t^{\shuffle,0}(\vec{k},\{1\}^{\alpha-i}) \), so after making this replacement, and cancelling the powers of 2, we obtain the claim in the proposition.
		\end{proof}
	\end{Prop}

	\section{\texorpdfstring{Evaluation of the stuffle-regualrised \( t^{\ast,V}(\{2\}^a, 1, \{2\}^b) \)}{
	Evaluation of the stuffle-regularised t\textasciicircum{}\{*,V\}(\{2\}\textasciicircum{}a, 1, \{2\}\textasciicircum{}b)
}}
	\label{sec:t2212ev}
	
	In this section we prove the following evaluation for the stuffle-regularised \( t^{\ast,V}(\{2\}^a, 1, \{2\}^b) \), with \( t^{\ast,V}(1) = V \).  Namely
	\begin{equation}
	\label{eqn:t2212ev}
	\begin{aligned}
	t^{\ast,V}(\{2\}^a,1,\{2\}^b) = {} & -\sum_{r=1}^{a+b} (-1)^r 2^{-2r} \bigg[ \binom{2r}{2a} + \frac{2^{2r}}{2^{2r} - 1} \binom{2r}{2b} \bigg] \zeta(\overline{2r+1}) t(\{2\}^{a + b - r}) \\
	& {} + \delta_{a=0} \log(2) t(\{2\}^b) + \delta_{b=0} (V - \log(2)) t(\{2\}^a) \,,
	\end{aligned}
	\end{equation}
	where \( \delta_{\bullet} \) is the Kronecker delta symbol, equal to 1 if the condition \( \bullet \) holds, and 0 otherwise.  One can write this very explicitly, if desired, as a polynomial in single zeta values, and powers of \( \pi \), using the following evaluation from \cite{hoffman19}:
	\begin{align}
	t(\{2\}^a) &= \frac{\pi^{2a}}{2^{2a} (2a)!} \,, \label{eqn:t222} 
	\end{align}
	and the evaluation \( \zeta(\overline{2r+1}) = -(1 - 2^{-2r}) \zeta(2r+1) \), for \( r > 0 \).  (Note \( \zeta(\overline{1}) = \log(2) \), while \( \zeta(1) \) is divergent and must be regularised to make sense.) \medskip
	
	In order to prove this identity, we first convert it to a generating series identity.  For this purpose introduce the following functions.
	\begin{Def}[Functions \( A(z) \), \( B(z) \)]
		For \( \abs{z} < 1 \), define the \( A(z) \) and \( B(z) \) via the following power series
		\begin{align*}
			A(z) &\coloneqq \sum_{r=1}^\infty \zeta(2r+1) z^{2r} \,, \\
			B(z) &\coloneqq \sum_{r=1}^\infty (1 - 2^{-2r}) \zeta(2r+1) z^{2r} = \sum_{r=1}^\infty -\zeta(\overline{2r+1}) z^{2r} \,.
		\end{align*}
	\end{Def} 

	\begin{Rem}
		The functions \( A(z) \) and \( B(z) \) are the same as defined in Zagier's evaluation of \( \zeta(\{2\}^a,3,\{2\}^b) \) in  \cite{zagier2232}, and Murakami's evaluation of \( t(\{2\}^a, 3, \{2\}^b) \) in \cite{murakami21}.  It is noted in the proof of Proposition 2 in \cite{zagier2232} that they can be expressed via the digamma function \( \psi(x) = \frac{\mathrm{d}}{\mathrm{d}x} \log \Gamma(x) = \frac{\Gamma'(x)}{\Gamma(x)} \), as follows:
		\begin{align*}
			A(z) = \psi(1) - \frac{1}{2} (\psi(1+z) + \psi(1-z)) \,, \quad B(z) = A(z) - A(\tfrac{z}{2}) \,.
		\end{align*}
		In this form the functions \( A(z) \) and \( B(z) \)  analytically continue to the whole complex plane, with simple poles at \( z \in \Z \setminus \{0\} \).
	\end{Rem} \medskip
	
	It is a routine manner to sum (a tweaked version of) the generating series of the right-hand side to see the claim is equivalent to the following Theorem.  For details of such summation techniques, we refer to the corresponding evaluations in both \cite[proof of Proposition 2]{zagier2232} and \cite[proof of Proposition 13]{murakami21}.
	\begin{Thm}\label{thm:t2212ev}
		The following generating series evaluation holds for the stuffle-regularised \( t^{\ast,V} \), with \( t^{\ast,V}(1) = V \),
		\begin{align*}
		 \sum_{a,b\geq0} (-1)^{a+b} t^{\ast,V}(\{2\}^a,1,\{2\}^b) \cdot (2x)^{2a} (2y)^{2b} = {}
		& \frac{1}{2} \cos(\pi x) (A(x-y) + A(x+y) + 2 (V - \log(2))) \\
		& + \frac{1}{2} \cos(\pi y) (B(x-y) + B(x+y) + 2 \log(2)) \,,
		\end{align*}
		where
		\begin{align*}
		A(z) &= \psi(1) - \frac{1}{2} (\psi(1+z) + \psi(1-z)) = \sum_{r=1}^\infty \zeta(2r+1) z^{2r}  \,, \\
		B(z) &= A(z) - A(\tfrac{z}{2}) = \sum_{r=1}^\infty (1 - 2^{-2r}) \zeta(2r+1) z^{2r} = \sum_{r=1}^\infty -\zeta(\overline{2r+1}) z^{2r} \,.
		\end{align*}
	\end{Thm}
		
	\subsection{Proof of \autoref{thm:t2212ev}}
	Firstly, recall that the \( {}_pF_{p-1} \) hypergeometric function is defined as
	\[
		\pFq{p}{p-1}{a_1,\ldots,a_p}{b_1,\ldots,b_{p-1}}{x} \coloneqq \sum_{m=0}^\infty \frac{\poch{a_1}{m} \cdots \poch{a_p}{m} }{\poch{b_1}{m} \cdots \poch{b_{p-1}}{m} } \frac{x^m}{m!} \,,
	\]
	where \( \poch{a}{m} = a(a+1) \cdots (a+m-1) \) is the ascending Pochhammer symbol.  Asymptotic and transformation properties of the \( {}_4F_3 \) and \( {}_3F_2 \) will play a key role in the proof of our generating series evaluation.

	In order to prove this theorem, we utilize a multiple \( t \)-polylogarithm type function, defined as follows.
	
	\begin{Def}[Multiple \( t \)-polylogarithm]\label{def:ti}
		For a choice of indices \( s_1,\ldots,s_d \in \Z_{\geq0} \), the \( \Ti \) functions is defined by
		\[
			\Ti_{s_1,\ldots,s_d}(x_1,\ldots,x_d) \coloneqq \sum_{0 < n_1 < \cdots < n_d} \frac{x_1^{2n_1 - 1} \cdots x_d^{2n_d-1}}{(2n_1-1)^{s_1} \cdots (2n_d-1)^{s_d}} \,,
		\]
		which converges when \( \abs{x_1\cdots x_i}  < 1 \), for \( i = 1, \ldots, r \).
	\end{Def}

	\begin{Rem}
		Closely related functions, at least for depth \( r = 1 \) and weight 2, are already studied in Lewin's book \cite{LewinBook} under the names `the inverse tangent integral' (Chapter 2 of \cite{LewinBook})
		\[
			\LewinTi_2(x) = \Im \Li_2(i x) = \sum_{n=1}^\infty \frac{(-1)^{n+1} x^{2n-1}}{(2n-1)^2}
		\] and `Legendre's chi-function' (Section 1.8 of \cite{LewinBook})
		\[
			\chi_2(x) = \sum_{n=1}^\infty \frac{x^{2n-1}}{(2n-1)^2} = \frac{1}{2} \Li_2(x) - \frac{1}{2} \Li_2(-x) \,.
		\]
		Lewin actually uses the notation \( \Ti \) for his function, but I write \( \LewinTi \) here to avoid confusion with the function introduced above.  Moreover, the notation \( \chi_2 \) is Lewin's choice, supplanting the too general notation \( \phi \) originally used by Legendre.
	\end{Rem}

	The function \( \Ti \) from \autoref{def:ti} is related to the classical multiple polylogarithm functions \( \Li_{s_1',\ldots,s_d'} \) in an analogous way to how the multiple $t$-value \( t(s_1,\ldots,s_d) \) is related to the classical multiple zeta values \( \zeta(s_1',\ldots,s_d') \) in \eqref{eqn:tasz}.  An explicit formula can be given, exactly as for \( t \) values, by inserting a factor \( \frac{1}{2} (1 - (-1)^{n_i}) \) into the numerator for \( i = 1, \ldots, r \), which allows one to extend the range of summation of the denominators and exponents from just odd integers, to all positive integers.  Namely
	\begin{align*}
		\Ti_{s_1,\ldots,s_d}(x_1,\ldots,x_d) & = \sum_{0 < n_1 < \cdots < n_d} \frac{x_1^{2n_1 - 1} \cdots x_d^{2n_d-1}}{(2n_1-1)^{s_1} \cdots (2n_d-1)^{s_d}} \\ 
		& = \sum_{0 < n_1 < \cdots < n_d} \frac{(1 - (-1)^{n_1}) \cdots ((1 - (-1)^{n_d}))}{2^d} \frac{x_1^{n_1} \cdots x_d^{n_d}}{n_1^{s_1} \cdots n_d^{s_d}} \\
		&= \frac{1}{2^d} \sum_{\epsilon_1, \ldots, \epsilon_d \in \{\pm 1\}} \epsilon_1 \cdots \epsilon_d \Li_{s_1,\ldots,s_d}(\epsilon_1 x_1, \ldots, \epsilon_d x_d) \,.
	\end{align*}
	We note also that when \( s_d > 1 \), the special value \( \Ti_{s_1,\ldots,s_d}(1,\ldots,1) = t(s_1,\ldots,s_d) \) is exactly the multiple \( t \) value of the given indices, as in this case the MtV is convergent.  We find, however, that
	\[
		\Ti_1(z) = \sum_{n_1=1}^\infty \frac{z^{2n_1-1}}{2n_1-1} = \tanh^{-1}(z) \,,
	\]
	so in particular \( \lim_{z\to1^-} \Ti_{1}(z) = \infty \).   \medskip
	
	Now, let us turn out attention to 
	\begin{align*}
	& \Ti_{\{2\}^a,1,\{2\}^b}(\{1\}^a, z, \{1\}^b) \\ 
	& = \sum_{\substack{0 < n_1 < \cdots < n_a < r \\ < m_1 < \cdots < m_b}} \, \frac{1}{(2n_1 - 1)^2 \cdots (2n_a-1)^2} \cdot \frac{z^{2r-1}}{2r-1} \cdot \frac{1}{(2m_1 - 1)^2 \cdots (2m_b-1)^2} \,.
	\end{align*}
	We will establish that a certain limit involving a similar generating series of these \( \Ti_{\{2\}^a,1,\{2\}^b} \)-polylogs can be used to give the desired generating series of \( t^{\ast,V=0} \) values.	We find
	\begin{equation}\label{eqn:Ti2212:gs:as4F3}
	\begin{aligned}
	& \sum_{a,b\geq0} (-1)^{a+b} \Ti_{\{2\}^a,1,\{2\}^b}(\{1\}^a,z,\{1\}^b) \cdot (2x)^{2a} (2y)^{2b} \\
	&= \sum_{r=1}^\infty \prod_{\ell<r} \Big( 1 - \frac{4x^2}{(2\ell-1)^2} \Big) \cdot \frac{z^{2r-1}}{2r-1} \cdot \prod_{k>r} \Big( 1 - \frac{4y^2}{(2k-1)^2} \Big) \\
	&= \cos({\pi y}) \sum_{r=1}^\infty \prod_{\ell<r} \Big( 1 - \frac{4x^2}{(2\ell-1)^2} \Big) \cdot \frac{z^{2r-1}}{2r - 1} \cdot \prod_{k \leq r}\Big( 1 - \frac{4y^2}{(2k-1)^2} \Big)^{-1} \\ 
	&= \frac{z \cos(\pi y)}{1-4y^2}  \cdot \pFq{4}{3}{1,\tfrac{3}{2},\tfrac{1}{2}-x,\tfrac{1}{2}+x}{\tfrac{1}{2},\tfrac{3}{2}-y,\tfrac{3}{2}+y}{z^2}
	\end{aligned}
	\end{equation}
	One checks directly that the summand above is expressible in the required form for the \( {}_4F_3 \) hypergeometric function.
	
	Now the divergent part (as \( z \to 1^- \)) of this generating series arises from 
	\begin{align*}
	\sum_{a\geq0} (-1)^a \Ti_{\{2\}^a,1}(\{1\}^a, z) \cdot (2x)^{2a} \,.
	\end{align*}
	We notice here that by stuffle-regularising,
	\begin{equation}\label{eqn:Ti221:gs:regularised}
	\begin{aligned} \Ti_{\{2\}^a,1}	(\{1\}^a, z) = {} 
		& t(\{2\}^a) \Ti_1(z) - \sum_{i=0}^{a-1} \Ti_{\{2\}^i,1,\{2\}^{a-i}}(\{1\}^i, z, \{1\}^{a-i}) \\[-1ex]
		& - \sum_{i=0}^{a-1} \Ti_{\{2\}^i,3,\{2\}^{a-1-i}}(\{1\}^i, z, \{1\}^{a-1-i}) \,.
	\end{aligned}
	\end{equation}
	So one can write that
	\begin{align*}
	& \sum_{a\geq0} (-1)^a \Ti_{\{2\}^a,1}(\{1\}^a, z) \cdot (2x)^{2a} = \tanh^{-1}(z) \cos(\pi x) + f(x,z)
	\end{align*}
	where
	\[
	f(x,z) = - \sum_{a=0}^\infty (-1)^a \begin{aligned}[t]
		& \Big( \sum_{i=0}^{a-1} \Ti_{\{2\}^i,1,\{2\}^{a-i}}(\{1\}^i,z,\{1\}^{a-i}) \\[-1ex]
		& \hspace{5em} + \sum_{i=0}^{a-1} \Ti_{\{2\}^i,3,\{2\}^{a-1-i}}(\{1\}^i,z,\{1\}^{a-1-i}) \Big) \cdot (2x)^{2a} \,. \end{aligned} 
	\]
	(Note that \( \cos(\pi x) \) arises as the generating series of \( t(\{2\}^a) \), after incorporating the normalisation factors \( (-1)^a \) and \( (2x)^{2a} \) above.  Namely
	\[
		\sum_{a=0}^\infty (-1)^a t(\{2\}^a) \cdot (2x)^{2a} = \sum_{a=0}^\infty (-1)^a \frac{\pi^{2a}}{2^{2a} (2a)!} \cdot (2x)^{2a} = \cos(\pi x) \,,
	\]
	wherein we have substituted the evaluation of \( t(\{2\}^a) \) from \cite{hoffman19}, given in \eqref{eqn:t222} above.)
	We see that at \( z = 1 \), 
	\begin{align*}
	 & \sum_{i=0}^{a-1} \Ti_{\{2\}^i,1,\{2\}^{a-i}}(\{1\}^{a+1}) + \sum_{i=0}^{a-1} \Ti_{\{2\}^i,3,\{2\}^{a-1-i}}(\{1\}^a)  \\
	& = \sum_{i=0}^{a-1} t(\{2\}^i,1,\{2\}^{a-i}) + \sum_{i=1}^{a-0} t(\{2\}^i,3,\{2\}^{a-1-i}) \\
	& = t^{\ast,V=0}(1) t(\{2\}^a) - t^{\ast,V=0}(\{2\}^a,1) \\
	& = - t^{\ast,V=0}(\{2\}^a,1) \,.
	\end{align*}
	So that \( f(x,1) \) (or at least the limit \( \lim_{z\to1^-} \) thereof) satisfies
	\[
		f(x,1) = \sum_{a=0}^\infty (-1)^a t^{\ast,V=0}(\{2\}^a, 1) \cdot (2x)^{2a} \,.
	\]
	Now subtract \eqref{eqn:Ti221:gs:regularised} from \eqref{eqn:Ti2212:gs:as4F3}, and take the limit \( \lim_{z\to1^-} \).  From this we see that the generating series of stuffle-regularised (at \( V = 0 \)) M\(t\)V's is obtained by computation of the following limit
	\begin{equation}\label{eqn:4f3limit}
	\begin{aligned}
	& \sum_{a,b\geq0} (-1)^{a+b} t^{\ast,V=0}(\{2\}^a,1,\{2\}^b) \cdot (2x)^{2a} (2y)^{2b} \\
	& = \lim_{z\to1^-} \frac{z \cos(\pi y)}{1-4y^2}  \cdot \pFq{4}{3}{1,\tfrac{3}{2},\tfrac{1}{2}-x,\tfrac{1}{2}+x}{\tfrac{1}{2},\tfrac{3}{2}-y,\tfrac{3}{2}+y}{z^2} - \tanh^{-1}(z) \cos(\pi x) \,.
	\end{aligned}
	\end{equation}
	
	We now apply some transformation properties of \( {}_4F_3 \) in order to reduce this to a combination of \( {}_3F_2 \) functions, whose asymptotic behaviour is established by \cite{evans84}.  First  make use of the contiguous function relation
	\[
	b \cdot \pFq{4}{3}{a, b+1, c, d}{p,q,r}{z} - a \cdot \pFq{4}{3}{a+1, b, c,d}{p,q,r}{z} + (a-b) \cdot \pFq{4}{3}{a,b,c,d}{p,q,r}{z} = 0
	\]
	in the case \( (a,b,c,d) = (1,\tfrac{1}{2}, \tfrac{1}{2}-x,\tfrac{1}{2}+x), (p,q,r) = (\tfrac{1}{2}, \tfrac{3}{2}-y,\tfrac{3}{2}+y) \), to obtain the following reduction of our \( {}_4F_3 \) to a combination of \( {}_3F_2 \)'s.  We find
	\begin{equation}\label{eqn:4f3ctg}
	\pFq{4}{3}{1,\tfrac{3}{2},\tfrac{1}{2}-x,\tfrac{1}{2}+x}{\tfrac{1}{2},\tfrac{3}{2}-y,\tfrac{3}{2}+y}{z^2} = 2\cdot\pFq{3}{2}{2,\tfrac{1}{2}-x,\tfrac{1}{2}+x}{\tfrac{3}{2}-y,\tfrac{3}{2}+y}{z^2} 
	- \pFq{3}{2}{1, \tfrac{1}{2}-x,\tfrac{1}{2}+x}{\tfrac{3}{2}-y,\tfrac{3}{2}+y}{z^2} 
	\end{equation}
	The second term is convergent at \( z = 1 \), and can be evaluated via Whipple's theorem (see Section 3.4 in \cite{bailey64}) to give (after some simplification with the reflection formula of the \( \Gamma \)-function) that
	\[
	\pFq{3}{2}{1, \tfrac{1}{2}-x,\tfrac{1}{2}+x}{\tfrac{3}{2}-y,\tfrac{3}{2}+y}{1}  = \frac{(1-2y)(1+2y)}{2(x-y)(x+y)} \sec(\pi y) \sin\big( \tfrac{\pi}{2}(x-y) \big) \sin\big( \tfrac{\pi}{2}(x+y) \big) \,.
	\]
	To deal with the first term, we need to recall the Stanton-Evans-Ramanujan asymptotic for 0-balanced \( {}_3F_2 \) hypergeometric functions.
	\begin{Thm}[Evans-Stanton 1984 \cite{evans84}, Ramanujan]
		If \( a + b + c = d + e \), and \( \Re(c) > 0 \), then as \( u \to 1^- \),
		\[
		\frac{\Gamma(a)\Gamma(b)\Gamma(c)}{\Gamma(d)\Gamma(e)} \cdot \pFq{3}{2}{a,b,c}{d,e}{u} = -\log(1-u) + L + O((1-u)\log(1-u)) \,,
		\]
		where
		\[
		L = -2\gamma - \frac{\Gamma'(a)}{\Gamma(a)} - \frac{\Gamma'(b)}{\Gamma(b)} + \sum_{k=1}^\infty \frac{\poch{d-c}{k} \poch{e-c}{k}}{\poch{a}{k} \poch{b}{k} k} \,.
		\]
		Here \( \gamma \approx 0.577\ldots \) is the Euler-Mascheroni constant, and \( \poch{x}{k} = x(x+1) \cdots (x + k-1) \) is the ascending Pochhammer symbol.
	\end{Thm}
	
	If we apply this asymptotic (with \( c = 2 \)) to the first term, and recall \( \psi(x) = \frac{\mathrm{d}}{\mathrm{d}x} \log\Gamma(x) = \frac{\Gamma'(x)}{\Gamma(x)}\), we obtain the asymptotic formula
	\begin{align*}
	-\frac{4}{1-4y^2} \frac{\cos(\pi y)}{\cos(\pi x)} \pFq{3}{2}{2,\tfrac{1}{2}-x,\tfrac{1}{2}+x}{\tfrac{3}{2}-y,\tfrac{3}{2}+y}{z^2} {} = {} & \log(1-u) + 2\gamma + \psi(\tfrac{1}{2} - x) + \psi(\tfrac{1}{2}+x) \\ 
	& \hspace{-2em} - \sum_{k=1}^\infty \frac{\poch{-\tfrac{1}{2}-y}{k}\poch{-\tfrac{1}{2}+y}{k}}{k \poch{\tfrac{1}{2}-x}{k}\poch{\tfrac{1}{2}+x}{k}} + O((1-z^2)\log(1-z^2)) \,.
	\end{align*}
	We also note
	\[
	4A(2x) - 2A(x) = 4 \log(2) + 2 \gamma + \psi(\tfrac{1}{2} -x) + \psi(\tfrac{1}{2} + x) \,,
	\]
	so that the digamma combination above can be rewritten via the function \( A \) defined earlier.  Applying these results to \eqref{eqn:4f3limit}, we find
	\begin{align*}
	\text{RHS (\ref{eqn:4f3limit})} = {}  & -B(x) + \frac{\cos(\pi x) - \cos(\pi y)}{4(x^2 - y^2)} - (1 - 2 \cos(\pi x))\log(2) \\ & {}  + \frac{1}{2} \cos(\pi x) \Bigg( 4A(2x) - 4(x)
	+ \sum_{k=1}^\infty \frac{\poch{-\tfrac{1}{2}-y}{k}\poch{-\tfrac{1}{2}+y}{k}}{k \poch{\tfrac{1}{2}-x}{k}\poch{\tfrac{1}{2}+x}{k}} \Bigg)
	\end{align*}
	We note next that
	\[
	\sum_{k=1}^\infty \frac{\poch{-\tfrac{1}{2}-y}{k}\poch{-\tfrac{1}{2}+y}{k}}{k \poch{\tfrac{1}{2}-x}{k}\poch{\tfrac{1}{2}+x}{k}} = \frac{\mathrm{d}}{\mathrm{d}Z}\Big|_{Z=0} \pFq{3}{2}{-\tfrac{1}{2}-y,-\tfrac{1}{2}+y,Z}{\tfrac{1}{2} - x,\tfrac{1}{2} + x}{1} \,.
	\]
	Compare Proposition 1 in \cite{zagier2232} for a similar summation, which we will in fact reduce this to.  Using the contiguous function relation
	\[
	(a-b)p \cdot \pFq{3}{2}{a,b,c}{p,q}{z} - b(a-p) \cdot \pFq{3}{2}{a, 1+b, c}{1+p,q}{z} + a(b-p) \cdot \pFq{3}{2}{1+a,b,c}{1+p,q}{z} = 0
	\]
	in the case \( (a,b,c) = (-\tfrac{1}{2}-y, -\tfrac{1}{2} + y, Z), (p,q) = (\tfrac{1}{2} - x, \tfrac{1}{2} + x) )\), we find (note the sign of \( y \) is different in various places in the coefficient of each \( {}_3F_2 \) on the right hand side) that
	\begin{equation}\label{eqn:3f2:reduced}
	\begin{aligned}
	\pFq{3}{2}{-\tfrac{1}{2}-y,-\tfrac{1}{2}+y,Z}{\tfrac{1}{2} - x,\tfrac{1}{2} + x}{1} = {} & -\frac{(1-x + y)(1-2y)}{2y(1-2x)} \pFq{3}{2}{{-}\big(\tfrac{1}{2}+y\big),\tfrac{1}{2}+y,Z}{1-\big(x-\tfrac{1}{2} \big),1 + \big(x-\tfrac{1}{2}\big)}{1} \\ & + \frac{(1-x-y)(1 + 2y)}{2y(1-2x)}\pFq{3}{2}{-\big({-}\tfrac{1}{2}+y\big),{-}\tfrac{1}{2}+y,Z}{1-\big(x-\tfrac{1}{2} \big),1 + \big(x-\tfrac{1}{2}\big)}{1} \,,
	\end{aligned}
	\end{equation}
	and the same expression upon replacing \( {}_3F_2 \) with \( \frac{\mathrm{d}}{\mathrm{d}Z} \big|_{Z=0} {}_3F_2 \) on both sides.  Both hypergeometric functions derivatives are now of the form
	\begin{align*}
	\frac{\mathrm{d}}{\mathrm{d}Z} \Big|_{Z=0} \pFq{3}{2}{-X,X,Z}{1-Y,1+Y}{1} = {} & \big[ A(X+Y) + A(X-Y) - 2A(Y) \big] \\ 
	& \quad + \frac{\sin(\pi X)}{\sin(\pi Y)} \big[ B(X+Y) - B(X-Y) \big] \,,
	\end{align*}
	the evaluation of which here follows as essentially the punchline to Section 4 of Zagier's evaluation of \( \zeta(\{2\}^a, 3, \{2\}^b) \) in \cite{zagier2232} after combining the results of Sections 2 and 3 therein.  (Namely the equality of \( F(x,y) = \widehat{F}(x,y) \) established in the proof Theorem 1 in \cite{zagier2232}, plus the expressions in Propositions 1 and 2 of \cite{zagier2232}, gives the above evaluation.)
	
	Substituting this evaluation into \( \frac{\mathrm{d}}{\mathrm{d}Z} \big|_{Z=0}\) of \eqref{eqn:3f2:reduced}, and substituting the resulting Poch\-hammer sum evaluation into the \( {}_4F_{3} \) limit produces an elementary expression for the  generating series of \( t^{\ast,V=0}(\{2\}^a,1,\{2\}^b) \) in terms of \( A, B \), sine and cosine.  Some simplification is necessary, using the functional relation of \( \psi \) which implies \( A(x+1) - A(x) = -\frac{1}{2x} - \tfrac{1}{2(1+x)} \), but one readily finds
	\begin{equation}\label{eqn:t2212:gs:reg0}
	\begin{aligned}
	& \sum_{a,b\geq0} (-1)^{a+b} t^{\ast,V=0}(\{2\}^a,1,\{2\}^b) \cdot (2x)^{2a} (2y)^{2b} = \\
	& \frac{1}{2} \cos(\pi x) (A(x-y) + A(x+y) - 2 \log(2)) + \frac{1}{2} \cos(\pi y) (B(x-y) + B(x+y) + 2 \log(2)) \,.
	\end{aligned}
	\end{equation}
	The generating series for the general regularisation is recovered upon noting that
	\[
	t^{\ast,V}(\{2\}^a,1) = V t(\{2\}^a)  + t^{\ast,V=0}(\{2\}^a,1) \,,
	\]
	i.e. the constant term in the regularisation polynomial is the regularisation at parameter \( V = 0 \).  Since 
	\[
	\sum_{a \geq 0} (-1)^a V t(\{2\}^a) \cdot (2x)^{2a} = V \cos(\pi x) \,,
	\]
	as already noted above without the \( V \), this gives the necessary correction term to add to the right hand side of \eqref{eqn:t2212:gs:reg0} to find the generating series for the general regularisation  Doing so gives the equality stated in \autoref{thm:t2212ev}, and so completes the proof. \hfill \qedsymbol
	
	\subsection{\texorpdfstring{Evaluation of shuffle-regularised \( t^{\shuffle,W}(\{2\}^a, 1, \{2\}^b) \)}{
		Evaluation of shuffle-regularised t\textasciicircum{}\{sh,W\}(\{2\}\textasciicircum{}a, 1, \{2\}\textasciicircum{}b)
	}} 

	Although the shuffle regularisation \( t^{\shuffle,0} \), arising from \( \zeta^{\shuffle,0}(1) = 0 \) is most important, we can in fact compute the regularisation for any \( t^{\shuffle,W} \) arising from \( \zeta^{\shuffle,W}(1) = W \) with equal case.  Clearly, if \( b > 0 \)
	\[
		t^{W,\shuffle}(\{2\}^a,1,\{2\}^b) = t(\{2\}^a, 1, \{2\}^b) \,,
	\]
	as no regularisation is necessary.  However when \( b = 0 \) we compute via \eqref{eqn:tasz} -- with the convention that \( \vec{\epsilon} = (\epsilon_1, \ldots, \epsilon_a) \) and the sum over \( \vec{\epsilon}, \delta \) implies all choice of signs in \( \{\pm1\} \) --  that
	\begin{align*}
		t^{W,\shuffle}(\{2\}^a, 1) &= \frac{1}{2^{a+1}} \sum_{\vec{\epsilon}, \delta} \epsilon_1 \cdots \epsilon_a \cdot \delta \zeta^{\shuffle,W}\sgnarg{\vec{\epsilon}, \delta}{\{2\}^a, 1} 
	\end{align*}
	Since the alternating MZV ends with at most a single entry 1 (with sign 1), we know from \autoref{lem:regshtost} that the shuffle and the stuffle regularisation in this case are exactly equal.  This is because the \( \mathbb{R} \)-linear map \( \rho \) from \autoref{def:rho} appearing in \autoref{lem:regshtost} has \( \rho(1) = 1 \) and \( \rho(T) = T \), so leaves a linear regularisation polynomial unchanged.  Hence
	\begin{align*}
		 &= \frac{1}{2^{a+1}} \sum_{\vec{\epsilon}, \delta} \epsilon_1 \cdots \epsilon_a \cdot \delta \zeta^{\ast,W}\sgnarg{\vec{\epsilon}, \delta}{\{2\}^a, 1}  \\
		 &= t^{\ast,\frac{1}{2}(W + \log(2))}(\{2\}^a,1)
	\end{align*}
	Recall from \autoref{cor:tregtozreg}: the stuffle regularisation of \( t^{\ast,V}(1) = V \) corresponds to the stuffle regularisation of \( \zeta^{\ast,U}(1) = U \) where \( U = 2V - \log(2) \), hence the change in regularisation parameter in the last line. \medskip
	
	This amounts to saying the shuffle regularised version \( t^{\shuffle,W} \) of the generating series in \autoref{thm:t2212ev} is obtained simply by changing the regularisation parameter on the RHS to \( \frac{1}{2} (W + \log(2)) \).  Hence we have the following proposition.
	
	\begin{Prop}
		\label{thm:sht2212ev}
			The following generating series evaluation holds for the shuffle-regularised \( t^{\shuffle,W} \), induced by \( \zeta^{\shuffle,W}(1) = W \),
			\begin{align*}
			\sum_{a,b\geq0} (-1)^{a+b} t^{\shuffle,W}(\{2\}^a,1,\{2\}^b) \cdot (2x)^{2a} (2y)^{2b} = {}
			& \frac{1}{2} \cos(\pi x) (A(x-y) + A(x+y) + (W - \log(2))) \\
			& + \frac{1}{2} \cos(\pi y) (B(x-y) + B(x+y) + 2 \log(2)) \,,
			\end{align*}
			where
			\begin{align*}
			A(z) &= \psi(1) - \frac{1}{2} (\psi(1+z) + \psi(1-z)) = \sum_{r=1}^\infty \zeta(2r+1) z^{2r}  \,, \\
			B(z) &= A(z) - A(\tfrac{z}{2}) = \sum_{r=1}^\infty (1 - 2^{-2r}) \zeta(2r+1) z^{2r} = \sum_{r=1}^\infty -\zeta(\overline{2r+1}) z^{2r} \,.
			\end{align*}
	\end{Prop}

	From this follows an explicit evaluation, analogous to \eqref{eqn:t2212ev}, by replacing \( V \) with \( \frac{1}{2}(W + \log(2)) \) therein:
	\begin{equation}
	\label{eqn:sht2212ev}
	\begin{aligned}
	t^{\shuffle,W}(\{2\}^a,1,\{2\}^b) = {} & -\sum_{r=1}^{a+b} (-1)^r 2^{-2r} \bigg[ \binom{2r}{2a} + \frac{2^{2r}}{2^{2r} - 1} \binom{2r}{2b} \bigg] \zeta(\overline{2r+1}) t(\{2\}^{a + b - r}) \\
	& {} + \delta_{a=0} \cdot \log(2) t(\{2\}^b) + \delta_{b=0} \cdot \frac{1}{2} (W - \log(2)) t(\{2\}^a) \,,
	\end{aligned}
	\end{equation}
	where \( \delta_{\bullet} \) is the Kronecker delta symbol, equal to 1 if the condition \( \bullet \) holds, and 0 otherwise.  
	
	\subsection{\texorpdfstring{Evaluation of \( t(1,\{2\}^n) \)}{
			Evaluation of t(1,\{2\}\textasciicircum{}n)
	}}
	\label{sec:t1222}

	In order to answer a question posed in \cite{chavan21}, we turn to the special case of \( t(1,\{2\}^n) \) , for \( n \geq 1 \). Here we extract from \eqref{eqn:t2212ev}, the following evaluation for \( t(1,\{2\}^n) \), where \( n \geq 1 \)
	\begin{equation*}
	t(1,\{2\}^n) = -\sum_{r=1}^{n} (-1)^r 2^{-2r} \bigg[ 1 + \frac{2^{2r}}{2^{2r} - 1} \binom{2r}{2n} \bigg] \zeta(\overline{2r+1}) t(\{2\}^{n - r}) +  \log(2) t(\{2\}^n) \,.
	\end{equation*}
	Since \( \binom{2r}{2n} = 0 \) for \( r < n \), and \( \binom{2r}{2n} = 1 \) for \( r = n \), this can be written as
	\begin{equation*}
	= \log(2) t(\{2\}^n) -\sum_{r=1}^{n-1} (-1)^r 2^{-2r} \zeta(\overline{2r+1}) t(\{2\}^{n - r}) - (-1)^n 2^{-2n} \bigg[ 1 + \frac{2^{2n}}{2^{2n} - 1} \bigg] \zeta(\overline{2n+1}) \,.
	\end{equation*}
	Now the first term can be incorporated as the \( r = 0 \) term of the sum, giving
	\begin{equation*}
	= -\sum_{r=0}^{n-1} (-1)^r 2^{-2r} \zeta(\overline{2r+1}) t(\{2\}^{n - r}) - (-1)^n 2^{-2n} \bigg[ 1 + \frac{2^{2n}}{2^{2n} - 1} \bigg] \zeta(\overline{2n+1})
	\end{equation*}
	Now substitute in the evaluation of \( t(\{2\}^a) = \frac{\pi^{2a}}{2^{2a} (2a)!} \)from \eqref{eqn:t222}, and convert the last term to a classical (non-alternating) MZV, to obtain
	\begin{equation*}
	= \frac{1}{2^{2n}} \Bigg( \sum_{r=0}^{n-1} (-1)^r (-\zeta(\overline{2r+1})) \frac{\pi^{2(n-r)}}{(2(n-r))!} + (-1)^n 2 \big( 1 - 2^{-2n-1} \big)\zeta(2n+1) \Bigg)
	\end{equation*}
	This confirms Conjecture 4.5 on the evaluation of \( t(1,\{2\}^n) \) stated in \cite{chavan21} (be aware, the opposite MZV/MtV convention is used therein).  The authors of \cite{chavan21} also write the Dirichlet eta function  \( \eta(m) = (1 - 2^{1-m}) \zeta(m) \), with \( \eta(1) = \log(2) \), in place of \( -\zeta(\overline{m}) \) used herein.
	
	\section{Motivic framework} \label{sec:mot}
	
	In this section we briefly recall the setup of motivic iterated integrals framework introduced by Brown \cite{brown12,brown17} (extending that of Goncharov \cite{goncharov01,goncharov05}).  We define the motivic (alternating) MZV's and the motivic MtV's; we introduce the necessary combinatorial operations and fundamental properties of these objects which will play a key role from \autoref{sec:mott2212} onwards. \medskip
	
	\subsection{Goncharov's motivic iterated integrals}
	
	In \cite{goncharov05}, Goncharov upgraded the iterated integrals \( I(x_0; x_1, \ldots, x_N; x_{N+1}) \), \( x_i \in \overline{\Q} \) (see \eqref{eqn:itint} above for the definition), to framed mixed Tate motives, in order to define motivic iterated integrals
	\[
		I^\umot(x_0; x_1, \ldots, x_N; x_{N+1})
	\]
	living in a graded (by the weight \( N \)) connected Hopf algebra \( \mathcal{A}_\bullet = \mathcal{A}_\bullet(\overline{\Q}) \).  The Hopf algebra \( \mathcal{A}_\bullet \) is the ring of regular functions on the unipotent part of the motivic Galois group.  In \cite{goncharov05}, they are denoted by \( I^{\mathscr{M}} \), but when incorporated into Brown's motivic framework, they are better denoted by \( I^\umot \) for the unipotent part.
	
	The motivic iterated integrals satisfy relations of a `geometric' origin, arising from change of variables in an iterated integral, the results of Stoke's theorem, or from the linearity of domain and integrand.  The coproduct \( \Delta \) on this Hopf algebra is computed via Theorem 1.2 in \cite{goncharov05} as
	\begin{align*}
	&  \Delta I^{\umot}(x_0; x_1,\ldots,x_N; x_{N+1}) = {} \\
	& \sum_{\substack{0 = i_0 < i_1 < \cdots \\ 
			< i_{k} < i_{k+1} = N+1}} \hspace{-1em} I^{\umot}(x_0; x_{i_1}, \ldots, x_{i_k}; x_{N+1}) \otimes \prod_{p=0}^{k} I^\umot(x_{i_p}; x_{i_{p}+1}, \ldots, x_{i_{p+1}-1}; x_{i_{p+1}}).
	\end{align*}
	
	In this Hopf algebra, the motivic version of \( \zeta^\umot(2) = -I^\umot(0; 1, 0; 1) = 0 \), or more fundamentally, the Lefschetz motive \( \mathbb{L}^\umot \), a motivic version of \( \ii \pi \), vanishes so that \( (\ii \pi)^\umot = \mathbb{L}^\umot = 0\).
	
	\subsection{\texorpdfstring{Brown's \( \mathcal{A}_\bullet \)-comodule of motivic iterated integrals}{
			Brown's A-comodule of motivic iterated integrals}}
		
		The motivic iterated integrals \( I^\mot(x_0; x_1, \ldots, x_N; x_{N+1}) \) in the sense of Brown \cite{brown12,brown17} are elements of the \( \mathcal{A}_\bullet \)-comodule \( \mathcal{H}_\bullet \) of regular functions on the torsor of tensor isomorphisms between Betti and de Rham realisations.
		
		This comodule is endowed with a coaction \( \Delta \colon \mathcal{H}_\bullet \to \mathcal{A}_\bullet \otimes \mathcal{H}_\bullet \) which, as noted in~\cite{brown17}, is given by the same formula as Goncharov's coproduct, transposed to this setting, i.e.
		\begin{equation}\label{eqn:coaction}
		\begin{aligned}[t]
		&  \Delta I^{\mot}(x_0; x_1,\ldots,x_N; x_{N+1}) = {} \\
		& \prod_{p=0}^{k} I^\umot(x_{i_p}; x_{i_{p}+1}, \ldots, x_{i_{p+1}-1}; x_{i_{p+1}}) \otimes \sum_{\substack{0 = i_0 < i_1 < \cdots \\ 
				< i_{k} < i_{k+1} = N+1}} \hspace{-1em} I^{\mot}(x_0; x_{i_1}, \ldots, x_{i_k}; x_{N+1}).
		\end{aligned}
.		\end{equation}
		(We have switched the order of the factors for later convenience.)  We will mainly use the derivation operations \( D_r \) defined as a linearised, weight-graded part of the coaction (see \autoref{sec:dr} below), but it will be useful to keep in mind from where these operations originate, particularly when considering how they act on primitive elements.
		
		In Brown's setting \( \zeta^\mot(2) = -I^\mot(0; 1, 0; 1) \neq 0 \), and therefore much more information about motivic iterated integrals is retained.  In particular, the coaction can fix identities up to primitive elements (namely motivic MZV's of depth 1, at some root of unity \( \zeta^\mot\sgnargsm{\exp(2 \pi \ii a/b)}{n} \)).  More concretely the coaction can be used to fix the coefficient of product terms involving \( \zeta(2n) \), in contrast to the coproduct above.  (One can think of Goncharov's motivic iterated integrals as \( \mathcal{H} / \zeta^\mot(2) \mathcal{H} \), wherein \( \zeta^\mot(2) \) is killed.)  
		
		One has a surjective \( \Q \)-algebra homomorphism \( \per \), called the period map,
		\begin{align*}
			\per \colon \mathcal{H} &\to \C \\
				I^\mot(x_0; x_1,\ldots,x_N; x_{N+1}) &\mapsto I(x_0; x_1,\ldots,x_N; x_{N+1}) \,,
		\end{align*}
		which means that the classical iterated integrals satisfy all motivically true relations.  Conjecturally, the space of motivic iterated integrals is isomorphic to the space of classical iterated integrals.  This conjecture is a special case of the Grothendeick period conjecture, which posits that the period map \( \per \) (in the most general setting) is injective, so that all relations are motivic (`geometric') in origin, i.e. there are no `spurious' or `coincidental' relations on the level of numbers. \medskip
		
		We briefly recall some main relations satisfied by the motivic iterated integrals.
		\begin{itemize}
			\item[i)] Unit: \( I^\mot(x_0; x_1) = 1 \) in weight 0,
			\item[ii)] Trivial integration: \( I^\mot(x_0; x_1,\ldots, x_N; x_{N+1}) = 0 \) if \( x_0 = x_{N+1} \) and \( N \geq 1 \),
			\item[iii)] Path composition: for any \( y \in \overline{\Q}  \),
			\[
				I^\mot(x_0; x_1,\ldots,x_N; x_{N+1}) = \sum_{i=0}^N I^\mot(x_0; x_1,\ldots,x_i; y) I^\mot(y; x_{i+1}, \ldots, x_N; x_{N+1})
			\]
			\item[iv)] Path reversal: \( I^\mot(x_0; x_1,\ldots,x_N;x_{N+1}) = (-1)^N I^\mot(x_{N+1}; x_N, \ldots, x_1; x_0) \)
			\item[v)] Homothety: if \( x_0 \neq x_1 \), and \( x_N \neq x_{N+1} \), then for any \( \alpha \in \overline{\Q} \),
			\[
				I^\mot(x_0; x_1,\ldots,x_N; x_{N+1}) = I^\mot(\alpha \cdot x_0;\alpha \cdot  x_1,\ldots, \alpha \cdot  x_N;\alpha \cdot  x_{N+1})
			\]
		\end{itemize}
		
		In the cases where \( x_0 = x_1 \) or \( x_N = x_{N+1} \) the motivic iterated integrals are regularised in the same manner as described in \autoref{sec:mzvreg}, via the shuffle product.
				
		\begin{Rem}[Tangential base-points]
			More formally, for the regularisation process, we replace \( 0 \) and \( 1 \) with a tangential base-points \( \overrightarrow{1}_0 \) and \( \overrightarrow{-1}_1 \) which denote the tangent vector \( \overrightarrow{1} \) at the point 0, and \( \overrightarrow{-1} \) at the point \( 1 \), respectively.  For the straight line path \( \mathrm{dch} \colon [0,1] \to [0,1] \), these are indeed the tangent vectors under consideration.  The details of iterated integrals with tangential base points (and the motivic versions thereof) can be found in \cite[Section 3.7, Section 4.5]{fresanBook}.  The notation with tangential base-points is helpful to identify when certain transformations and relations are invalid in the case of regularised integrals.
		\end{Rem}
	
		The homothety property fails if \( x_0 = x_1 \) or if \( x_N = x_{N+1} \), because in this case the integral depends on the vector of the tangential base-points at \( x_0 \) or \( x_{N+1} \), which are changed when we scale by \( x_i \mapsto x_i \alpha \).  This point is glossed over in \cite[Section 2.3]{glanoisTh}, \cite[Section 2.2]{glanoisUnramified} and in \cite[Section 2]{murakami21}.  However, whenever the homothety property is applied in \cite{murakami21}, one only needs it to hold modulo products and \( \zeta^\umot(2) \), i.e. in the Lie coalgebra \( \mathcal{L}_\bullet = \mathcal{A}_\bullet / \mathcal{A}_{>0} \cdot \mathcal{A}_{>0} \) which we introduce momentarily.  This version of homothety does hold in general for \( x_i \in \{ 0, \pm 1 \} \), and so Murakami's conclusions are valid; see \autoref{rem:homothety} below for more explicit details.
		
		\subsection{\texorpdfstring{Motivic multiple zeta values and motivic multiple \( t \) values}{
			Motivic multiple zeta values and motivic multiple t values}}
		
		For \( \ell \in \Z_{\geq0} \), any index \( \vec{k} = (k_1,\ldots,k_d) \) with \( k_i \in \Z_{\geq1} \) and any choice of signs \( \epsilon_i \in \{ \pm 1\} \), we define the \emph{(alternating) motivic multiple zeta values} by
		\[
			\zeta_\ell^\mot\sgnarg{\epsilon_1,\ldots,\epsilon_d}{k_1,\ldots,k_d} \coloneqq (-1)^d I^\mot(0; \{0\}^\ell, \eta_1, \{0\}^{k_1-1}, \eta_2, \{0\}^{k_2-1}, \ldots, \eta_d, \{0\}^{k_d-1}; 1) \,,
		\]
		where \( \{k\}^n = \overbrace{k, \ldots, k}^n \) denotes the argument \( k \) repeated \( n \) times, and \( \eta_i = \prod_{j=1}^d \epsilon_j \).  This arises by transposing \eqref{eqn:zasi} to the motivic world, as a definition, and extending to the case of leading 0's, as already indicated in \autoref{sec:mzvreg}.
		
		When \( \ell > 0 \), this integral is regularised in the same manner as described in \autoref{sec:mzvreg}, in particular via \eqref{eqn:unshufflestart0} in \autoref{lem:unshufflestart0} to express \( \zeta^\mot_\ell \) in terms of \( \zeta^\mot_0 = \zeta^\mot \).  We have a further property
		\begin{itemize}
			\item[vi)] Unshuffling of 0's:
			\[
				\zeta_\ell^{\mot}\sgnarg{\epsilon_1,\ldots,\epsilon_d}{k_1,\ldots,k_d} = (-1)^\ell \sum_{i_1 + \cdots + i_d = \ell} \binom{k_1 + i_1 - 1}{i_1} \cdots \binom{k_d + i_d - 1}{i_d} \, \zeta_0^\mot\sgnarg{\epsilon_1,\ldots,\epsilon_d}{k_1 + i_1, \ldots, k_d + i_d}
			\]
		\end{itemize}
		When \( \ell = 0 \), we can write \( \zeta^\mot \) instead of \( \zeta^\mot_0 \).  When all signs \( \epsilon_i = 1 \), we can write
		\[
			\zeta_\ell^\mot\sgnarg{1,\ldots,1}{k_1,\ldots,k_d} \eqqcolon \zeta_\ell^\mot(k_1,\ldots,k_d) \,,
		\]
		and will refer to this as a \emph{(non-alternating) motivic MZV}.  As shorthand notation, we also write \( \overline{k_i} \) to denote the argument \( k_i \) which has associated sign \( \epsilon_i = -1 \), and write \( \zeta^\mot \) with only one row of arguments.
		
		It is convenient to give notation to the space of all motivic MZV's and all motivic alternating MZV's, within the space of all motivic iterated integrals.
		
		\begin{Def}[Space of (alternating) motivic MZV's]
			Let \( \mathcal{H}^{(1)} \) be the \( \Q \)-vector space generated by all (non-alternating) motivic MZV's.  Likewise, let \( \mathcal{H}^{(2)} \) be the \( \Q \)-vector space generated by all alternating motivic MZV's.  Moreover, write \( \mathcal{H}^{(1)}_N \), or \( \mathcal{H}^{(2)}_N \) for the space of weight \( N \) (non-alternating) motivic MZV's, and weight \( N \) alternating motivic MZV's respectively.
		\end{Def}
		
		We then define the motivic multiple \( t \) values, using \eqref{eqn:tasz} as follows.
		
		\begin{Def}[Motivic multiple \( t \) value]
			For any index \( \vec{k} = (k_1,\ldots,k_d) \), \( k_i \in \Z_{\geq1} \), the \emph{motivic multiple \( t \) value} \( t^\mot(k_1,\ldots,k_d) \) is defined by
			\[
				t^\mot(k_1,\ldots,k_d) \coloneqq \frac{1}{2^d} \sum_{\epsilon_i \in \{ \pm 1 \}} \epsilon_1 \cdots \epsilon_d \, \zeta^\mot\sgnarg{\epsilon_1,\ldots,\epsilon_d}{k_1,\ldots,k_d} \,.
			\]
			It will often be convenient to work with the following rescaled version
			\[
				\ttw^\mot(k_1,\ldots,k_d) \coloneqq 2^{\abs{\vec{k}}-d} \sum_{\epsilon_i \in \{ \pm 1 \}} \epsilon_1 \cdots \epsilon_d \, \zeta^\mot\sgnarg{\epsilon_1,\ldots,\epsilon_d}{k_1,\ldots,k_d} \,,
			\]
			where \( \abs{\vec{k}} = k_1 + \cdots + k_d \) denotes the weight of the index \( \vec{k} \).
		\end{Def}

		We call a motivic MtV \( t^\mot(k_1,\ldots,k_d) \)  or \( \ttw^\mot(k_1,\ldots,k_d) \) 
		\begin{itemize}
			\item[i)] a \emph{convergent motivic MtV} if \( k_d \geq 2 \), and 
			\item[ii)] an \emph{extended motivic MtV} if \( k_d \geq 1 \).
		\end{itemize}
		By view of \eqref{eqn:tasz} and \eqref{eqn:tshuffle}, we see that the image of \( t^\mot(k_1,\ldots,k_d) \) under the period map,
		\[
			\per t^\mot(k_1,\ldots,k_d) = t^{\shuffle,0}(k_1,\ldots,k_d) \,,
		\]
		gives the shuffle-regularised multiple \( t \) value \( t^{\shuffle,0}(k_1,\ldots,k_d) \) arising from the shuffle regularisation with parameter \( \zeta^{\shuffle,0}(1) = 0 \).  Under the period map, the convergent motivic MtV's give convergent MtV's in the sense of \autoref{def:mtvconv}, and in particular correspond to convergent series.  Likewise the extended motivic MtV's correspond to extended MtV's in the sense of \autoref{def:mtvconv}, and require regularisation to be defined.
		
		\subsection{\texorpdfstring{Derivations \( D_r \), and the kernel of \( D_{<N} \)}{
				Derivations D\_r, and the kernel of D\_<N}}
			\label{sec:dr}
		
		Finally, we turn to one of the most useful features of the motivic MZV's, the combinatorial operations \( D_r \) arising from the coaction, which allow us to recursively find and verify identities. \medskip
		
		Recall the coaction \( \Delta \colon \mathcal{H}_\bullet \to \mathcal{A}_\bullet \otimes_\Q \mathcal{H}_\bullet \) defined in \eqref{eqn:coaction}.  We wish to consider a linearised version of this, which is computationally less complex to calculate, but still very rich in information.  By the earlier remark, we have that \( \mathcal{A}_\bullet = \mathcal{H} / \mathcal{H}\zeta^\mot(2) \).  Moreover, introduce the linearised quotient of \( \mathcal{A}_\bullet \) -- which then has the structure of a Lie coalgebra -- defined by
		\[
			\mathcal{L}_\bullet = \mathcal{A} / \mathcal{A}_{>0} \cdot \mathcal{A}_{>0} \,.
		\]
		Denote by \( I^\lmot \) and \( \zeta^\lmot \), the image of \( I^\mot \) and \( \zeta^\mot \) respectively, in \( \mathcal{L} \).
		
		\begin{Def}[Derivation \( D_r \)]
		For any \( r \geq 1 \) odd, define the derivation
		\[
			D_r \colon \mathcal{H} \to \mathcal{L}_r \otimes_\Q \mathcal{H} 
		\]
		as the composition of \( \Delta - (1\otimes\id) \) with \( \pi_r \otimes \id \), where \( \pi_r \) is the projection \( \mathcal{A} \to \mathcal{L} \to \mathcal{L}_r \).
		\end{Def}
	
		Essentially \( D_r \) is given by the terms in \( \Delta \) which have weight \( r \) in the left hand factor, and are irreducible.  Therefore, one has the following explicit and combinatorial formula to compute \( D_r \),
		\begin{align*}
			& D_r\big( I^\mot(x_0; x_1,\ldots,x_N; x_{N+1}) \big) = {} \\
			& \sum_{p=0}^{N-r} I^\lmot(x_p; x_{p+1}, \ldots, x_{p+r}; x_{p+r+1}) \otimes I^\mot(x_0; x_1,\ldots, x_p, x_{p+r+1}, \ldots, x_N; x_{N+1})
		\end{align*}
		Often, the following mnemonic picture is used to describe this formula.
		
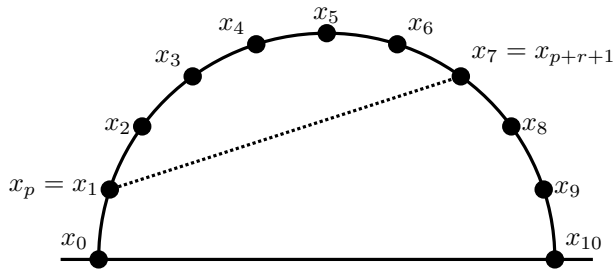
\begin{center}
			\begin{tikzpicture}[style = {very thick}]
			\draw (3,0) arc [start angle=0, end angle = 180, radius=3]
			node[pos=0.0, name=x10, inner sep=0pt] {}
			node[pos=0.1, name=x9, inner sep=0pt] {}
			node[pos=0.2, name=x8, inner sep=0pt] {}
			node[pos=0.3, name=x7, inner sep=0pt] {}
			node[pos=0.4, name=x6, inner sep=0pt] {}
			node[pos=0.5, name=x5, inner sep=0pt] {}
			node[pos=0.6, name=x4, inner sep=0pt] {}
			node[pos=0.7, name=x3, inner sep=0pt] {}
			node[pos=0.8, name=x2, inner sep=0pt] {}
			node[pos=0.9, name=x1, inner sep=0pt] {}
			node[pos=1.0, name=x0, inner sep=0pt] {};
			\draw (-3.5,0) -- (3.5,0);
			\filldraw
			(x0) circle [radius=0.1, fill=black] node[above left] {$x_0$}
			(x1) circle [radius=0.1, fill=black] node[left] {$x_p = x_1$}
			(x2) circle [radius=0.1, fill=black] node[left] {$ x_2$}
			(x3) circle [radius=0.1, fill=black] node[above left] {$x_3$}
			(x4) circle [radius=0.1, fill=black] node[above left] {$x_4$}
			(x5) circle [radius=0.1, fill=black] node[above] {$x_5$}
			(x6) circle [radius=0.1, fill=black] node[above right] {$x_6$}
			(x7) circle [radius=0.1, fill=black] node[above right] {$x_7 = x_{p+r+1}$}
			(x8) circle [radius=0.1, fill=black] node[right] {$x_8$}
			(x9) circle [radius=0.1, fill=black] node[right] {$x_9$}
			(x10) circle [radius=0.1, fill=black] node[above right] {$x_{10}$};
			\draw [densely dotted] (x1) edge (x7);
			\end{tikzpicture}
		\end{center} 
		The terms in \( D_r \) correspond to segments cut out of the semicircular polygon with vertices labelled by the integral parameters \( x_0, x_1,\ldots, x_N, x_{N+1} \).  (In this picture, \( N = 9 , r = 5, p = 1 \).)  Each term corresponds to a particular segment which cuts off a small polygon with \( r \) interior points.  The small polygon \( (x_1,x_2,\ldots,x_7) \) above gives the left hand factor \( I^\lmot(x_1; x_2,\ldots,x_6;x_7) \) in the formula, while the main polygon, containing the integration endpoints \( x_0 \) and \( x_{N+1} = x_{10} \) gives rise to the right hand factor \( I^\mot(x_0; x_1, x_7, \ldots, x_9; x_{10}) \), by deleting the interior points from the segment. \medskip
		
		The following theorems illustrate the power and information contained in these operations.  For \( N \geq 1 \), write
		\[
			D_{<N} = \oplus_{1 \leq 2r+1 \leq N} D_{2r+1}
		\]
		as the overall combination of all (relevant) derivations in weight \( <N \).
		\begin{Thm}[Brown, Theorem 3.3 \cite{brown12}]
			The kernel of \( D_{<N} \) on motivic MZV's is 1 dimensional in weight \( N \), and spanned by \( \zeta^\mot(N) \),
			\[
				\ker D_{<N} \cap \mathcal{H}^{(1)} = \zeta^\mot(N) \Q \,.
			\]
		\end{Thm}
		This (often) allows one to recursively lift identities of real MZV's to motivic MZV's, by recursively verifying \( D_{<N} \) vanishes, and using the numerical identity to fix the final unknown coefficient of \( \zeta^\mot(N) \) via the period map.
		
		This was extended by Glanois \cite{glanoisTh,glanoisUnramified} to the case of alternating motivic MZV's (and motivic MZV's at higher roots of unity).
		\begin{Thm}[Glanois, Corollary 2.4.5 \cite{glanoisTh}, Theorem 2.2 \cite{glanoisUnramified}]\label{thm:glanois:kerDN}
			The kernel of \( D_{<N} \) on alternating motivic MZV's is 1 dimensional in weight \( N \), and spanned by \( \zeta^\mot(\overline{N}) \),
			\[
			\ker D_{<N} \cap \mathcal{H}^{(2)} = \zeta^\mot(\overline{N}) \Q \,.
			\]
		\end{Thm}
	
		This again (often) allows one to recursively lift identities of real alternating MZV's to alternating motivic MZV's, by recursively verifying \( D_{<N} \) vanishes, and using the numerical identity to fix the final unknown coefficient of \( \zeta^\mot(\overline{N}) \) via the period map.
	
		For \( N > 1 \), one can take \( \zeta^\mot(N) \) instead as the generator, however, for \( N = 1 \), one must take \( \zeta^\mot(\overline{1}) = \log^\mot(2) \), as \( \zeta^\mot(1) = 0 \).
		
		\begin{Rem}
			The analogous result for higher roots of unity is not true, as further primitive elements come into play.  For example \( \zeta^\mot(N) \) and \( \zeta^\mot\sgnargsm{\exp(2 \pi \ii /3)}{N} \) are both primitive for \( \Delta \), and therefore vanish under all derivations \( D_r \).  However, (after application of the period map, to take real and imaginary parts), one sees they are linearly independent.  Therefore the kernel of \( D_{<N} \) on motivic MZV's at 3rd roots of unity is (at least) two dimensional.  Glanois gives such characterisations in more cases in Corollary 2.4.5  \cite{glanoisTh}.
		\end{Rem}
		
		Finally, Glanois also studied when alternating motivic MZV's Galois descend to be(come) linear combinations of (non-alternating) motivic MZV's.  The following Theorem gives a criterion to check this recursively using \( D_r \).
		\begin{Thm}[Glanois, {\cite[Corollary 5.1.3]{glanoisTh}}, {\cite[Corollary 2.4]{glanoisUnramified}}]\label{thm:glanois:descent}
		Let \( \mathfrak{Z} \in \mathcal{H}^{(2)} \) be a motivic alternating MZV.  Then \( \mathfrak{Z} \in \mathcal{H}^{(1)} \), i.e. \( \mathfrak{Z} \) is a linear combination of (non-alternating) motivic MZV's, if and only if
		\begin{itemize}
			\item[i)] \( D_1(\mathfrak{Z}) = 0 \), and
			\item[ii)] \( D_{2r+1} \mathfrak{Z} \in \mathcal{L}^{(1)}_{2r+1} \otimes \mathcal{H}^{(1)}	\) for all \( r \geq 1 \) \,,
		\end{itemize}
		where \( \mathcal{L}^{(1)}_{2r+1} \) is the subspace of \( \mathcal{L} \) generated by all (non-alternating) motivic MZV's of weigh \(2r+1\).
		\end{Thm}
		
	\section{\texorpdfstring{Regularised distribution relations, and the derivations \( D_r \)}
		{Regularised distribution relations, and the derivations D\_r}}
	\label{sec:dt}
	
	In \cite{murakami21}, Murakami calculated the derivations \( D_r \) on motivic MtV's of the form \( t^\mot(k_1,\ldots,k_d) \) with each \( k_i > 1 \).  The case where some \( k_i = 1 \) was not treated, as the distribution relations used in the proof would not hold exactly.  For the purposes of treating the more general case, we need to consider the case of regularised distribution relations, and verify them on the motivic level.
	
	\subsection{Classic and motivic regularised distribution relations}
	
	The (convergent) distribution relation of `level \( N=2 \)' states that if \( k_d > 1 \), the following holds
	\[
		2^{k_1 + \cdots + k_d - d} \sum_{\epsilon_1,\ldots,\epsilon_d \in \{\pm1\}} \zeta\sgnarg{\epsilon_1,\ldots,\epsilon_d}{k_1,\ldots,k_d} = \zeta(k_1,\ldots,k_d) \,,
	\]
	This immediately follows from a corresponding distribution relation for multiple polylogarithms which holds on the power-series level
	\[
		2^{k_1 + \cdots + k_d - d} \sum_{\epsilon_1,\ldots,\epsilon_d \in \{\pm1\}} \Li_{k_1,\ldots,k_d}(\epsilon_1 x_1,\ldots, \epsilon_d x_d)  = \Li_{k_1,\ldots,k_d}(x_1^2,\ldots,x_d^2) \,,
	\]
	by setting \( x_i = 1 \).  The distribution relations are known to be motivic, and Murakami indeed even verified this again to be the case in Proposition 10 of \cite{murakami21}, at least when all \( k_i > 1 \).  Geometrically, they follow by taking the pullback under \( s \mapsto s^2 \) (for level \( N = 2 \)), and analogously in general.  For example
	\[
		\sum_{\epsilon_1,\epsilon_2\in\{\pm1\}} \zeta\sgnarg{\epsilon_1,\epsilon_2}{a,b} = \begin{aligned}[t]
			& I(0; 1, \{0\}^{a-1}, 1, \{0\}^{b-1}; 1) + I(0; 1, \{0\}^{a-1}, -1, \{0\}^{b-1}; 1) \\
			& + I(0; -1, \{0\}^{a-1}, 1, \{0\}^{b-1}; 1) + I(0; -1, \{0\}^{a-1}, -1, \{0\}^{b-1}; 1) \,.
		\end{aligned}
	\]
	where as always \( a_i \) within the bounds of the integral represents the form \( \frac{\mathrm{d}s}{s-a_i} \), as in \eqref{eqn:itint} above.  By linearity of integration, the forms can be combined as
	\[
		\frac{\mathrm{d}s}{s-1} + \frac{\mathrm{d}s}{s+1} = \frac{2s\mathrm{d}s}{s^2-1} \,,
	\]
	and so we can write the combination of integrals as 
	\[
		= \int\limits_{0 < s_1 < \cdots < s_{a+b} < 1} \frac{2s_1 \mathrm{d}s_1}{s_1^2-1} \wedge \overbrace{\frac{\mathrm{d}s_2}{s_2} \wedge \cdots \wedge \frac{\mathrm{d}s_{a}}{s_a}}^{\text{\( a - 1 \) terms}} \wedge \frac{2s_{a+1} \mathrm{d}s_{a+1}}{s_{a+1}^2-1} \wedge \overbrace{\frac{\mathrm{d}s_{a+2}}{s_{a+2}} \wedge \cdots \wedge \frac{\mathrm{d}s_{a+b}}{s_{a+b}}}^{\text{\( b - 1 \) terms}} \,,
	\]
	Now set \( y_i = s_i^2 \), for which the bounds \( 0 < s_1 < \cdots < s_{a+b} < 1 \) become \( 0 < y_1 < \cdots < y_{a+b} < 1 \), and the forms become \[
		\frac{\mathrm{d}y_i}{2y_i} = \frac{\mathrm{d}s_i}{s_i} \, \quad \frac{\mathrm{d}y_i}{y_i - 1} = \frac{2s_i \mathrm{d}s_i}{s_i^2 - 1} \,.
	\]
	This means the integral is equal to
	\begin{align*}
		&= \int\limits_{0 < y_1 < \cdots < y_{a+b} < 1} \frac{\mathrm{d}y_1}{y_1-1} \wedge \overbrace{\frac{\mathrm{d}y_2}{2y_2} \wedge \cdots \wedge \frac{\mathrm{d}y_{a}}{2y_a}}^{\text{\( a - 1 \) terms}} \wedge \frac{\mathrm{d}y_{a+1}}{y_{a+1}-1} \wedge \overbrace{\frac{\mathrm{d}y_{a+2}}{2y_{a+2}} \wedge \cdots \wedge \frac{\mathrm{d}y_{a+b}}{2y_{a+b}}}^{\text{\( b - 1 \) terms}} \\
		&= \frac{1}{2^{a+b-2}} \zeta(a,b) 
	\end{align*}
	Therefore (modulo some formalities to translate this carefully), they indeed have a geometric (`motivic') origin. \medskip
	
	On the level of real numbers we claim the following version of the distribution relations of level \( N = 2 \) holds.  (A regularised version of the distribution relations is discussed in general in \cite[Section 13.3.4]{zhaoBook}.)
	\begin{Prop}\label{prop:regdist}
		For \( \vec{k} = (k_1,\ldots,k_d) \), with \( k_d \neq 1 \) an index, and any \( \alpha \geq 0 \), the following regularised version of the distribution relation holds
		\begin{align*}
			2^{k_1+\cdots+k_d + \alpha - (d+\alpha)} \!\! \sum_{\substack{\vec{\epsilon} = (\epsilon_1,\ldots,\epsilon_{d}) \\ \vec{\delta} = (\delta_1,\ldots,\delta_\alpha) \\ \epsilon_i, \delta_j \in \{ \pm 1\}}}  \!\! \zeta^{\shuffle,W}\sgnarg{\vec{\epsilon}, \mathrlap{\,\vec{\delta}}\phantom{\{1\}^\alpha}}{\vec{k},\{1\}^\alpha} 
			- \zeta^{\shuffle,W}(\vec{k},\{1\}^{\alpha}) = \sum_{i=1}^\alpha \zeta^{\shuffle,W}(\vec{k}, \{1\}^{\alpha-i}) \frac{(-\log(2))^i}{i!}
		\end{align*}
		
		\begin{proof}
			We first of all convert this to a statement about integrals; we also multiply by the global sign \( (-1)^{d + \alpha} \) in the integral representation of these MZV's.  The claim is then for any \( \alpha \geq 0 \)
			\begin{align*}
				& 2^{k_1 + \cdots + k_d - d} \!\! \sum_{\substack{\vec{\epsilon} = (\epsilon_1,\ldots,\epsilon_{d}) \\ \vec{\delta} = (\delta_1,\ldots,\delta_\alpha) \\ \epsilon_i, \delta_j \in \{ \pm 1\}}}  \!\! I^{\shuffle,W}(0; \epsilon_1, \{0\}^{k_1-1}, \ldots, \epsilon_d, \{0\}^{k_d-1}, \delta_1, \ldots, \delta_\alpha; 1) \\
				&= \sum_{i=0}^\alpha I^{\shuffle,W}(0; 1, \{0\}^{k_1-1}, \ldots, 1, \{0\}^{k_d-1}, \{1\}^{\alpha-i}; 1) \frac{I(0; -1; 1)^i}{i!} \,.
			\end{align*}
			The proof proceeds inductively on \( \alpha \).  The case \( \alpha = 0 \) is true by virtue of being the non-regularised distribution relation, so we may take \( \alpha \geq 1 \).  For notational ease, write \( w_0 = 1\{0\}^{k_1-1}\cdots1\{0\}^{k_d-2} \), and \( w = w_00 \) where the last letter is noted (explicitly) to be 0 since \( k_d \geq 2 \).  We then claim in the shuffle algebra (writing \( \oplus \) for concatenation as emphasis, and clarity)
			\begin{equation}\label{eqn:regdist:shuffleid}
			\begin{aligned}
				w_0 0 \oplus \delta_1 \cdots \delta_\alpha = {} &\sum_{i=0}^{\alpha-1} (-1)^{\alpha-i-1} ( w_0 0 \oplus \delta_1 \cdots \delta_{i}) \shuffle \delta_\alpha \delta_{\alpha-1} \cdots \delta_{i+1} \\ 
				& + (-1)^\alpha (w_0 \shuffle \delta_\alpha \delta_{\alpha-1} \cdots \delta_2 \delta_1) \oplus 0
			\end{aligned}
			\end{equation}
			This is verified by a pair-wise cancellation of each term when expanding out via the recursive definition \( (wa) \shuffle (vb) = (w \shuffle vb) a + (wa \shuffle v) b \) of the shuffle product.  Now sum (the integral of) this identity over all choices of signs \( \epsilon_1,\ldots,\epsilon_d \) and \( \delta_1,\ldots,\delta_\alpha \).  Every integral appearing on the right hand side and involving \( w_0 \) has strictly fewer trailing \( \pm 1 \) terms, so we can inductively apply the proposition here.  Moreover since all choices of \( \pm 1 \) appear in every position, we have directly by the shuffle product
			\[
				\sum_{\delta_1,\ldots,\delta_k \in \{\pm1\}} I^{\shuffle,W}(0; \delta_1,\ldots,\delta_k; 1) = \frac{1}{k!} (I^{\shuffle,W}(0; 1; 1) + I(0; -1, 1))^k \,.
			\]
			So after summing over all signs, we can make these substitutions into the following expression
			\begin{align*}
				& \sum_{\substack{\epsilon_1,\ldots,\epsilon_d \\ \delta_1,\ldots, \delta_\alpha \in \{ \pm 1\}}}  \!\! I^{\shuffle,W}(0; \epsilon_1, \{0\}^{k_1-1}, \ldots, \epsilon_d, \{0\}^{k_d-1}, \delta_1, \ldots, \delta_\alpha; 1) \\
				& = \begin{aligned}[t] 
					 \sum_{i=0}^{\alpha-1} (-1)^{\alpha-i-1} \!\!\! \sum_{\substack{\epsilon_1,\ldots,\epsilon_d, \\ \delta_1,\ldots, \delta_{i} \in \{ \pm 1\}}}  \!\!\! I^{\shuffle,W}(0; \epsilon_1, \{0\}^{k_1-1}, \ldots, \epsilon_d, \{0\}^{k_d-1}, \delta_1, \ldots, \delta_{i}; 1) & \\[-2ex]
					 {} \cdot \sum_{\substack{\delta_{i+1},\ldots,\delta_{\alpha} \\ \in \{ \pm 1\}}}  \!\! I^{\shuffle,W}(0; \delta_\alpha, \ldots, \delta_{i+1}; 1) & \end{aligned} \\
				&\hspace{1em} + \sum_{\substack{\epsilon_1,\ldots,\epsilon_d, \\ \delta_1,\ldots, \delta_{\alpha} \in \{ \pm 1\}}}  (-1)^\alpha I^{W,\shuffle}(0; \epsilon_1\{0\}^{k_1-1}\cdots\epsilon_d\{0\}^{k_d-2} \shuffle \delta_\alpha \delta_{\alpha-1} \cdots \delta_2 \delta_1) \oplus 0; 1)
			\end{align*}
			(The last term here is still written as a word in the shuffle algebra, for lack of any significantly better notation.)  Overall we find
			\begin{align*}
				& = \begin{aligned}[t] 
				\sum_{i=0}^{\alpha-1} (-1)^{\alpha-i-1} \cdot 2^{d-k_1 - \cdots - k_d} \, \sum_{j=0}^{i} I^{\shuffle,W}(0; 1, \{0\}^{k_1-1}, \ldots, 1, \{0\}^{k_d-1}, \{1\}^{i-j}; 1) \cdot \frac{1}{j!} I(0; -1; 1)^j & \\[-2ex]
				{} \cdot \frac{1}{(\alpha-i)!} (I^{\shuffle,W}(0; 1; 1) + I(0; -1, 1))^{\alpha-i}  & \\
				 {} + 2^{d - k_1 - \cdots - k_d} (-1)^\alpha I^{W,\shuffle}(0; 1\{0\}^{k_1-1}\cdots1\{0\}^{k_d-2} \shuffle \{1\}^\alpha) \oplus 0; 1) \,.
				\end{aligned}
			\end{align*}
			Reverse the \( j \) sum, then switch the summation \( i \) and \( j \) order, and then set \( i \mapsto i-j \) to re-index the resulting \( i = j,\ldots,\alpha-1 \) as \( i = 0,\ldots,\alpha-1-j \).  This gives
			\begin{align*}
			& = \begin{aligned}[t] 
			2^{d-k_1 - \cdots - k_d} \, \sum_{j=0}^{\alpha-1} \sum_{i=0}^{\alpha-1-j} (-1)^{\alpha-i-j-1} I^{\shuffle,W}(0; 1, \{0\}^{k_1-1}, \ldots, 1, \{0\}^{k_d-1}, \{1\}^{j}; 1) \cdot \frac{1}{i!} I(0; -1; 1)^{i} & \\[-2ex]
			{} \cdot \frac{1}{(\alpha-i-j)!} (I^{\shuffle,W}(0; 1; 1) + I(0; -1, 1))^{\alpha-i-j}  & \\
			{} + 2^{d - k_1 - \cdots - k_d} (-1)^\alpha I^{W,\shuffle}(0; 1\{0\}^{k_1-1}\cdots1\{0\}^{k_d-2} \shuffle \{1\}^\alpha) \oplus 0; 1) \,.
			\end{aligned}
			\end{align*}
			Now recognise the sum over \( i \) as a binomial sum (missing its last term \( i = \alpha -j \)), giving
			\begin{align*}
			& = \begin{aligned}[t] 
			2^{d-k_1 - \cdots - k_d} \, \sum_{j=0}^{\alpha-1} (-1)^{\alpha-j-1} I^{\shuffle,W}(0; 1, \{0\}^{k_1-1}, \ldots, 1, \{0\}^{k_d-1}, \{1\}^{j}; 1)  & \\[-2ex]
			{} \cdot \Big(
				\frac{1}{(\alpha-j)!} I^{\shuffle,W}(0; 1; 1)^{\alpha-j} - \frac{1}{(\alpha-j)!} & (-I(0; -1; 1))^{\alpha-j} 
			\Big)
			\\
			{} + 2^{d - k_1 - \cdots - k_d} (-1)^\alpha I^{W,\shuffle}(0; 1\{0\}^{k_1-1}\cdots1\{0\}^{k_d-2} \shuffle \{1\}^\alpha) \oplus 0; 1) \,.
			\end{aligned}
			\end{align*}
			This simplifies to the following			
			\begin{align*}
			& = 2^{d-k_1 - \cdots - k_d} \,  \bigg( \begin{aligned}[t] 
			&  \sum_{j=0}^{\alpha-1} I^{\shuffle,W}(0; 1, \{0\}^{k_1-1}, \ldots, 1, \{0\}^{k_d-1}, \{1\}^{j}; 1) \frac{I(0; -1; 1)^{\alpha-j}}{(\alpha-j)!}
			\\
			& + \sum_{j=0}^{\alpha-1} (-1)^{\alpha-j} I^{\shuffle,W}(0; 1, \{0\}^{k_1-1}, \ldots, 1, \{0\}^{k_d-1}, \{1\}^{j}; 1) \cdot I^{\shuffle,W}(0; \{1\}^{\alpha-j}; 1) \\
			& {} + (-1)^\alpha I^{W,\shuffle}(0; 1\{0\}^{k_1-1}\cdots1\{0\}^{k_d-2} \shuffle \{1\}^\alpha) \oplus 0; 1) \bigg) \,.
			\end{aligned}
			\end{align*}
			Finally, using \eqref{eqn:regdist:shuffleid} in reverse, the last two lines combine to give
			\[
				I^{\shuffle,W}(0; 1, \{0\}^{k_1-1}, \ldots, 1, \{0\}^{k_d-1}, \{1\}^{\alpha}; 1) \,,
			\]
			which is the \( j = \alpha \) term of the first sum.  Overall we obtain	
			\begin{align*}
			& = 2^{d-k_1 - \cdots - k_d} \, \begin{aligned}[t] 
			&  \bigg( \sum_{j=0}^{\alpha} I^{\shuffle,W}(0; 1, \{0\}^{k_1-1}, \ldots, 1, \{0\}^{k_d-1}, \{1\}^{j}; 1) \frac{I(0; -1; 1)^{\alpha-j}}{(\alpha-j)!} \Bigg) \,,
			\end{aligned}
			\end{align*}
			which is the identity we wanted, up to reversing the sum over \( j \).
		\end{proof}
	\end{Prop}
	
	In particular, the weight \( w > 1 \) distribution relations holds modulo products, whether or not regularisation is necessary.  In the case of weight \( 1 \) however, we find
	\[
		2^{0} \bigg( \zeta^{\shuffle,W}\sgnarg{1}{1} + \zeta^{\shuffle,W}\sgnarg{-1}{1} \bigg) - \zeta^{\shuffle,W}(1) = -\log(2) \,,
	\]
	which is non-zero modulo products (at least assuming the usual conjectures), as this is a logarithm (i.e. weight 1).  More precisely, on the motivic level, the weight 1 distribution identity (with \( W = 0 \)) is clearly satisfied, and then \( \log^\lmot(2) \) (the motivic logarithm \( \log^\mot(2) \) modulo products) does not vanish, as it lives in a the weight 1 component.  
	
	We now specialise to the case \( W = 0 \), in line with the usual regularisation prescription via the tangential base-points of the straight line path \( \gamma \colon [0,1] \to [0,1] \), \( \gamma(t) = t \).  In this prescription: \( I(0; 0; 1) = I(0; 1; 1) =  0 \).  Using this, we can extend \autoref{prop:regdist} to the case of \( \zeta_\ell^{\shuffle,0}(\vec{k},\{1\}^\alpha) \), wherein the integral representation starts with a string \( \{0\}^\ell \) of \( \ell \) many 0's.
	
	\begin{Cor}\label{cor:distrega}
		For \( \vec{k} = (k_1,\ldots,k_d) \), with \( k_d \neq 1 \) an index, any \( \alpha \geq 0 \), and any \( \ell \geq 0 \), the following regularised version of the distribution relation holds
		\begin{align*}
		2^{k_1+\cdots+k_d + \ell - d} \!\! \sum_{\substack{\vec{\epsilon} = (\epsilon_1,\ldots,\epsilon_{d}) \\ \vec{\delta} = (\delta_1,\ldots,\delta_\alpha) \\ \epsilon_i, \delta_j \in \{ \pm 1\}}}  \!\! \zeta_\ell^{\shuffle,0}\sgnarg{\vec{\epsilon}, \mathrlap{\,\vec{\delta}}\phantom{\{1\}^\alpha}}{\vec{k},\{1\}^\alpha} 
		- \zeta^{\shuffle,W}(\vec{k},\{1\}^{\alpha}) = \sum_{i=1}^\alpha \zeta_\ell^{\shuffle,0}(\vec{k}, \{1\}^{\alpha-i}) \frac{(-\log(2))^i}{i!}
		\end{align*}
		
		\begin{proof}
			In the case \( \alpha = 0 \), applying the unshuffling of starting 0's from \eqref{eqn:unshufflestart0} in \autoref{lem:unshufflestart0} shows that 
			\[
				\zeta_\ell^{\shuffle,0}\sgnarg{\vec{\epsilon}}{\vec{k}} = (-1)^\ell \sum_{i_1 + \cdots + i_d = \ell} \binom{k_1 + i_1 - 1}{i_1} \cdots \binom{k_d + i_d - 1}{i_d} \zeta\sgnarg{\vec{\epsilon}}{k_1 + i_1, \ldots, k_d + i_d} \,.
			\]
			This reduces \( \zeta_\ell \) to convergent zetas on the right hands side, to which the distribution relation applies exactly.  So after summing over all choices of signs, and applying the usual distribution relation, one finds
			\begin{align*}
				& \sum_{\substack{\vec{\epsilon} = (\epsilon_1,\ldots,\epsilon_d) \\ \epsilon_i \in \{\pm1\}}}\zeta_\ell^{\shuffle,0}\sgnarg{\vec{\epsilon}}{\vec{k}} \\
				& = 2^{k_1 + \cdots + k_d + \ell - d} (-1)^d \sum_{i_1 + \cdots + i_d = \ell} \binom{k_1 + i_1 - 1}{i_1} \cdots \binom{n_d + i_d - 1}{i_d} \zeta\sgnarg{\{1\}^d}{k_1 + i_1, \ldots, k_d + i_d} \\
				&= 2^{k_1 + \cdots + k_d + \ell - d}  \zeta_\ell^{\shuffle,0}(\vec{k}) \,,
			\end{align*}
			where the last equality arises by applying the unshuffling process again.  The case with trailing \( 1 \)'s follows by applying the proof of \autoref{prop:regdist} again, mutatis mutandis.
		\end{proof}
	\end{Cor}

	\begin{Rem}
		One can also give a version of the unshuffling identity \eqref{eqn:unshufflestart0} in \autoref{lem:unshufflestart0} which holds for different regularisation parameters, and so one can extend \autoref{cor:distrega} to the general regularisation parameter (even taking different regularisations for \( I(0; 0; 1) = W \) and \( I(0; 1; 1) = W' \)).
	\end{Rem}

	The proofs of \autoref{prop:regdist} and \autoref{cor:distrega} proceeded purely by using the shuffle product of iterated integrals and the non-regularised distribution relations, so the result holds true on the motivic level as well, as both ingredients are already known to be motivic.  So as a result, we obtain the following corollaries.
	
	\begin{Cor}\label{cor:distreg:mot}
		For \( \vec{k} = (k_1,\ldots,k_d) \), with \( k_d \neq 1 \) an index, any \( \alpha \geq 0 \), and any \( \ell \geq 0 \) the following regularised version of the distribution relation holds for motivic multiple zeta values.
		\begin{align*}
		2^{k_1+\cdots+k_d + \ell - d} \!\! \sum_{\substack{\vec{\epsilon} = (\epsilon_1,\ldots,\epsilon_{d}) \\ \vec{\delta} = (\delta_1,\ldots,\delta_\alpha) \\ \epsilon_i, \delta_j \in \{ \pm 1\}}}  \!\! \zeta^{\mot}_\ell\sgnarg{\vec{\epsilon}, \mathrlap{\,\vec{\delta}}\phantom{\{1\}^\alpha}}{\vec{k},\{1\}^\alpha} 
		- \zeta^{\mot}(\vec{k},\{1\}^{\alpha}) = \sum_{i=1}^\alpha \zeta_\ell^{\mot}(\vec{k}, \{1\}^{\alpha-i}) \frac{(-\log^\mot(2))^i}{i!}
		\end{align*}
	\end{Cor}

	\begin{Cor}\label{cor:distreg:lmot}
	For \( \vec{k} = (k_1,\ldots,k_d) \), with \( k_d \neq 1 \) an index, any \( \alpha \geq 0 \), and any \( \ell \geq 0 \), the following regularised version of the distribution relation holds for motivic multiple zeta values of weight \( w > 1 \) or \( \ell > 0 \) modulo products.
	\begin{align*}
	2^{k_1+\cdots+k_d + \ell - d} \!\! \sum_{\substack{\vec{\epsilon} = (\epsilon_1,\ldots,\epsilon_{d}) \\ \vec{\delta} = (\delta_1,\ldots,\delta_\alpha) \\ \epsilon_i, \delta_j \in \{ \pm 1\}}}  \!\! \zeta_\ell^{\lmot}\sgnarg{\vec{\epsilon}, \mathrlap{\,\vec{\delta}}\phantom{\{1\}^\alpha}}{\vec{k},\{1\}^\alpha} 
	= \zeta_\ell^{\lmot}(\vec{k},\{1\}^{\alpha}) \,.
	\end{align*}
	In the case of weight 1 and \( \ell = 0 \), the distribution relation modulo products has an extra \( -\log^\lmot(2) \) correction, namely
	\[
		\zeta^{\lmot}(1) + \zeta^{\lmot}(\overline{1}) = \zeta^\lmot(1)  -\log^\lmot(2) \,.
	\]
\end{Cor}

	\subsection{\texorpdfstring{Derivations on \( \ttw^\mot(k_1,\ldots,k_d) \)}
		{Derivations on t\textasciitilde\textasciicircum{}m(k\_1, ..., k\_d)}}

	Now that we have the motivic version of the distribution relations for arbitrary arguments, we may directly generalise Murakami's computation of \( D_r \) given in Proposition 11 of \cite{murakami21}.
	
	\begin{Prop}[Generalisation of Proposition 11, \cite{murakami21}]\label{prop:dkt}
		Let \( \vec{k} = (k_1,\ldots,k_d) \in (\Z_{\geq1})^d \) be an index.  Write \( \vec{k}_{i,j} = (k_i, \ldots, k_j) \) for a subindex of \( \vec{k} \) and \( \abs{(a_1,\ldots,a_r)} = a_1 + \cdots + a_r \) for the total (weight) of an index.  Then the derivation \( D_r \), \( r \) odd, is computed as follows
		\begin{align}
		& D_r \big( \ttw^\mot(k_1,\ldots,k_d) \big) = \notag \\
		& \label{eqn:dr:deconcat} \sum_{1 \leq j \leq d} \delta_{\abs{\vec{k}_{1,j}}=r} \ttw^\lmot(k_1,\ldots,k_j) \otimes \ttw^\mot(k_{j+1},\ldots,k_d) \\[1ex]
		& \label{eqn:dr:0eps} + \!\! \sum_{1 \leq i < j \leq d} \begin{aligned}[t] 
			\delta_{\abs{\vec{k}_{i+1,j}} \leq r < \abs{\vec{k}_{i,j}}-1} \Big( \zeta^\lmot_{r-\abs{\vec{k}_{i+1,j}}}&(k_{i+1},\ldots,k_j) - \delta_{r=1} \log^\lmot(2) \Big) \\
		&  {} \otimes \ttw^\mot(k_1,\ldots,k_{i-1},\abs{\vec{k}_{i,j}} - r, k_{j+1},\ldots,k_d) \end{aligned} \\[1ex]
		& \label{eqn:dr:eps0} - \!\! \sum_{1 \leq i < j \leq d} \begin{aligned}[t]
			\delta_{\abs{\vec{k}_{i,j-1}} \leq r < \abs{\vec{k}_{i,j}}-1} \Big( \zeta^\lmot_{r-\abs{\vec{k}_{i,j-1}}}&(k_{j-1},\ldots,k_i) - \delta_{r=1} \log^\lmot(2) \Big) \\
		& {} \otimes \ttw^\mot(k_1,\ldots,k_{i-1},\abs{\vec{k}_{i,j}} - r, k_{j+1},\ldots,k_d) \end{aligned}
		\end{align}		
	\end{Prop}

	\begin{proof}
		The proof of this proposition works in precisely the same way as Murakami's proof of the special case, where each \( k_i \geq 2 \), given in \cite[Proposition 11]{murakami21}.  In particular, the terms arise from the following cuts on the integral defining \( \ttw^\mot \)
		\begin{align*}
			(-1)^d 2^{d - k_1 - \cdots - k_d} & \ttw^\mot(k_1,\ldots,k_p) = \\
			& \sum_{\eta_i \in \{\pm1\}} \eta_1 I^\mot(0; \eta_1, \{0\}^{k_1-1}, \eta_2, \{0\}^{k_2-1}, \ldots, \eta_d, \{0\}^{k_d-1}; 1)
		\end{align*}
		The correspondence in particular is as follows
			\begin{center}
			\begin{tikzpicture}[baseline]
			\matrix[name=M1, matrix of nodes, inner sep=1pt, column sep=0pt]{
				\node () [] {$I^\mot($};
				& \node (s1) [] {${\phantom{\mathllap{\fbox{$,0$}}}} 0$}; 
				& \node () [] {$;$}; 
				& \node () [] {${\phantom{\mathllap{\fbox{$,0$}}}} \eta_1$}; 
				& \node () [] {$,$};
				& \node () [] {$0, \ldots, 0,$};
				& \node (xie) [] {${\phantom{\mathllap{\fbox{$,0$}}}}  \eta_i$};  			& \node () [] {$,$};
				& \node (xi2) [] {${\phantom{\mathllap{\fbox{$,0$}}}}  0$};  			& \node () [] {$,$};
				& \node (xi3) [] {${\phantom{\mathllap{\fbox{$,0$}}}}  \ldots$};  			& \node () [] {$,$};
				& \node (xiz) [] {${\phantom{\mathllap{\fbox{$,0$}}}}  0$};  & \node () [] {$,$};
				& \node (xi5) [] {${\phantom{\mathllap{\fbox{$,0$}}}}  \ldots$};  			& \node () [] {$,$};
				& \node (xi6) [] {${\phantom{\mathllap{\fbox{$,0$}}}}  0$};  			& \node () [] {$,$};
				& \node () [] {$\eta_{i+1},\ldots,$};
				& \node (yie) [] {${\phantom{\mathllap{\fbox{$,0$}}}}  \eta_j$};  			& \node () [] {$,$};
				& \node (yi2) [] {${\phantom{\mathllap{\fbox{$,0$}}}}  0$};  			& \node () [] {$,$};
				& \node (yi3) [] {${\phantom{\mathllap{\fbox{$,0$}}}}  \ldots$};  			& \node () [] {$,$};
				& \node (yiz) [] {${\phantom{\mathllap{\fbox{$,0$}}}}  0$};  & \node () [] {$,$};
				& \node (yi5) [] {${\phantom{\mathllap{\fbox{$,0$}}}}  \ldots$};  			& \node () [] {$,$};
				& \node (yi6) [] {${\phantom{\mathllap{\fbox{$,0$}}}}  0$};  			& \node () [] {$,$};
				& \node (zie) [] {${\phantom{\mathllap{\fbox{$,0$}}}} \eta_{j+1}$};
				& \node () [] {${\phantom{\mathllap{\fbox{$,0$}}}} \ldots$};
				& \node () [] {$) \,.$};
				\\
			};
			\draw[thick,black] (xiz.north) to[bracket=12pt] node[above=-4mm] {$_\eqref{eqn:dr:0eps}$} (zie.north);
			\draw[thick,black] (s1.north) to[bracket=17pt] node[below=-4mm] {$_{\eqref{eqn:dr:deconcat}}$} (zie.north);
			\draw[thick,black] (xie.south) to[ubracket=12pt] node[below=-4mm] {$_{\eqref{eqn:dr:eps0}}$} (yiz.south);
			\draw[thick,black] (xie.south) to[ubracket=17pt] node[above=-5mm] {$_{(\sum=0)}$} (zie.south);
			\draw[black,thin,solid] ($(xie.north west)+(0,0)$)  rectangle ($(xi6.south east)+(0,0)$);
			\draw[black,thin,solid] ($(yie.north west)+(0,0)$)  rectangle ($(yi6.south east)+(0,0)$);
			\draw[black,thin,solid] ($(s1.north west)+(0,0)$)  rectangle ($(s1.south east)+(0,0)$);
			\end{tikzpicture}
		\end{center}
		The only difference from Murakami's proof \cite[Proposition 11, proof]{murakami21} comes when computing \( D_1 \), wherein terms \eqref{eqn:dr:0eps} and \eqref{eqn:dr:eps0} are simplified using the regularised distribution relation, and so pick up an extra \( {-}\log^\lmot(2) \) in weight 1. 
	\end{proof}
	
	\begin{Rem}\label{rem:homothety}
		We should note here, again, that the iterated integrals \( I^\mot(a_0; a_1,\ldots,a_n; a_{n+1}) \) satisfy the homothety
	\[
	I^\mot(\lambda a_0; \lambda a_1, \ldots, \lambda a_n; \lambda a_{n+1}) = I^\mot(a_0; a_1,\ldots,a_n; a_{n+1}) \,,
	\]
	if \( a_0 \neq a_1 \) and \( a_n \neq a_{n+1} \).  However, if one of these is actually an equality, the integral must be regularised, and the homothety can  actually change the tangential base-point, so the equality does not in general hold.  Viz:
	\begin{align*}
	I^\mot(\overrightarrow{1}_{0}; 0, 1; -1) &= I^\mot(\overrightarrow{1}_{0}; 0; -1) I^\mot(\overrightarrow{1}_{0}; 1; -1) - I^\mot(\overrightarrow{1}_{0}; 1, 0; -1) \\
	&= (\ii \pi)^\mot \log^\mot(2) + \zeta^\mot(\overline{2}) \\
	I^\mot(\overrightarrow{1}_{0}; 0, -1; 1) &= I^\mot(\overrightarrow{1}_{0}; 0; 1) I^\mot(\overrightarrow{1}_{0}; -1; 1) - I^\mot(\overrightarrow{1}_{0}; -1, 0; 1) \\
	&= 0 \cdot \log^\mot(2) + \zeta^\mot(\overline{2}) \,,
	\end{align*}
	where \( \overrightarrow{1}_0 \) denotes the tangential base-point at 0 with tangent vector in the direction \( \overrightarrow{1} \).  So there is a difference of \( (\ii \pi)^\mot \log^\mot(2) \) between the homotheties.  However, one can say that in general (for \( a_i \in \{0,\pm1\} \)), with weight \( w > 1 \), that 
	\[
	I^\lmot(\lambda a_0; \lambda a_1, \ldots, \lambda a_n; \lambda a_{n+1}) = I^\lmot(a_0; a_1,\ldots,a_n; a_{n+1}) \,.
	\]
	In weight 1, homotheties by \( \lambda \) (with \( \abs{\lambda} = 1 \)) will rotate the tangential base-point, and so contribute some rational times \( (\ii \pi)^\lmot = 0 \) (already \( (\ii\pi)^\umot = 0 \), even before killing products) by considering the decomposition of paths \( I^\lmot(\overrightarrow{\lambda}_{0}, 0, a) = I^\lmot(\overrightarrow{\lambda}_{0}, 0, \overrightarrow{1}_{0}) + I^\lmot(\overrightarrow{1}_{0}, 0, a)\).  So the homothety property still holds. \medskip
	
	This property is applied in various places in Murakami's proof of Proposition 11 (and some earlier results upon which it is dependent).  But in every case, it is applied to the \( I^\lmot \) part of the coaction, and so is valid.
	\end{Rem}
	
	We now make some observations which simplify the calculation of \( D_{r} \) when \( r = 1 \).  Note that \( D_{r} \), for \( r > 1 \) is computed with exactly the same formula as Murakami \cite[Proposition 11]{murakami21}, we have merely extended the range of validity to all (shuffle regularised) multiple \( t \) values.  So let us focus on the case \( D_1 \).
	
	\begin{Prop}[Calculation of \( D_1 \)]\label{prop:d1}
		Let \( \vec{k} = (k_1,\ldots,k_d) \in (\Z_{\geq1})^d\) be an index.  Then
		\[
			D_1 \ttw^\mot(k_1,\ldots,k_d) = \begin{aligned}[t]
				&\delta_{k_1 = 1} \cdot 2 \log^\lmot(2) \otimes \ttw^\mot(k_2,\ldots,k_d) \\
				& - \delta_{k_d = 1} \log^\lmot(2) \otimes \ttw^\mot(k_1,\ldots,k_{d-1}) \,. \end{aligned}
		\]
		That is \( D_1 \) acts by deconcatenation of trailing 1's and of leading 1's (with coefficient 2).
		
		\begin{proof}
			Assuming \( r = 1 \), we consider how terms \eqref{eqn:dr:0eps} and \eqref{eqn:dr:eps0} can contribute.  For \( \eqref{eqn:dr:0eps} \), the delta condition requires \( \abs{\vec{k}_{i+1,j}} \leq 1 < \abs{\vec{k}_{i,j}} - 1 \).  The first condition forces \( i+1 = j \), and \( \vec{k}_{i+1} = 1 \), so \( \vec{k}_{i,j} = (\alpha, 1) \), for some \( \alpha > 1 \).  One has that \( \abs{\vec{k}_{i,j}} - r = \alpha + 1 - 1 = \alpha \).  In this case, we contribute
			\[
				(\zeta_0(1) - \log^\lmot(2)) \otimes \ttw^2(\vec{k}_{1,i-1}, \alpha, \vec{k}_{j+1,d})  \,,
			\]
			which can be seen as deleting the 1 following \( \alpha = k_i \) in \( \vec{k}_{i,j} = (\alpha, 1) \).
			
			Likewise, for \eqref{eqn:dr:eps0}, the delta condition requires \( \abs{\vec{k}_{i,j-1}} \leq 1 < \abs{\vec{k}_{i,j}} - 1 \).  The first condition forces \( i = j-1 \), and \( \vec{k}_{i} = 1 \), so \( \vec{k}_{i,j} = (1,\alpha) \), for some \( \alpha > 1 \).  One has that \( \abs{\vec{k}_{i,j}} - r = \alpha + 1 - 1 = \alpha \).  In this case, we contribute
			\[
			-(\zeta_0(1) - \log^\lmot(2)) \otimes \ttw^2(\vec{k}_{1,i-1}, \alpha, \vec{k}_{j+1,d})  \,,
			\]
			which can be seen as deleting the 1 proceeding \( \alpha = k_{i+1} \) in \( \vec{k}_{i,j} = (1,\alpha) \).
			
			This means that for any subindex \( (\alpha, \{1\}^n, \beta) \), \( \alpha,\beta> 1\) appearing in \( \vec{k} \), the term from deleting the 1 after \( \alpha \) cancels with the term from deleting the \( 1 \) before \( \beta \).  The only terms which can survive this process are of the form \( (\{1\}^n, \beta) \) at the start of \( \vec{k} \), and \( (\alpha,\{1\}) \) at the end of \( \vec{k} \).  Combined with the pre-existing deconcatenation term \eqref{eqn:dr:deconcat}, we obtain the claimed expression.
		\end{proof}
	\end{Prop}

	\begin{Rem}[Hoffman's derivation with respect to \( \log(2) \)]\label{rem:diff:t}
		In \cite[Conjecture 2.1]{hoffman19}, Hoffman conjectures that the algebra of MtV's admits a derivation \( \mathrm{d} \) which acts on \( t(k_1, \ldots, k_d) \) by \[
		 \mathrm{d} t(k_1,\ldots,k_d) = \begin{cases}
		 	t(k_2,\ldots,k_d) & \text{if \( k_1 = 1 \)} \\ 
		 	0 & \text{otherwise\,.}
		 \end{cases}
		\]
		
		If \( \vec{k} = (k_1,\ldots, k_d) \), with \( k_d \neq 1 \), so that \( \ttw^\mot(\vec{k}) \) is a convergent motivic MtV, one obtains the formula
		\[
			D_1 \ttw^\mot(\vec{k}) = \delta_{k_1 = 1} \log^\lmot(2) \otimes \ttw^\mot(k_2, \ldots, k_d) \,.
		\]
		Since \( D_1 \) acts as derivation in the sense
		\[
			D_1 (X Y) = (1 \otimes X) D_1 Y + (1 \otimes Y) D_1 X \,,
		\]
		after projecting \( \log^\lmot(2) \mapsto 1 \) we see that Hoffman's conjectural derivation is nothing but the action of \( D_1 \) on the motivic MtV's, in the convergent case. \medskip
		
		Moreover, one also notes that for \( \ell_1,\ldots\ell_f, n \in \Z \), \( n \geq 2 \), we have
		\[
			D_1 \zeta^\mot(\ell_1,\ldots,\ell_f, \overline{n}) = 0 \,.
		\]
		This is because the strings \( \{0, 1, -1\}, \{0, -1, 1\}, \{-1, 1, 0\}, \{ 1, -1, 0\} \) do not occur in the integral representation of the MZV.  These strings lead to \( \log^\lmot(2) \) factors, whereas \( I^\lmot(1; 0; -1) = I^\lmot(-1; 0; 1) = 0 \).  The algebra basis of the MZV Data Mine \cite{mzvDM} exclusively invokes alternating MZV's of the form
		\[
			\zeta(\ell_1,\ldots,\ell_f, \overline{n}) \,,
		\]
		with \( \ell_i, n \) odd.  So one sees that Hoffman's claim, with regard to the action of \( d \) as differentiation wrt \( \log(2) \) on the formulae in Appendix A of \cite{hoffman19} is generally valid, for this specific choice of basis.
		
		For example, the following identity is verified by the Data Mine, and so actually holds on the motivic level
		\begin{align*}
			t^\mot(1,3,2) = {} &-\frac{2}{21} t^\mot(6) - \frac{3}{196} t^\mot(3)^3 - \frac{1}{2} t^\mot(2) \zeta^\mot(1, \overline{3}) + \frac{1}{4} \zeta^\mot(1, \overline{5}) \\
			& - \frac{1}{2} t^\mot(5) \log^\mot(2) + \frac{4}{7} t^\mot(2)t^\mot(3) \log^\mot(2) \,.
		\end{align*}
		Application of \( D_1 \) (after scaling to write it via \( \ttw\), so \autoref{prop:d1} can be applied, and rescaling afterwards) leads to
		\[
			\log^\lmot(2) \otimes t^\mot(3,2) = \log^\lmot(2) \otimes \bigg( -\frac{1}{2} t^\mot(5) + \frac{4}{7} t^\mot(2) t^\mot(3) \bigg) \,,
		\]
		or equivalently
		\[
			t^\mot(3,2) = -\frac{1}{2} t^\mot(5) + \frac{4}{7} t^\mot(2) t^\mot(3) \,,
		\]
		as expected from Hoffman's claim.
	\end{Rem}

	\begin{Rem}
	The formula for \( D_1 \ttw^\mot \) in \autoref{prop:d1} shows immediately that the convergent \( \ttw^\mot(1, \vec{k}) \) cannot be a motivic MZV, as \( D_1 \ttw^\mot(1,\vec{k}) = \log^\lmot(2) \otimes \ttw^\mot(\vec{k}) \neq 0 \).  On the other hand, this gives us a place and means to search for other Galois descent candidates.
\end{Rem}

	\begin{Prop}\label{prop:t21232}
		Let \( a,b,c,n \in \Z_{\geq0} \), such that \( a \geq 1 \) and \( n \geq 1 \).  Then the motivic multiple \( t \) value
		\[
			\tau = \ttw^\mot(\{2\}^a, 1, \{2\}^b, 2n+1, \{2\}^c)
		\]
		is always a (linear combination of) motivic MZV's.
		
		\begin{proof}
			From the above remark, we know \( D_1 \tau = 0 \).  We must only check the second part of Glanois's motivic Galois descent criterion from \autoref{thm:glanois:descent}, namely that \( D_{2r+1} \tau \in \mathcal{L}^{(1)}_{2r+1} \otimes \mathcal{H}^{(1)} \), i.e. the parts of the coaction are already motivic MZV's.
			
			We first note that the deconcatenation term \eqref{eqn:dr:deconcat} takes the form
			\[
				\ttw^\lmot(\{2\}^a, 1, \{2\}^{r - a}) \otimes \ttw^\mot(\{2\}^{b-(r-a)}, 2n+1, \{2\}^c)
			\]
			The left hand factor is (modulo products!) a motivic MZV by \autoref{thm:mott2212} below.  (The terms \( \log^\mot(2) \) only appear as products in weight \( > 1 \), so vanish when we project to \( \mathcal{L} \).)  The right hand factor is a motivic MZV by Theorem 8 in \cite{murakami21}; therein Murakami showed that whenever \( \vec{k} \in (\Z_{\geq2})^d \) is an index with all entries \( \geq 2 \), then \( \ttw^\mot(\vec{k}) \in \mathcal{H}^{(1)} \) is a motivic MZV.
			
			Then for the terms \eqref{eqn:dr:0eps} and \eqref{eqn:dr:eps0}, one only needs to consider the right hand factor, as the left hand one is already an MZV.  One can also assume \( \vec{k}_{i,j} \) does not contain \( 1 \), for if it does contain 1, then the condition \( r < \abs{\vec{k}_{i,j}} - 1 \) means that \( \abs{\vec{k}_{i,j}} - r > 1 \), so that the subindex we removed is replaced with \( \geq 2 \).  Hence by Theorem 8 \cite{murakami21} is already a motivic MZV.  More generally, we note that in \eqref{eqn:dr:0eps} by subtracting the delta condition from \( \abs{\vec{k}_{i,j}} \), one has
			\[
				k_i \geq \abs{\vec{k}_{i,j}} - r > 1
			\]
			So the replacement value \( \abs{\vec{k}_{i,j}} - r \) for the entire subindex \( \vec{k}_{i,j} \) is between \( 2 \) and \( k_i \), the left endpoint.  Likewise in \eqref{eqn:dr:eps0}, the replacement is between \( 2 \) and \( k_j \), the right endpoint.
			
			We have the following subindices which exhaust all remaining possible cases.   The subindex \( D \) here may start or end at \( 2n+1,  \).
			\begin{center}
				\begin{tikzpicture}[baseline]
				\matrix[name=M1, matrix of nodes, inner sep=1pt, column sep=0pt]{
					\node () [] {$\ttw^\mot($};
					& \node (a1) [] {${\phantom{\mathllap{\fbox{$,0$}}}} 2$}; & \node () [] {$,$};
					& \node (a2) [] {${\phantom{\mathllap{\fbox{$,0$}}}}\ldots$};& \node () [] {$,$};
					& \node (a3) [] {${\phantom{\mathllap{\fbox{$,0$}}}} 2$}; & \node () [] {$,$};
					& \node () [] {${\phantom{\mathllap{\fbox{$,0$}}}}  1$};  			& \node () [] {$,$};
					& \node (b1) [] {${\phantom{\mathllap{\fbox{$,0$}}}}  2$};  			& \node () [] {$,$};
					& \node (b2) [] {${\phantom{\mathllap{\fbox{$,0$}}}}  \ldots$};  			& \node () [] {$,$};
					& \node (b3) [] {${\phantom{\mathllap{\fbox{$,0$}}}}  2$};  & \node () [] {$,$};
					& \node (x) [] {${\phantom{\mathllap{\fbox{$,0$}}}}  2n+1$};  			& \node () [] {$,$};
					& \node (c1) [] {${\phantom{\mathllap{\fbox{$,0$}}}}  2$};  			& \node () [] {$,$};
					& \node (c2) [] {${\phantom{\mathllap{\fbox{$,0$}}}}  \ldots$};  			& \node () [] {$,$};
					& \node (c3) [] {${\phantom{\mathllap{\fbox{$,0$}}}}  2$};  
					& \node () [] {$) \,.$};
					\\
				};
				\draw[thick,black] (a1.north) to[bracket=12pt] node[above=-4mm] {$_A$} (a3.north);
				\draw[thick,black] (b1.north) to[bracket=12pt] node[above=-4mm] {$_B$} (b3.north);
				\draw[thick,black] (c1.north) to[bracket=12pt] node[above=-4mm] {$_C$} (c3.north);
				\draw[thick,black] (b2.south) to[ubracket=12pt] node[above=-4mm] {$_D$} (c2.south);
				\end{tikzpicture}
			\end{center}
			We already note though that \( A, B  \) and \( C \) cannot in fact contribute.  The replacement value \( k_i = k_j = 2 = \abs{\vec{k}_{i,j}} - r > 1 \) must be 2.  But this implies \( \vec{k}_{i,j} = r + 2 \) is odd.  So we are left with the case \( D \), and for the same reason the replacement must be even in this case, namely:
			\begin{center}
			\begin{tabular}{c|c|c|c}
				Subindex & $\vec{k}_{i,j}$ & $\abs{\vec{k}_{i,j}}-r$ & $\ttw^\mot(\vec{k}_{1,i-1}, \abs{\vec{k}}_{i,j} - r, \vec{k}_{j+1,d}) $ \\ \hline
				D & $(\{2\}^\alpha,2n+1,\{2\}^\beta) $ & 2 & $\ttw^\mot(\{2\}^a, 1, \{2\}^\gamma, 2, \{2\}^\delta)$ 
			\end{tabular}
		\end{center}
		Since this MtV is of the form \( \ttw^\mot(\{2\}^a, 1, \{2\}^b) \) with \(a, b > 0 \), it is a motivic multiple zeta value via \autoref{thm:mott2212} below.
		\end{proof}
	\end{Prop}

	It would be interesting to see how far this proof can be generalised, and whether one can give some complete combinatorial criterion for when \( \ttw^\mot(\vec{k}) \) descends to a motivic MZV.  Certainly other families of motivic MtV's which descend seem to exist.  A promising candidate is as follows: let \( a, n \in \Z_{\geq1} \), and \( \vec{k},\vec{\ell} \) be indices containing only even entries.  Then it appears that the following MtV, a generalisation of the above, is also a motivic multiple zeta value.
	\[
		\ttw^\mot(\{2\}^a, 1, \vec{k}, 2n+1, \vec{\ell}) \overset{?}{\in} \mathcal{H}^{(1)} \,.
	\]
	There are also indices with multiple 1's that Galois descend, such as
	\[
		\ttw^\mot(\{2,1,3\}^2) \in \mathcal{H}^{(1)} \,,
	\]
	although this pattern does not seem to continue.  Once can check (via the MZV Data Mine \cite{mzvDM}) that \( D_7 \ttw^\mot(\{2,1,3\}^3) \notin \mathcal{L}_{7}^{(1)} \otimes \mathcal{H}_{11}^{(1)} \).

	\section{\texorpdfstring{Lift to a motivic \( \ttw^{\mot}(\{2\}^a, 1, \{2\}^b) \) evaluation}
		{Lift to a motivic t\textasciitilde\textasciicircum{}m(\{2\}\textasciicircum{}a, 1, \{2\}\textasciicircum{}b) evaluation}}
	\label{sec:mott2212}
	
	The aim of this section is to first lift the evaluation for \( t^{\shuffle,W=0}(\{2\}^a,1,\{2\}^b) \) given in \autoref{thm:sht2212ev} (more precisely, the explicit version given in \eqref{eqn:sht2212ev} thereafter) to an identity amongst motivic multiple \( t \) values.
	
	The shuffle version, with regularisation parameter \( W = 0 \) is the key identity for the rest of this work, since the motivic MZV's are naturally and almost-always regularised in this manner.  Moreover, via \autoref{prop:tstVtotshandz111}, we can express the stuffle regularisation at arbitrary parameters \( t^{\ast,V}(1) = V \) via the shuffle regularised version at \( t^{\shuffle,W=0} \).  This will be used to sidestep later the issue of how to take the motivic stuffle regularisation.
	
	\begin{Thm}\label{thm:mott2212}
		The following motivic identity holds for all \( a, b \geq 0 \)
		\begin{equation}
		\label{eqn:mott2212ev}
		\begin{aligned}
		\ttw^{\mot}(\{2\}^a,1,\{2\}^b) = {} & \sum_{r=1}^{a+b} (-1)^{r+1} \cdot 2 \bigg[ \binom{2r}{2a} + \frac{2^{2r}}{2^{2r} - 1} \binom{2r}{2b} \bigg] \zeta^\mot(\overline{2r+1}) \ttw^\mot(\{2\}^{a + b - r}) \\
		& {} + \delta_{a=0} 2 \cdot \log^\mot(2) \ttw^\mot(\{2\}^b) - \delta_{b=0} \log^\mot(2) \ttw^\mot(\{2\}^a) \,,
		\end{aligned}
		\end{equation}
	\end{Thm}

	Before we begin, it will be useful for later purposes to recall the motivic identity proven in \cite{murakami21} for \( \ttw^\mot(\{2\}^a, 3, \{2\}^b) \).  This also gives us an opportunity to compare and contrast the two evaluations, which in the MZV case would be equal by duality.

	\begin{Thm}[Murakami, {\cite[Theorem 22]{murakami21}}]\label{thm:mott2232}
		The following motivic identity holds for all \( a,b \geq 0 \)
		\begin{equation}
		\label{eqn:mott2232ev}
		\begin{aligned}
		\ttw^{\mot}(\{2\}^a,3,\{2\}^b) = {} & \!\! \sum_{r=1}^{a+b+1} \! (-1)^{r+1} \cdot 2 \bigg[ \binom{2r}{2a+1} + (1 - 2^{-2r}) \binom{2r}{2b+1} \bigg] \begin{aligned}[t] & \zeta^\mot({2r+1}) \\
		& \quad \cdot \ttw^\mot(\{2\}^{a + b + 1 - r}) \hspace{-1em}
		\end{aligned}
		\end{aligned}
		\end{equation}
	\end{Thm}
	
	\subsection{Proof of \autoref{thm:mott2212}}
	
	After application of the period map, the identity in the theorem reduces to (\( 2^{2a+2b+1} \) times) the identity in  \eqref{eqn:sht2212ev}, with \( W = 0 \).  Therefore we only need to verify that 
	\[
	 D_{2r+1} \big( \text{LHS } \eqref{eqn:mott2212ev} \big)
	 - D_{2r+1} \big( \text{RHS }  \eqref{eqn:mott2212ev} \big) = 0
	\]
	This will show that the purported identity lies in the kernel of \( D_{<N} \), \( N = 2a + 2b + 1 \).  Hence by Glanois's Theorem (\autoref{thm:glanois:kerDN}) it holds up to an additive constant \( c \zeta^\mot(\overline{N}) \), and application of the period map shows that \( c = 0 \).  This will verify that the identity holds on the motivic level, as claimed. \medskip
	
	Write \begin{align*}
		L^{a,b} = {} & \ttw^{\mot}(\{2\}^a,1,\{2\}^b) \\[1ex]
		R^{a,b} = {} & -\sum_{r=1}^{a+b} (-1)^r \cdot 2 \bigg[ \binom{2r}{2a} + \frac{2^{2r}}{2^{2r} - 1} \binom{2r}{2b} \bigg] \zeta^\mot(\overline{2r+1}) \ttw^\mot(\{2\}^{a + b - r}) \\
		& {} + \delta_{a=0} 2 \cdot \log^\mot(2) \ttw^\mot(\{2\}^b) - \delta_{b=0} \log^\mot(2) \ttw^\mot(\{2\}^a) \,,
	\end{align*}
	for the left and right hand side of the purported identity.  In order to check 
	\[
		D_{2r+1} \big( L^{a,b} \big)
		- D_{2r+1} \big( R^{a,b} \big) = 0 \,,
	\]
	we proceed inductively on \( a + b \).  
	
	In the case \( a = b = 0 \), we immediately find
	\[
		L^{0,0} = \ttw^\mot(1) = \zeta^\mot(1) - \zeta^\mot(\overline{1}) = \log^\mot(2) = R^{0,0}
	\]
	So we may take \( a+b > 0 \).
	
	\begin{Lem}
		The following expression for \( D_{2r+1} L^{a,b} \) holds for any \( a, b \geq 0 \), and \( 0 \leq r \leq a + b \),
		\[
			D_{2r+1} L^{a,b} = D_{2r+1} \ttw^\mot(\{2\}^a,1,\{2\}^b) = \pi(\widehat{\xi}_{a,b}^r) \otimes \ttw^\mot(\{2\}^{a+b-r}) \,,
		\]
		where \( \pi \colon \mathcal{H}^2 \to \mathcal{L}^2 \) denotes the projection, and \( \widehat{\xi}_{a,b}^r \) is given by (the sums running over all indices \( \alpha,\beta \geq 0 \) satisfying \( \alpha + \beta = r \))
		\begin{align*}
			\widehat{\xi}_{a,b}^r = {} 	& \delta_{r=0} \delta_{a=0} \log^\mot(2) - \delta_{r=0} \delta_{b=0} \log^\mot(2)  + \delta_{a\leq r}  \ttw^\mot(\{2\}^a, 1, \{2\}^{r-a}) \\
			& + \sum_{\substack{\alpha \leq a-1 \\ \beta \leq b}} \zeta_0^\mot(\{2\}^\alpha, 1, \{2\}^\beta) - \sum_{\substack{\alpha \leq a \\ \beta \leq b-1}} \zeta_0^\mot(\{2\}^\beta, 1, \{2\}^\alpha)
		\,,
		\end{align*}
		
		\begin{proof}
			This is a direct, if somewhat tedious, calculation which follows from \autoref{prop:dkt}
		\end{proof}
	\end{Lem}

	Now introduce the following notation from \cite{brown12}
	\[
		A_{a,b}^r = \binom{2r}{2a+2} \,, B_{a,b}^r = (1 - 2^{-2r}) \binom{2r}{2b+1} \,,
	\]
	and recall one of the main results proved therein.
	
	\begin{Thm}[Brown {\cite[Theorem 4.3]{brown12}}]\label{thm:brown:z2232}
		For all \( a, b \geq \), the following identity amongst motivic MZV's holds
		\[
			\zeta^\mot(\{2\}^a, 3, \{2\}^b) = \sum_{r=1}^{a+b+1} (-1)^{r-1} \cdot \big( {-}A_{a,b}^r + B_{a,b}^r \big) \zeta^\mot(2r+1) \zeta^\mot(\{2\}^{a+b+1-r}) \,.
		\]
	\end{Thm}

	We therefore have for \( \alpha \geq 0, \beta > 0 \) that
	\begin{equation}\label{eqn:coeffzl2232}
	\begin{aligned}
		\zeta_0^\lmot(\{2\}^\alpha, 1, \{2\}^\beta) & {} = \zeta_0^\lmot(\{2\}^{\beta-1}, 3, \{2\}^\alpha) = 2(-1)^{\alpha+\beta} \big( A_{\beta-1,\alpha}^{\alpha+\beta} - B_{\beta-1,\alpha}^{\alpha+\beta} \big) \zeta^\lmot(2\alpha + 2\beta + 1) \,, \\
		\zeta_0^\lmot(\{2\}^\alpha, 1) & {} = \zeta_1^\lmot(\{2\}^\alpha) = 2 (-1)^\alpha \zeta^\lmot(2\alpha+1) \,.
	\end{aligned}
	\end{equation}
	The first follows by duality and  extracting the coefficient of \( \zeta^\mot(2a+2b+3) \) in \autoref{thm:brown:z2232}.  The second follows by shuffle regularising, or from the stuffle product, as shown in Lemma 3.8 \cite{brown12}.
	
	By the induction assumption, we are also granted \( \ttw^\lmot(1) = \log^\lmot(2) \), and that for \( 0 < a' + b' < a + b \) we have
	\begin{equation}\label{eqn:coefftl2212}
		\ttw^\lmot(\{2\}^{a'}, 1, \{2\}^{b'}) = 2(-1)^{1+a' + b'} \bigg[ \binom{2a'+2b'}{2a'} + \frac{2^{2(a'+b')}}{2^{2(a'+b')}-1} \binom{2a'+2b'}{2b'} \bigg] \zeta^\lmot(\overline{2a'+2b'+1}) \,.
	\end{equation}
	
	\paragraph{\bf Case \( D_1 \):} We check explicitly and directly the case \( r = 0 \), because it can have a distinctly different form, on account of the \( \delta_{r=0} \log^\lmot(2) \) terms.  Explicit we find (also directly from \autoref{prop:d1})
	\begin{align*}
		\widehat{\xi}_{a,b}^0 = {} & \delta_{a=0} \log^\mot(2) - \delta_{b=0} \log^\mot(2)  + \delta_{a\leq 0}  \ttw^\mot(\{2\}^a, 1, \{2\}^{-a}) \\
		& + \sum_{\substack{\alpha \leq a-1 \\ \beta \leq b \\ \alpha + \beta = 0}} \zeta_0^\mot(\{2\}^\alpha, 1, \{2\}^\beta) - \sum_{\substack{\alpha \leq a \\ \beta \leq b-1 \\ \alpha + \beta = 0}} \zeta_0^\mot(\{2\}^\beta, 1, \{2\}^\alpha) \\[1ex]
		= {}  & \delta_{a=0} \log^\mot(2) - \delta_{b=0} \log^\mot(2)  + \delta_{a\leq 0} \ttw^\mot(\{2\}^0, 1, \{2\}^0) \\
		& +  \zeta_0^\mot(\{2\}^0, 1, \{2\}^0) - \zeta_0^\mot(\{2\}^0, 1, \{2\}^0) \\[1ex]
		= {} & \delta_{a=0} 2 \cdot \log^\mot(2) - \delta_{b=0} \log^\mot(2) 
	\end{align*}
	So that since \( D_1 = D_{2\cdot0+1} \) with \( r = 0 \), we have
	\[
		D_{1} L^{a,b} = \big( \delta_{a=0} 2 \cdot \log^\mot(2) - \delta_{b=0} \log^\mot(2) \big) \otimes \ttw^\mot(\{2\}^{a+b})
	\]
	
	Whereas, directly from \( R^{a,b} \), we can compute the following.  We make use of some simple properties of \( D_{2r+1} \), such as the derivation and that \( \zeta^\mot(N) \) and \( \zeta^\mot(\overline{N}) \) are primitive for the coaction, viz.: \( \Delta \zeta^\mot(N) = 1 \otimes \zeta^\mot(N) + \zeta^\lmot(N) \otimes 1 \).  Overall this means
	\begin{align*}
		 & D_{2r+1} XY = (1\otimes Y)D_{2r+1} X + (1\otimes X) D_{2r+1} Y \,,  \\
		 & D_{2r+1} \zeta^\mot(N) = \begin{cases} 0 \,, & \text{if \( 2r + 1 \neq N \)} \\
		 		\zeta^\lmot(N) \,, & \text{if \( 2r + 1 = N \)} \,,
			\end{cases}
	\end{align*}
	the latter also for \( N \) replaced by \( \overline{N} \), in particular also for \( \zeta^\mot(\overline{1}) = -\log^\mot(2) \).  Applying these to the computation of \( D_{1} R^{a,b} \) gives the following
	\begin{align*}
	D_{1} R^{a,b} = {} & -\sum_{r'=1}^{a+b} (-1)^{r'} \cdot 2 \bigg[ \binom{2r'}{2a} + \frac{2^{2r'}}{2^{2r'} - 1} \binom{2r'}{2b} \bigg] \begin{aligned}[t]
	\Big( & (1 \otimes \zeta^\mot(\overline{2r'+1})) \overbrace{D_{1} \ttw^\mot(\{2\}^{a + b - r'})}^{=0}\\
	& + \underbrace{D_{1} \zeta^\mot(\overline{2r'+1})}_{=\delta_{2r'+1=1}} (1 \otimes \ttw^\mot(\{2\}^{a + b - r'})) \end{aligned} \\
& {} + \delta_{a=0} 2 \cdot \Big( D_1 \log^\mot(2) (1 \otimes \ttw^\mot(\{2\}^b) +  (1 \otimes \log^\mot(2)) \underbrace{D_1 \ttw^\mot(\{2\}^b)}_{=0} \Big) \\
& {} - \delta_{b=0} \Big( D_1 \log^\mot(2) (1 \otimes \ttw^\mot(\{2\}^a)) + (1 \otimes \log^\mot(2)) \underbrace{(D_1 \ttw^\mot(\{2\}^a))}_{=0} \Big) \,,
	\end{align*}
	So all terms vanish apart from the two terms involving \( D_1 \log^\mot(2) \), which leads to
	\begin{align*}
		D_1 R^{a,b} & {} = \delta_{a=0} 2 \cdot \log^\lmot(2) \otimes \ttw^\mot(\{2\}^b) - \delta_{b=0} \log^\lmot(2) \otimes \ttw^\mot(\{2\}^a)  \\
			& {} = \big( \delta_{a=0} 2 \cdot \log^\lmot(2) - \delta_{b=0} \log^\lmot(2) \big) \otimes \ttw^\mot(\{2\}^{a+b}) \,.
	\end{align*}
	The last simplification holds because \( R^{a,b} \) has total weight \( 2a + 2b + 1 \), so the right hand tensor factor of \( D_{1} \) must have weight \( 2a + 2b \), irrespective of checking the various cases of the Kronecker delta conditions. \medskip
	
	In particular, we have that \( D_1 L^{a,b} = D_1 R^{a,b} \) in this case. \medskip
	
	\paragraph{\bf Case \( r > 0 \):}  Now we turn to the case \( r > 0 \), which will have no extra \( \log^\lmot(2) \) contribution.  We find it helpful to separate out the terms where \( \beta = 0 \) or \( \beta > 0 \) in the sum involving \( \zeta_0^\lmot(\{2\}^\alpha, 1, \{2\}^\beta) \), and similarly for the one involving \( \zeta_0^\lmot(\{2\}^\beta, 1, \{2\}^\alpha) \).  This is on account of the different form of the coefficient of \( \zeta^\lmot(2r + 1 ) \) therein might take.  We have that
	\begin{align*}
		\widehat{\xi}_{a,b}^r = {} 	& \delta_{a\leq r}  \ttw^\mot(\{2\}^a, 1, \{2\}^{r-a}) + \sum_{\substack{\alpha \leq a-1 \\ 1 \leq \beta \leq b}} \zeta_0^\mot(\{2\}^\alpha, 1, \{2\}^\beta) - \sum_{\substack{1 \leq \alpha \leq a \\ \beta \leq b-1}} \zeta_0^\mot(\{2\}^\beta, 1, \{2\}^\alpha) \\
		& {} + \sum_{\substack{\alpha \leq a-1 \\ \beta = 0}} \zeta_0^\mot(\{2\}^\alpha, 1, \{2\}^\beta) - \sum_{\substack{\alpha = 0 \\ \beta \leq b-1}} \zeta_0^\mot(\{2\}^\beta, 1, \{2\}^\alpha) 
	\end{align*}
	Since we sum over \( \alpha + \beta = r \), the last two summations resolve to a Kronecker delta condition, namely \( \delta_{r \leq a-1} \) and \( \delta_{r \leq b-1} \) respectively.  Making the substitutions for the various \( \zeta_0^\lmot \) using \eqref{eqn:coeffzl2232} and for \( \ttw^\lmot \) from \eqref{eqn:coefftl2212} by induction, we find
	\begin{align*}
	\pi(\widehat{\xi}_{a,b}^r) = {} 	
	& \delta_{a\leq r} 2 \cdot (-1)^{r+1} \bigg( \binom{2r}{2a} + \frac{2^{2r}}{2^{2r}-1} \binom{2r}{2r-2a} \bigg) \zeta^\lmot(\overline{2r+1}) \\
	 & \raisebox{-1ex}{${} + \Bigg\{$} \begin{aligned}[t] 
	 & \sum_{\substack{\alpha \leq a-1 \\ 1 \leq \beta \leq b}} 2(-1)^{r} \big( A_{\beta-1,\alpha}^r - B_{\beta-1,\alpha}^r \big) - \sum_{\substack{1 \leq \alpha \leq a \\ \beta \leq b-1}} 2(-1)^{r} \big( A_{\alpha-1,\beta}^r - B_{\alpha-1,\beta}^r \big) \\
	& {} + \delta_{r \leq a-1} 2(-1)^r  - \delta_{r \leq b-1} 2(-1)^r  \bigg\} \zeta^\lmot(2r+1) \end{aligned}
	\end{align*}
	If we make the change of variables \( \beta = \beta' + 1 \), in the first sum, and \( \alpha = \alpha' + 1 \) in the second sum, then the (implicit) summation range \( \alpha + \beta = r\) is converted to \( \alpha + \beta' + 1 = r \) and \( \alpha' + 1 + \beta = r \) respectively.  (We shall write this explicitly from now on.)  Doing so, and simplifying the expression coming from \( \ttw^\lmot \) with \( \zeta^\mot(\overline{2r+1}) = -(1 - 2^{-2r}) \zeta^\mot(2r+1) \), gives
	\begin{align*}
	\pi(\widehat{\xi}_{a,b}^r) = {} 	
	& 2 (-1)^{r} \bigg\{  \begin{aligned}[t] 
	&  \delta_{a\leq r} \big( 2 - 2^{-2r} \big) \binom{2r}{2a} {} + \delta_{r \leq a-1}  - \delta_{r \leq b-1} \\
	& {} + \!\! \sum_{\substack{\alpha \leq a-1 \,, \beta \leq b-1 \\ \alpha + \beta = r-1}} \!\! \big( A_{\beta,\alpha}^r - B_{\beta,\alpha}^r \big) - \!\! \sum_{\substack{\alpha \leq a-1 \,, \beta \leq b-1 \\ \alpha + \beta = r-1}} \!\! \big( A_{\alpha,\beta}^r - B_{\alpha,\beta}^r \big) 
	 \bigg\} \zeta^\lmot(2r+1) \end{aligned}
	\end{align*}
	Now we may apply the following Lemma
	\begin{Lem}[Brown, {\cite[Lemma 4.2]{brown12}}]
		For any \( a, b \geq 0 \), and \( 1 \leq r \leq a+b+1 \) we have
		\begin{align*}
			& \sum_{\substack{\alpha < a, \beta \leq b \\ \alpha + \beta + 1 = r}} A_{\alpha,\beta}^r - \sum_{\substack{\alpha \leq a \,, \beta < b \\ \alpha + \beta + 1 = r}} A_{\beta,\alpha} + \delta_{b\geq r} - \delta_{a \geq r} = 0 \\
			& \sum_{\substack{\alpha \leq a \,, \beta \leq b \\ \alpha + \beta + 1 = r}} B_{\alpha,\beta}^r - \sum_{\substack{\alpha \leq a \,, \beta < b \\ \alpha + \beta + 1 = r}} B_{\beta, \alpha} = B_{a,b}^r
		\end{align*}
	\end{Lem}	

	In the case \( a-1, b - 1 \), we may still apply the lemma whenever \( r \leq (a-1) + (b-1) + 1 = a + b - 1 \).  But we are computing \( D_{2r+1} \) with \( 3 \leq 2r + 1 < 2a + 2b + 1 \), i.e. \( 3 \leq 2r + 1 \leq 2a + 2b - 1 \) or equivalently \( 1 \leq r \leq a + b - 1 \).  Application of this lemma (taking care with the range of summation indices, some are \( < \) while others are \( \leq \)) gives
	\begin{align*}
	\pi(\widehat{\xi}_{a,b}^r) = {} 	
	& 2 (-1)^{r} \bigg\{  \begin{aligned}[t] 
	&  \delta_{a\leq r} \cdot \big( 2 - 2^{-2r} \big) \binom{2r}{2a} {} + \delta_{r \leq a-1}  - \delta_{r \leq b-1} \\
	& {} - \delta_{a \leq r} A_{a-1,r -a}^r  + \delta_{b \leq r} A_{b-1,r-b}^r - \!\! \sum_{\substack{\alpha < a-1 \,, \beta \leq b-1 \\ \alpha + \beta = r-1}} \!\! A_{\alpha,\beta}^r  
	 + \!\! \sum_{\substack{\alpha \leq a-1 \,, \beta < b-1 \\ \alpha + \beta = r-1}} \!\!  A_{\beta,\alpha}^r   \\
	& {} - \delta_{b \leq r} B_{b-1, r-b}^r + \!\! \sum_{\substack{\alpha \leq a-1 \,, \beta \leq b-1 \\ \alpha + \beta = r-1}} \!\!  B_{\alpha,\beta}^r 
	- \!\! \sum_{\substack{\alpha \leq a-1 \,, \beta < b-1 \\ \alpha + \beta = r-1}} \!\!  B_{\beta,\alpha}^r  
	\bigg\} \zeta^\lmot(2r+1) \end{aligned} \\[1ex]
	= {} 	
	& 2 (-1)^{r} \bigg\{  \begin{aligned}[t] 
	&  \delta_{a\leq r} \cdot \big( 2 - 2^{-2r} \big) \binom{2r}{2a} {} + \delta_{r \leq a-1}  - \delta_{r \leq b-1} \\
	& {} - \delta_{a \leq r} A_{a-1,r -a}^r  + \delta_{b \leq r} A_{b-1,r-b}^r + \delta_{b-1 \geq r} - \delta_{a-1 \geq r}   \\
	& {} - \delta_{b \leq r} B_{b-1, r-b}^r + B_{a-1,b-1}^r
	\bigg\} \zeta^\lmot(2r+1) \end{aligned} \\[1ex]
	= {} 	
	& 2 (-1)^{r} \bigg\{  \begin{aligned}[t] 
	&  \delta_{a\leq r} \cdot \big( 2 - 2^{-2r} \big) \binom{2r}{2a} {} - \delta_{a \leq r} A_{a-1,r -a}^r  + \delta_{b \leq r} A_{b-1,r-b}^r  \\
	& {} - \delta_{b \leq r} B_{b-1, r-b}^r + B_{a-1,b-1}^r
	\bigg\} \zeta^\lmot(2r+1) \end{aligned}
	\end{align*}
	We now make a number of straight forward simplifications.  Namely, \( A_{a,b}^r \) only depends on \( a,r \), and \( B_{a,b}^r \) only depends on \( b,r \).  Moreover since \( r \geq 1 \), if \( a \geq r \) or \( a < 0 \), then already \( A_{a,b}^r = 0 \), likewise if \( b \geq r \) or \( b < 0 \) then \( B_{a,b}^r = 0 \).  So we find
	\begin{align*}
	\pi(\widehat{\xi}_{a,b}^r) = {} 	
	& 2 (-1)^{r} \bigg\{  \begin{aligned}[t] 
		&   \big( 2 - 2^{-2r} \big) \binom{2r}{2a} {} -  A_{a-1,b-1}^r  + A_{b-1,a-1}^r  - B_{b-1, r-b}^r + B_{a-1,b-1}^r 	\bigg\} \zeta^\lmot(2r+1) \end{aligned}
	\end{align*}
	Note that \( - B_{b-1, r-b}^r + B_{a-1,b-1}^r = 0 \) just by their definitions, so overall we obtain
	\begin{align*}
	\pi(\widehat{\xi}_{a,b}^r) = {} 	
	& 2 (-1)^{r} \bigg\{  \begin{aligned}[t] 
	&   \big( 1 - 2^{-2r} \big) \binom{2r}{2a} + \binom{2r}{2b} 	\bigg\} \zeta^\lmot(2r+1) \end{aligned}
	\end{align*}
	Therefore
	\begin{align*}
		D_{2a+1} L^{a,b} & {} = \pi(\widehat{\xi}_{a,b}^r) \otimes \ttw^\mot(\{2\}^{a+b-r}) \\
		& {} = 2 (-1)^{r} \bigg\{  \big( 1 - 2^{-2r} \big) \binom{2r}{2a} + \binom{2r}{2b} 	\bigg\} \zeta^\lmot(2r+1) \otimes \ttw^\mot(\{2\}^{a+b-1})
	\end{align*}
	gives us the derivation of the left hand side of \eqref{eqn:mott2212ev}, for \( r > 0 \). \medskip

	On the other hand, a direct computation of \( D_{2r+1} R^{a,b} \) gives us that
	\begin{align*}
		D_{2a+1} R^{a,b} = {} & (-1)^{r+1} \cdot 2 \bigg[ \binom{2r}{2a} + \frac{2^{2r}}{2^{2r} - 1} \binom{2r}{2b} \bigg] \zeta^\lmot(\overline{2r+1}) \otimes \ttw^\mot(\{2\}^{a+b-r}) \\
		= {} & 2 (-1)^r \Bigg[ (1 - 2^{-2r}) \binom{2r}{2a} + \binom{2r}{2b} \Bigg] \zeta^\lmot(2r+1) \otimes \ttw^\mot(\{2\}^{a+b-r})
	\end{align*}
	is the derivation of the right hand side of \eqref{eqn:mott2212ev}, for \( r > 0 \). \medskip
	
	\paragraph{\bf Conclusion:} We have shown that \( D_{2r+1} L^{a,b} - R^{a,b} = 0 \) for \( 0 \leq r \leq a + b - 1 \), hence \( L^{a,b} - R^{a,b} \in \ker D_{<N} \).  Therefore by Glanois's theorem, we know that
	\[
		L^{a,b} - R^{a,b} = c \zeta^\mot(2a + 2b + 1) \,,
	\]
	for some \( c \in \Q \).  Then by applying the period map, we reduce to the numerically valid identity in \autoref{thm:sht2212ev} (with \( W = 0 \)), and hence see that \( c = 0 \).   Therefore the identity \( L^{a,b} = R^{a,b} \) is true on the motivic level, and this complete the proof.  \hfill \qedsymbol
	
	\section{Independence of Saha's elements}
	\label{sec:saha}
	
	We now turn to the first application of this motivic identity.  We show that the elements that Saha conjectured \cite{saha17} to be a basis for convergent MtV's are, at least, linearly independent. \medskip
	
	We recall briefly Saha's conjecture.
	
	\begin{Conj}[Saha, \cite{saha17}]
		Let
		\[
			\mathcal{B}^S \coloneqq \{ t(k_1,\ldots,k_{m-1}, k_m+1) \mid k_i \in \{1,2\} \} \,.
		\]
		Then \( \mathcal{B}^S \) is a basis for convergent MtV's.  Moreover, the weight \( w \) component of \( \mathcal{B}^S \) is
		\[
				\mathcal{B}_w^S = \{ t(k_1,\ldots,k_{m-1}, k_m+1) \mid k_i \in \{1,2\} \,, k_1 + \cdots + k_m = w-1 \} \,,
		\]
		which has cardinality \( \#{\mathcal{B}_N^S} = F_N \), for \( N > 1 \).  Here \( F_n = F_{n-1} + F_{n-2} \) is the \( n \)-th Fibonacci number, with \( F_1 = F_2 = 1 \).
	\end{Conj}
	
	We note that the arguments of such MtV's can be written as an arbitrary word in 1's and 2's, followed by either a 2 or a 3.  We can therefore schematically describe the set of arguments as follows
	\[
		w \in (\{1,2\}^\times \oplus 2) \cup (\{1,2\}^\times \oplus 3) \,,
	\]
	where \( \oplus \) denotes concatenation of words.

	\begin{Def}[Saha filtration]
		For, \(w \in (\{1,2\}^\times \oplus 2) \cup (\{1,2\}^\times \oplus 3) \), we define the level of \( w \) to be \( \deg_1 w + \deg_3 w \), i.e. the total number of \( 1 \)'s and \( 3 \)'s in the word.  We define \( \Q \)-subspace of \( \mathcal{H}^{(2)} \), and the level \( \leq \ell \) piece of the level filtration by
		\begin{align*}
			\mathcal{H}^S &\coloneqq \langle t^\mot(w) \mid w \in  (\{1,2\}^\times \oplus 2) \cup (\{1,2\}^\times \oplus 3) \rangle_\Q \,, \\
			S_\ell \mathcal{H}^S &\coloneqq \langle t^\mot(w) \mid w \in  (\{1,2\}^\times \oplus 2) \cup (\{1,2\}^\times \oplus 3) , \text{ s.t. } \deg_1 w + \deg_3 w \leq \ell \rangle_\Q 
		\end{align*}
		The associated graded to this filtration is then given by
		\[
			\gr_\ell^S \mathcal{H}^S \coloneqq S_\ell \mathcal{H}^S / S_{\ell-1} \mathcal{H}^S \,.
		\]
	\end{Def}

	\begin{Eg}
	The level \( \leq 1 \) part of this filtration is generated by the following elements
	\[
		S_1 \mathcal{H}^S = \langle t^\mot(\{2\}^a, 1, \{2\}^b), t^\mot(\{2\}^c, 3), t^\mot(\{2\}^d) \mid a,c,d \geq 0, b \geq 1\rangle_\Q \,,
	\]
	whereas the level \( \leq 0 \) part of this filtration is generated by 
	\[
		S_0 \mathcal{H}^S = \langle t^\mot(\{2\}^d) \mid d \geq 0 \rangle_\Q \,.
	\]
	\end{Eg}

	\begin{Lem}\label{lem:sahamotivic}
		The Saha-level is motivic.  More precisely, the following holds for all \( r' \geq 0 \)
		\[
			D_{2r'+1} S_\ell \mathcal{H}^S \subseteq \mathcal{L}_{2r+1}^{(2)} \otimes_\Q S_{\ell-1} \mathcal{H}^S \,.
		\]
		
		\begin{proof}
			Let \( r \geq 0 \) be odd, and \( \vec{k}= (k_1,\ldots,k_d) \in (\{1,2\}^\times \oplus 2) \cup (\{1,2\}^\times \oplus 3) \) with level \( \ell \).  We consider how to compute \( D_r \ttw^\mot(k_1,\ldots,k_d) \) via \autoref{prop:dkt}. 
			
			Firstly, if the deconcatenation term \eqref{eqn:dr:deconcat} \( \ttw^\mot(k_1,\ldots,k_j) \otimes \ttw^\mot(k_{j+1},\ldots,k_d) \) contributes, then the string \( (k_1,\ldots,k_j) \) of odd weight must contain a 1 or a 3.  Hence the level of \( (k_{j+1},\ldots,k_d) \) is reduced.
			
			Now, if the term \eqref{eqn:dr:0eps} contributes, we must satisfy the conditions \( \abs{\vec{k}_{i+1,j}} \leq r < \abs{\vec{k}_{i,j}} - 1 \).  This means that \( k_i \geq \abs{\vec{k}_{i,j}} - r > 1 \), so that \( \vec{k}_{i,j} - r = 2,3 \).  The case \( \abs{\vec{k}_{i,j}} = 3 \) could occurs if \( k_i = 3 \), and this can only occur if \( k_i = k_d \) with \( k_d = 3 \), so that \( i = j \).  But this is excluded from the sum, so \( \abs{\vec{k}_{i,j}} - r = 2 \).  Since \( r \) is odd, this implies \( \abs{\vec{k}_{i,j}} \) is  also odd, and so the subindex must contain (at least) one 1 or 3.  This is replaced by a 2, and so the level is reduced.  
			
			Likewise, if \eqref{eqn:dr:eps0} contributes, we must have \( \abs{\vec{k}_{i,j-1}} \leq r < \abs{\vec{k}_{i,j}} - 1 \).  This means \( k_j \geq \abs{\vec{k}_{i,j}} - r > 1 \), so that \( \abs{\vec{k}_{i,j}} - r = 2,3 \).  The case \( \abs{\vec{k}_{i,j}} - r = 2 \) is analogous to the previous: \( \abs{\vec{k}_{i,j}} \) is odd, so contains at least one 1 or 3.  This is replaced by a 2 and so the level is reduced.  Now, though, \( \abs{\vec{k}_{i,j}} - r = 3 \) occurs if \( j = d \) and \( k_d = 3 \).  But we see that \( \abs{\vec{k}_{i,j}} \) must be even, and already contains a three (from \( k_j = k_d = 3 \)).  Therefore it must also contain at least one 1.  Since a 3 and a 1 are replaced with a single 3, at the end of the string as \( j= d \), the element again is a Saha element, and of lower level.
		\end{proof}
	\end{Lem}
	
	From this lemma, we obtain a level-graded derivation
	\[
		\gr^S_\ell D_{2r+1} \colon \gr^S_\ell \mathcal{H}^S \to \mathcal{L}_{2r+1} \otimes_\Q \gr^S_{\ell-1} \mathcal{H}^S
	\]
	Moreover we claim, this map lands in the subspace of \( \mathcal{L}_{2r+1} \) generated by the single zeta element \( \zeta^\lmot(\overline{2r+1}) \).
	
	\begin{Lem}\label{lem:saha:cases}
		For \( \ell \geq 1 \), \( r' \geq 0 \), the level-graded derivation \( \gr^S_\ell D_{2r'+1} \) satisfies 
		\[
			\gr^S_\ell D_{2r'+1} \big( \gr^S_\ell \mathcal{H}^S \big) \subseteq \zeta^\lmot(\overline{2r'+1}) \Q \otimes_\Q \gr^S_{\ell-1} \mathcal{H}^S \,.
		\]
		
		\begin{proof}
			Let \( r \geq 1 \) be odd, and \( \vec{k}= (k_1,\ldots,k_d) \in (\{1,2\}^\times \oplus 2) \cup (\{1,2\}^\times \oplus 3) \) with level \( \ell \).  We consider how to compute \( D_r \ttw^\mot(k_1,\ldots,k_d) \) via \autoref{prop:dkt}, and more carefully track the contributions when we take elements of level \(\ell-1 \) in the right hand tensor factor.  For \( r = 1 \), this is clear, as \( \mathcal{H}^{(2)} = \langle \log^\mot(2) \rangle_\Q = \langle \zeta^\mot(\overline{1}) \rangle_\Q \).  I.e. the \( \mathcal{L} \)-factor of \( D_{1} \) can only be a multiple of \( \zeta^\lmot(\overline{1}) \), for dimensional reasons.  So we can assume \( r > 1 \)
			
			With \eqref{eqn:dr:deconcat}, the deconcatenation term \( \ttw^\lmot(k_1,\ldots,k_j) \otimes \ttw^\mot(k_{j+1}, \ldots,k_d) \), we see that for \( \vec{k}_{j+1,d} \) to have level \( \ell - 1 \), a single 1 or 3 must have been removed.  Therefore if \( \vec{k}_{1,j} = (\{2\}^a, 1, \{2\}^b) \), we know via \autoref{thm:mott2212}, that \( \ttw^\lmot(\vec{k}_{1,j}) \in \zeta^\lmot(\overline{2r+1}) \Q \).  Likewise, if \( \vec{k}_{1,j} = (\{2\}^a,3) \) (for if it contains a 3, it must be that \( j = d \), and \( k_d = 3 \).), one has from Murakami's evaluation in \autoref{thm:mott2232}, that \( \ttw^\lmot(\vec{k}_{1,j}) \in \zeta^\lmot(2r+1) \Q = \zeta^\lmot(\overline{2r+1}) \Q \).
			
			Then we turn to the contribution from \eqref{eqn:dr:0eps}, and recall the considerations in the proof of \autoref{lem:sahamotivic}.  Namely, \( \abs{\vec{k}_{i,j}} - r = 2 \), \( k_i = 2 \), which forces certain behaviour onto \( \vec{k}_{i,j} \).  If \( \abs{\vec{k}_{i,j}} - r = 2 \), then \( \vec{k}_{i,j} \) must contain an odd number of 1's and 3's.  But for level-grading reasons, it actually must contain exactly one such, which if it were a 3, must appear in the last position.  We have the following cases.
			\begin{center}
			\begin{tabular}{c|c|c}
			$\abs{\vec{k}_{i,j}}-r$ & $ \vec{k}_{i,j} $ & Contribution to \( D_r \) \\ \hline
			2 & $(2,\{2\}^a,1,\{2\}^b)$ & $ \zeta_{2-2}^\lmot(\{2\}^a,1,\{2\}^b) \otimes \ttw^\mot(\vec{k}_{1,i-1}, 2, \vec{k}_{j+1,d})$ \\
			2 & $(2,\{2\}^a,3)$ & $ \zeta_{2-2}^\lmot(\{2\}^a,3)  \otimes \ttw^\mot(\vec{k}_{1,i-1}, 2, \vec{k}_{j+1,d})$ \\
			\end{tabular}		
			\end{center}
			In either case, we see via \autoref{thm:brown:z2232}, or rather \eqref{eqn:coeffzl2232} thereafter, that each \( \zeta^\lmot_\alpha(\vec{k}_{i+1,j}) \in \zeta^\lmot(\overline{2r+1}) \Q \).
			
			Likewise, from \eqref{eqn:dr:eps0}, we have \( \abs{\vec{k}_{i,j}} - r = 2, 3 \).  If \( \abs{\vec{k}_{i,j}} - r = 2 \), then \( k_j = 2, 3 \) and  \( \abs{\vec{k}_{i,j}} \) contains an odd number of 1's and 3's.  In the level-graded, it therefore must contain exactly one 1 or one three (where a 3 would appear at the end).  Otherwise \( \abs{\vec{k}_{i,j}} - r = 3 \), so \( k_j = 3 \), and \( \abs{\vec{k}_{i,j}} \) contains an even number of 1's and 3's.  As it already must contain a 3 at the end (since \( k_j = 3 \) and so \( j = d \)), it must also contain 1 somewhere else.	
			
			Be aware that we must reverse \( \vec{k}_{i,j} \) when inserting it into \( \zeta^\lmot \) in term \eqref{eqn:dr:eps0}.  We have the following cases.
			\begin{center}
				\begin{tabular}{c|c|c}
					$\abs{\vec{k}_{i,j}}-r$ & $ \vec{k}_{i,j} $ & Contribution to \( D_r \) \\ \hline
					2 & $(\{2\}^a,1,\{2\}^b,2)$ & $ -\zeta^\lmot_{2-2}(\{2\}^b,1,\{2\}^a) \otimes \ttw^\mot(\vec{k}_{1,i-1}, 2, \vec{k}_{j+1,d})$  \\
					2 & $(\{2\}^a,3)$ & $ -\zeta^\lmot_{3-2}(\{2\}^a) \otimes \ttw^\mot(\vec{k}_{1,i-1}, 2, \vec{k}_{j+1,d})$ \\
					3 & $(\{2\}^a,1,\{2\}^b,3)$ & $ -\zeta^\lmot_{3-3}(\{2\}^b,1,\{2\}^a) \otimes \ttw^\mot(\vec{k}_{1,i-1}, 3)$  \\
				\end{tabular}	
			\end{center}
			Once again, we see via \autoref{thm:brown:z2232}, or rather \eqref{eqn:coeffzl2232} thereafter, that in every case the term \( \zeta^\lmot_\alpha(\vec{k}_{j-1,i}) \in \zeta^\lmot(\overline{2r+1}) \Q \).
		\end{proof}
	\end{Lem}

	We now look at the action of these derivations on elements of a given level, and package them together into the following linear map.
	
	\begin{Def}
		For all \( N ,\ell \geq 1 \), let \( \partial^S_{N,\ell} \) be the linear map
		\[
 			\partial_{N,\ell}^S \colon \gr_\ell^S \mathcal{H}_N^S \to \bigoplus_{1 \leq 2r+1 \leq N} \gr_{\ell-1}^S \mathcal{H}_{N-2r-1}^S \,,
		\]
		defined by first applying \( \bigoplus_{1\leq2r+1\leq N} \gr_\ell^S D_{2r+1} \big|_{\gr_\ell^S \mathcal{H}_N^S} \), and then sending all \( \log^\mot(2) \mapsto \tfrac{1}{2}, \zeta^\lmot(2r+1) \mapsto 2^{2r-1} \), \( r > 0 \) to by the projection \begin{align*}
			\widetilde{\pi}_{2r+1} & \colon \Q \zeta^\lmot(\overline{2r+1}) \to \Q  \\
				& \begin{cases} 
					\log^\mot(2) \mapsto \frac{1}{2} \,, & \text{if \( r = 0 \)\,,}\\
					\zeta^\lmot(2r+1) \mapsto 2^{2r-1} \,, & \text {if \( r > 0 \)}
					\end{cases}
		\end{align*}
	\end{Def}

	The goal is to show that the maps \( \partial_{N,\ell}^S \) are injective for \( \ell \geq 1 \).  Then by recursion, we will establish the elements of level \( \ell \) are linearly independent (otherwise \(\partial_{N,\ell}^S \) would construct a non-trivial relation of strictly smaller level).

	\begin{Def}[Matrix basis]
		Let \( \ell,N \geq 1 \), with \( N \equiv \ell \pmod*{2} \).  Define the following sets
		\begin{align*}
			B_{S,N,\ell} &\coloneqq \{ w \in (\{1,2\}^\times \oplus 2) \cup (\{1,2\}^\times \oplus 3) \mid \deg_1 w + \deg_3 w = \ell, \abs{w} = N \} \\
			B'_{S,N,\ell} &\coloneqq \{ w \in (\{1,2\}^\times \oplus 2) \cup (\{1,2\}^\times \oplus 3) \mid \deg_1 w + \deg_3 w = \ell-1, \abs{w} < N \}
		\end{align*}
		In the case \( \ell = 1 \), the set \( B'_{S,N,\ell} \) also includes the empty word (of weight 0 and level 0).  Sort both sides in reverse colexicographic order (i.e. reading right to left, and the largest first), with \( 3 < 1 < 2 \).  In this ordering, all terms ending with a 3 will appear last.
	\end{Def}

	A counting argument shows that \( \# B_{S,N,\ell} = \# B'_{S,N,\ell} \) for all such choices of \( \ell, N \).  These basis elements will be used to define the matrix form of the linear map \(	\partial_{N,\ell}^S \), and the claim of injectivity corresponds to non-zero determinant.

	\begin{Eg}
		For \( N = 8, \ell = 2 \), we have
		\begin{align*}
			& B_{S,N,\ell} = \{ 11222 \,, 12122\,, 21122\,, 12212\,, 21212\,, 22112\,, 1223\,, 2123\,, 2213 \} \, \\ 
			& B'_{S,N,\ell} = \{ 1222\,, 2122\,, 122\,, 2212\,, 212\,, 12\,, 223\,, 23\,, 3 \}
		\end{align*}
	\end{Eg}

	\begin{Def}[Matrix of \( \partial^S_{N,\ell} \)]
		For \( \ell \geq 1, N \geq 1 \), with \( N \equiv \ell \pmod*{2} \), let 
		\[
		 M_{S,N,\ell} \coloneqq \big(f_{w'}^w\big)_{w\in B_{S,N,\ell}, w' \in B_{S,N,\ell}'}
		 \]
		 be the matrix of \( \partial^S_{N,\ell} \) with respect to the bases \( B_{S,N,\ell} \) and \( B'_{S,N,\ell} \).  Here \( f_w'^w \) denotes the coefficient of \( \ttw^\mot(w') \) in \( \partial^S_{N,\ell} \ttw^\mot(w) \), and in the matrix \( w \) corresponds to rows, and \( w' \) to columns.
	\end{Def}

	It will be helpful to introduce some notation to talk more directly about there coefficients of \( \zeta^\lmot(2r+1) \) in various identities.

	\begin{Def}
		Write \( c_{2^a32^b} \), \( c_{2^a1} \), \( d_{2^a12^b} \), \( d_{2^a32^b} \) to be the coefficient such that
		\begin{align*}
			\zeta^\lmot(\{2\}^a,3,\{2\}^b) &= \zeta^\lmot(\{2\}^b,1,\{2\}^{a+1})= c_{2^a32^b} \zeta^\lmot(2a+2b+3)  \\
			\zeta^\lmot(\{2\}^a,1) &= \zeta_1^\lmot(\{2\}^a) = c_{2^a1} \zeta^\lmot(2a+1)  \\
			\ttw^\lmot(\{2\}^a,1,\{2\}^b) &= d_{2^a12^b} \zeta^\lmot(2a+2b+1) \\
			\ttw^\lmot(\{2\}^a,3,\{2\}^b) &= d_{2^a32^b} \zeta^\lmot(2a+2b+3) \,.
		\end{align*}
		Moreover note that \( d_1 = 2 \) so that \( \ttw^\lmot(1) = d_1 \cdot \frac{1}{2} \log^\lmot(2) \).  From the computations in \autoref{thm:brown:z2232} and \eqref{eqn:coeffzl2232} thereafter, and from \autoref{thm:mott2212} and \autoref{thm:mott2232}, we have the following explicit formulae.
		\begin{align*}
			c_{2^a32^b} &= 2(-1)^{a+b}\bigg({-}\binom{2a+2b+2}{2a+2} + (1 - 2^{-2a-2b-2})\binom{2a+2b+2}{2b+1} \bigg) \\
			c_{1} &= 0 \\
			c_{2^a1} &= 2(-1)^a \\
			d_{2^a12^b} &= 2(-1)^{a+b}\bigg((1 - 2^{-2a-2b})\binom{2a + 2b}{2a} + \binom{2a+2b}{2b}\bigg) \\
				&= 4(-1)^{a+b}(1 - 2^{-2a-2b-1})\binom{2a + 2b}{2a} \\
			d_{2^a32^b} &= 2(-1)^{a+b} \bigg( \binom{2a+2b+2}{2a+1} + (1 - 2^{-2a-2b-2}) \binom{2a+2b+2}{2b+1} \bigg) \\
			&= 4(-1)^{a+b} (1 - 2^{-2a-2b-3}) \binom{2a+2b+2}{2a+1}
		\end{align*}
	\end{Def}

	\begin{Eg}
		For \( N = 8, \ell = 2 \), the matrix \( M_{S,8,2} \) is as follows; the first row and column label the elements of \( B'_{S,8,2} \) and \( B_{S,8,2} \) respectively.
		\begin{center}
			\small
		\begin{tabular}{c||c|c|c|c|c|c|c|c|c}
		\!	& \!1222\! & \!2122\! & \!122\! & \!2212\! & \!212\! & \!12\! & \!223\! & \!23\! & \!3\! \\ \hline \hline
		\!11222\!	& 1 & 0 & ${-}2 c_{21}$ & 0 & 0 & ${-}8 c_{221}$ & 0 & 0 & 0 \\
		\!12122\!	& 0 & 1 & \!$2 d_{12}{-}2 c_{21}$\! & 0 & 0 & $8 c_{23}{-}8 c_{32}$ & 0 & 0 & 0 \\
		\!21122\!	& 0 & 0 & $2 d_{21}$ & 0 & ${-}2 c_{21}$ & 0 & 0 & 0 & 0 \\
		\!12212\!	& 0 & 0 & $2 c_{21}$ & 1 & $2 d_{12}{-}2 c_{21}$ & ${-}8 c_{23}{+}8 c_{32}{+}8 d_{122}$ & 0 & 0 & 0 \\
		\!21212\!	& 0 & 0 & 0 & 0 & $2 d_{21}$ & $8 d_{212}$ & 0 & 0 & 0 \\
		\!22112\!	& 0 & 0 & 0 & 0 & $2 c_{21}$ & $8 d_{221}$ & 0 & 0 & 0 \\
		\!1223\!	& 0 & 0 & $2 c_3{-}2 c_{21}$ & 0 & 0 & $8 c_{23}{-}8 c_{221}$ & 1 & $2 d_{12}{-}2 c_{21}$ & $8 d_{122}{-}8 c_{221}$ \\
		\!2123\!	& 0 & 0 & 0 & 0 & $2 c_3{-}2 c_{21}$ & 0 & 0 & $2 d_{21}{-}2 c_{21}$ & $8 d_{212}{-}8 c_{32}$ \\
		\!2213\!	& 0 & 0 & 0 & 0 & 0 & 0 & 0 & $2 c_{21}{-}2 c_3$ & $8 d_{221}{-}8 c_{23}$
		\end{tabular}
		\end{center}
		The entries \( 1 \) in the matrix arise from both the deconcatenation term \( 2 \ttw^\lmot(1) \) which appears in \( D_{1} \) as per \autoref{prop:d1}.  With the projection \( \log^\lmot(2) \to \frac{1}{2} \), this combination gives 1 above.
		
		After substituting the values for \( c_\bullet \) and \( d_\bullet \) using the formulae above, we obtain the matrix
		\begin{center}
		\begin{tabular}{c||c|c|c|c|c|c|c|c|c}
			\!	& \!1222\! & \!2122\! & \!122\! & \!2212\! & \!212\! & \!12\! & \!223\! & \!23\! & \!3\!  \\ \hline\hline
			\!11222\!	& 1 & 0 & 4 & 0 & 0 & ${-}16$ & 0 & 0 & 0 \\
			\!12122\!	& 0 & 1 & ${-}3$ & 0 & 0 & ${-}80$ & 0 & 0 & 0 \\
			\!21122\!	& 0 & 0 & ${-}7$ & 0 & 4 & 0 & 0 & 0 & 0 \\
			\!12212\!	& 0 & 0 & ${-}4$ & 1 & $-3$ & 111 & 0 & 0 & 0 \\
			\!21212\!	& 0 & 0 & 0 & 0 & ${-}7$ & 186 & 0 & 0 & 0 \\
			\!22112\!	& 0 & 0 & 0 & 0 & ${-}4$ & 31 & 0 & 0 & 0 \\
			\!1223\!	& 0 & 0 & 6 & 0 & 0 & ${-}60$ & 1 & ${-}3$ & 15 \\
			\!2123\!	& 0 & 0 & 0 & 0 & 6 & 0 & 0 & ${-}3$ & 150 \\
			\!2213\!	& 0 & 0 & 0 & 0 & 0 & 0 & 0 & ${-}6$ & 75 \\
		\end{tabular}
		\end{center}
		We notice already that the matrix has odd entries on the diagonal, and all entries below the diagonal are even.  Therefore the matrix is upper triangle modulo 2 with 1's on the whole diagonal.  So it has determinant \( \equiv 1 \pmod*{2} \) and is invertible.  We aim to show this is a general phenomenon for level \( \ell > 1 \).  In fact, we shall show that modulo 2, \( \partial^S_{N,\ell} \) acts by deconcatenation, so that the only entries \( d_{\bullet} \) occur above (or rather right) of the main diagonal inclusive.
	\end{Eg}

	\begin{Rem}
		In the case of level \( \ell = 1 \), the matrix actually has even determinant, and so the above considerations would fail.  However, we will show that the evenness of the determinant arises exactly from the single even entry in the last row.  This single entry, is the deconcatenation term from \( \partial_{N,1}^S \ttw^\mot(\{2\}^a,3) = d_{2^a3} \emptyset \), and so by expanding about the last row, we would reduce to \( d_{2^a3} \) times a determinant which is invertible modulo 2, at least once we prove this previous claim.
	\end{Rem}

	\begin{Lem}
		Let \( \ell > 1 \) and \( w \in B_{S,N,\ell}\).  Then every coefficient of \( \ttw^\mot(u) \), \( u \in B'_{S,N,\ell} \) in
		\[
			\partial^S_{N,\ell}\ttw^\mot(w) - \sum_{\substack{w = uv \\ \deg_3u + \deg_1u = 1 }} 2^{\abs{u}-2}d_u \ttw^\mot(v) \,,
		\]
		is an even integer.
		
		\begin{proof}
			Let \( \vec{k} \in (\{1,2\}^\times \oplus 2) \cup (\{1,2\}^\times \oplus 3) \), with \( \ell = \deg_1 \vec{k} + \deg_3 \vec{k} \).  We consider how to compute \( \gr_\ell D_r \ttw^\mot(\vec{k}) \) via \autoref{prop:dkt} and the simplification in \autoref{prop:d1} for \( D_1 \).  For \( r = 1 \), we immediately find
			\begin{align*}
				\gr_\ell D_1 \ttw^\mot(\vec{k}) &= \widetilde{\pi}_1(2 \log^\lmot(2)) \delta_{k_1 = 1} \ttw^\mot(k_2,\ldots,k_d) \\
				&= \delta_{k_1 = 1} 2^{-1} d_1 \ttw^\mot(k_2,\ldots,k_d)
			\end{align*}
			Now if we assume \( r > 1 \), we have 
			\begin{align}
		& D_r \big( \ttw^\mot(k_1,\ldots,k_d) \big) = \notag \\
		& \label{eqn:drgr:deconcat} \sum_{1 \leq j \leq d} \delta_{\abs{\vec{k}_{1,j}}=r} \widetilde{\pi}_r \big( \ttw^\lmot(k_1,\ldots,k_j) \big) \cdot \ttw^\mot(k_{j+1},\ldots,k_d) \\[1ex]
		& \label{eqn:drgr:0eps} + \!\! \sum_{1 \leq i < j \leq d} 
		\delta_{\abs{\vec{k}_{i+1,j}} \leq r < \abs{\vec{k}_{i,j}}-1} \begin{aligned}[t] \widetilde{\pi}_{r} \Big( \zeta^\lmot_{r-\abs{\vec{k}_{i+1,j}}}&(k_{i+1},\ldots,k_j) \Big) \\
		&  {} \cdot \ttw^\mot(k_1,\ldots,k_{i-1},\abs{\vec{k}_{i,j}} - r, k_{j+1},\ldots,k_d) \end{aligned} \\[1ex]
		& \label{eqn:drgr:eps0} - \!\! \sum_{1 \leq i < j \leq d} 
		\delta_{\abs{\vec{k}_{i,j-1}} \leq r < \abs{\vec{k}_{i,j}}-1} \begin{aligned}[t] \widetilde{\pi}_r \Big( \zeta^\lmot_{r-\abs{\vec{k}_{i,j-1}}}&(k_{j-1},\ldots,k_i)  \Big) \\
		& {} \cdot \ttw^\mot(k_1,\ldots,k_{i-1},\abs{\vec{k}_{i,j}} - r, k_{j+1},\ldots,k_d) \end{aligned}
		\end{align}	
		We note that since \( \ell > 1 \), \( \gr_\ell D_{\abs{\vec{k}}} \ttw^\mot(\vec{k}) = \widetilde{\pi}_{\abs{\vec{k}}} \big( \ttw^\mot(\vec{k}) \big) \ttw^\mot(\emptyset) = 0 \), since \( \emptyset \) -- the empty word -- has level \( 0 < \ell - 1 \).  (For this, consider how to compute \( D_{2r+1} \) via the graded parts of the coaction \( \Delta (x) = 1 \otimes x + x \otimes 1 + \Delta'(x) \); the part with full weight in the left hand factor is then clearly \( x \otimes 1 \).)  This means the deconcatenation part in \eqref{eqn:drgr:deconcat} can never involve \( \ttw^\lmot(\ldots, 3) \), so must be of the form \( \widetilde{\pi}_{2r+1} \big( \ttw^\lmot(\{2\}^a, 1, \{2\}^b)  \big) = 2^{2a + 2b -1} d_{2^a12^b} \).
		
		Now consider \eqref{eqn:drgr:0eps}.  According to the table of cases in \autoref{lem:saha:cases}, we have the following contributions.
		\begin{center}
			\begin{tabular}{c|c|c}
			$\abs{\vec{k}_{i,j}}-r$ & $ \vec{k}_{i,j} $ & Contribution to \( \gr_\ell D_r \) \\ \hline
			2, \( b = 0 \) & $(2,\{2\}^a,1,\{2\}^b)$ & $ 2^{2a-1} c_{2^a1} \ttw^\mot(\vec{k}_{1,i-1}, 2, \vec{k}_{j+1,d})$ \\
			2, \( b > 0 \) & $(2,\{2\}^a,1,\{2\}^b)$ & $ 2^{2a+2b-1} c_{2^{b-1}32^a} \ttw^\mot(\vec{k}_{1,i-1}, 2, \vec{k}_{j+1,d})$ \\
			2 & $(2,\{2\}^a,3)$ & $ 2^{2a+1} c_{2^a3} \ttw^\mot(\vec{k}_{1,i-1}, 2, \vec{k}_{j+1,d})$ \\
		\end{tabular}
		\end{center}	
		Likewise for \eqref{eqn:drgr:eps0}, we have the following contributions.
		\begin{center}
			\begin{tabular}{c|c|c}
				$\abs{\vec{k}_{i,j}}-r$ & $ \vec{k}_{i,j} $ & Contribution to \( \gr_\ell D_r \) \\ \hline
				2, \( a = 0 \) & $(\{2\}^a,1,\{2\}^b,2)$ & $ -2^{2b-1} c_{2^b1} \ttw^\mot(\vec{k}_{1,i-1}, 2, \vec{k}_{j+1,d})$  \\
				2, \( a > 0 \) & $(\{2\}^a,1,\{2\}^b,2)$ & $ -2^{2a+2b-1} c_{2^{a-1}32^b} \ttw^\mot(\vec{k}_{1,i-1}, 2, \vec{k}_{j+1,d})$  \\
				2 & $(\{2\}^a,3)$ & $ -2^{2a-1} c_{2^a1} \ttw^\mot(\vec{k}_{1,i-1}, 2, \vec{k}_{j+1,d})$ \\
				3, \( a = 0 \) & $(\{2\}^a,1,\{2\}^b,3)$ & $ -2^{2b-1}c_{2^b1}\ttw^\mot(\vec{k}_{1,i-1}, 3)$  \\
				3, \( a > 0 \) & $(\{2\}^a,1,\{2\}^b,3)$ & $ -2^{2a+2b-1}c_{2^{a-1}32^b} \ttw^\mot(\vec{k}_{1,i-1}, 3)$  \\
			\end{tabular}
		\end{center}
		The key points to observe now are that
		\begin{align}
			2^{2a-1} c_{2^a1} &= \begin{cases}
				0 & a = 0 \\
				2^{2a} (-1)^a & a  > 0
			\end{cases}\label{eqn:even:221} \\
			2^{2a+2b+1} c_{2^a32^b} &=  (-1)^{a+b}\bigg({-} 2^{2a+2b+2} \binom{2a+2b+2}{2a+2} + (2^{2a+2b+2} - 1)\binom{2a+2b+2}{2b+1} \bigg) \,, \label{eqn:even:232}
		\end{align}
		hence both coefficients are always even integers, and so \( \equiv 0 \pmod*{2} \).  For \(c_{2^a32^b} \) it follows by writing \( \binom{2a+2b+2}{2b+1} = \frac{2(a + b + 1)}{2a+1} \binom{2a + 2b + 1}{2b+1} \).  Compare Corollary 4.4 \cite{brown12} for a more precise statement about \( \nu_2(c_{2^a32^b}) \), the 2-adic valuation thereof, which was one of the key lemmas in Brown's proof of the linear independence of \( \zeta^\mot(\{2,3\}^\times) \).
		
		Since all terms arising from \( \zeta^\lmot \) have even coefficient in \( \gr_\ell D_{r} \), we see that the only remaining terms arise from the deconcatenation part, and so the lemma follows.
		\end{proof}
	\end{Lem}

	\begin{Thm}\label{thm:matS:inj}
		For \( N, \ell \geq 1 \), the matrix \( M_{S,N,\ell} \) is invertible.
		
		\begin{proof}
			We proceed in a similar way as to the proofs by Murakami \cite[Theorem 36]{murakami21}, and Brown \cite[Corollary 6.2]{brown12}.
			
			We assume initially that \( \ell > 1 \), so that \( d_{2^a3} \) does not appear as an entry.  For \( \ell > 1 \), consider the map
			\begin{align*}
				B'_{S,N,l} &\to B_{S,N,\ell} \\
				u &\mapsto 2^{r}1u \,,
			\end{align*}
			where \( r \) is the unique integer such that \( \abs{2^r1u} = N \).  This map is a bijection, and preserves the ordering of both \( B'_{S,N,\ell} \) and \( B_{S,N,\ell} \).  That is to say, \( u < v \) iff \( 2^r1u < 2^{r'}1v \), which holds as we are in the reverse colexicographic (reading right to left, largest first) order with \( 3 < 1 < 2 \).  The diagonal entries of \( M_{S,N,\ell} \) are of the form \( f_{u}^{2^r1u} = 2^{2a-1}d_{2^a1} + 2n \), \( n \Z\).  Moreover, the only other non-even entries in the column indexed by \( u \in B'_{S,N,\ell} \) occur for rows indexed by \( w = 2^a12^bu \), however since \( 2^r1u < 2^a12^bu \), these occur above the diagonal.  These entries are also integral since
			\begin{align*}
				2^{2a+2b-1} d_{2^a12^b} & {} = (-1)^{a+b} 2^{2a+2b+1} (1 - 2^{-2a-2b-1})\binom{2a + 2b}{2a} \\
				&{} = (-1)^{a+b} (2^{2a+2b+1} - 1)\binom{2a + 2b}{2a} \,.
 			\end{align*}
			We finally note that
			\[
				2^{2a-1}d_{2^a1} = (-1)^{a}(2^{2a+1} - 1) \equiv 1 \pmod*{2} \,,
			\]
			so that \( f_u^{2^r1u} \equiv 1 \pmod*{2} \).  This means the matrix is integral, and modulo 2 it reduces to an upper triangle matrix with leading diagonal equal to 1.  Hence \( M_{S,N,\ell} \) has determinant \( \equiv 1 \pmod*{2} \), and so is invertible. \medskip
			
			When \( \ell = 1 \), we note that all of the previous steps apply for all words \( w = 2^a12^b \), \( b \geq 0 \) indexing the rows.  However, we obtain as the last row corresponding to the word \( w = 2^a3 \), the row vector
			\[
				(0, \ldots, 0, 2^{2a+1}d_{2^a3}) \,.
			\]
			since one must deconcatenate the entire word to reduce the level by 1.  Expand the determinant out about the last row, and we reduce to the submatrix involving words \( w = 2^a12^b \), \( b > 0 \) indexing the rows, and \( 2^a \), \( a > 0 \) indexing the columns.  This submatrix is integral, and modulo 2 it is upper triangle with 1's on the diagonal.  Hence has non-zero determinant.  Since
			\begin{align*}
				2^{2a+1}d_{2^a3} &= 2^{2a+3}(-1)^{a} (1 - 2^{-2a-3}) \binom{2a+2}{2a+1} \\
				&= (-1)^{a} (2^{2a+3} - 1) (2a + 2)  \\
				&\neq 0 
			\end{align*}
			the determinant of \( M_{S,N,\ell} \) is still non-zero, and so \( M_{S,N,\ell} \) is also invertible when \( \ell = 1 \).
		\end{proof}
	\end{Thm}

	\begin{Cor}\label{cor:saha:indep}
		The Saha elements 
		\[
			\{ \ttw^\mot(k_1, \ldots, k_{d-1}, k_d+1 ) \mid k_i \in \{ 1, 2 \}\}
		\]
		are linearly independent.
		
		\begin{proof}
			We proceed by induction on the level, as in \cite[Theorem 7.4]{brown12}, \cite[Corollary 38]{murakami21}.  The elements of level \( \ell = 0 \) are of the form \( \ttw^\mot(\{2\}^n) \), which are linearly independent because weight is a grading on \( \mathcal{H}^{(2)} \).  Now suppose the elements
			\[
				\{ \ttw^\mot(w) \mid w \in (\{1,2\}^\times \oplus 2) \cup (\{1,2\}^\times \oplus 3) \,, \deg_1 w + \deg_w 3 \leq \ell -1  \} \,,
			\]
			 of level \( \leq \ell - 1 \) are linearly independent.  Since weight is a grading on \( \mathcal{H}^{(2)} \), any non-trivial linear relation between elements of level \( \ell \) can be assumed as homogeneous of some weight \( N \).  By \autoref{thm:matS:inj}, the map \( \partial^S_{N,\ell} \) is injective as the matrix of the map is invertible.  Application of \( \partial_S^{N,\ell} \) to a non-trivial linear relation between level \( \ell \) elements produces a non-trivial linear relation of strictly smaller level, which does not exist by the induction assumption.  So the elements of level \( \ell \) are also linearly independent, which completes the proof by induction.
		\end{proof}
	\end{Cor}

	\begin{Cor}
		The space \( \mathcal{T}_w^\mathrm{conv} \) of convergent motivic MtV's has dimension \( \geq F_{N} \) in weight \( N > 1\), where \( F_n = F_{n-1} + F_{n-2} \) are the Fibonacci numbers, with \( F_1 = F_2 = 1 \).
	\end{Cor}

	\begin{Rem}
		For strictly convergent motivic MtV's, we do not appear yet to have the correct upper bound to show the Saha elements are a basis.  The motivic MtV's fit into the following inclusions
		\[
			\mathcal{H}^{(1)}_N \subseteq \mathcal{T}^\mathrm{conv}_N \subseteq \mathcal{T}^\mathrm{ext}_N \subseteq \mathcal{H}^{(2)}_N \,,
		\]
		where \( \mathcal{T}_N^\mathrm{conv} \) denotes the space of convergent motivic MtV's (with last argument \( \geq 2 \)) of weight \( N \), and \( \mathcal{T}_N^\mathrm{ext} \) denotes the space of all shuffle regularised motivic MtV's of weight \( N \).  (The first inclusion follows from Murakami's motivic Galois \cite[Theorem 8]{murakami21} descent showing \( \ttw^\mot(k_1, \ldots, k_d) \in \mathcal{H}^{(1)} \), whenever all \( k_i \geq 2 \).  The upper bound of \( \dim_\Q \mathcal{H}_N^{(2)} \leq F_{N+1} \) (established in \cite{deligneGoncharov}) only gives us the bound that \( F_{N} \leq \dim_{Q} \mathcal{T}_N^\mathrm{conv} \leq F_{N+1} \).  Below in \autoref{cor:text:eq:H2}, we will show however that \( \mathcal{T}^\mathrm{ext} = \mathcal{H}^{(2)} \)  using the independence of the Hoffman one-two elements.
	\end{Rem}

	\section{The Hoffman one-two elements as a basis}
	\label{sec:hoff}
	
	We now turn to the second application of the motivic identity.  We show that the elements whose arguments consist of only 1's and 2's are linearly independent as motivic MtV's (analogous to Hoffman's conjectured (motivically true) basis of MZV's as those with arguments 2's and 3's).  Dimension counting then shows that the elements -- \( F_{N+1} \) many in weight \( N \) -- must be a basis for motivic MtV's and alternating motivic MZV's, as these spaces have known dimensions \( \leq F_{N+1} \).
	
	\begin{Def}[Hoffman $t$ filtration]
		For, \(w \in \{1,2\}^\times \), we define the level of \( w \) to be \( \deg_1 w \), i.e. the total number of \( 1 \) in the word.  We define \( \Q \)-subspace of \( \mathcal{H}^{(2)} \), and the level \( \leq \ell \) piece of the level filtration by
		\begin{align*}
		\mathcal{H}^{H} &\coloneqq \langle t^\mot(w) \mid w \in  \{1,2\}^\times \rangle_\Q \,, \\
		H_\ell \mathcal{H}^H &\coloneqq \langle t^\mot(w) \mid w \in  \{1,2\}^\times \,, \text{ s.t. } \deg_1 w \leq \ell \rangle_\Q 
		\end{align*}
		The associated graded to this filtration is then given by
		\[
		\gr_\ell^H \mathcal{H}^H \coloneqq H_\ell \mathcal{H}^H / H_{\ell-1} \mathcal{H}^H \,.
		\]
	\end{Def}
	
	\begin{Eg}
		The level \( \leq 1 \) part of this filtration is generated by the following elements
		\[
		H_1 \mathcal{H}^H = \langle t^\mot(\{2\}^a, 1, \{2\}^b), t^\mot(\{2\}^c) \mid a,b,c \geq 0\rangle_\Q \,,
		\]
		whereas the level \( \leq 0 \) part of this filtration is generated by 
		\[
		H_0 \mathcal{H}^H = \langle t^\mot(\{2\}^c) \mid c \geq 0 \rangle_\Q \,.
		\]
	\end{Eg}
	
	\begin{Lem}\label{lem:hoffmanmotivic}
		The Hoffman-level is motivic.  More precisely, the following holds for all \( r' \geq 0 \)
		\[
		D_{2r'+1} S_\ell \mathcal{H}^H \subseteq \mathcal{L}_{2r+1}^{(2)} \otimes_\Q H_{\ell-1} \mathcal{H}^H \,.
		\]
		
		\begin{proof}
			The proof is essentially the same as for \autoref{lem:sahamotivic}, except the cases \( \vec{k}_{i,j} - r = 3 \) cannot occur, since every argument \( k_i \leq 2 \).
		\end{proof}
	\end{Lem}

	From this lemma, we obtain a level-graded derivation
\[
\gr^H_\ell D_{2r+1} \colon \gr^H_\ell \mathcal{H}^H \to \mathcal{L}_{2r+1} \otimes_\Q \gr^H_{\ell-1} \mathcal{H}^H
\]
Moreover we claim, this map lands in the subspace of \( \mathcal{L}_{2r+1} \) generated by the single zeta element \( \zeta^\lmot(\overline{2r+1}) \).

\begin{Lem}\label{lem:hoffmot:form}
	For \( \ell \geq 1 \), \( r' \geq 0 \), the level-graded derivation \( \gr^H_\ell D_{2r'+1} \) satisfies 
	\[
	\gr^H_\ell D_{2r'+1} \big( \gr^H_\ell \mathcal{H}^H \big) \subseteq \zeta^\lmot(\overline{2r'+1}) \Q \otimes_\Q \gr^H_{\ell-1} \mathcal{H}^H \,.
	\]
	
	\begin{proof}
		With \( r = 0 \), the claim is clear as \( \mathcal{L}_{1}^{(2)} \) is generated by \( \log^\lmot(2) \).  So let \( r \geq 0 \) be odd, and \( \vec{k}= (k_1,\ldots,k_d) \in \{1,2\}^\times \) with level \( \ell \).  We consider how to compute \( D_r \ttw^\mot(k_1,\ldots,k_d) \) via \autoref{prop:dkt}, and more carefully track the contributions when we take elements of level \(\ell-1 \) in the right hand tensor factor. 
		
		With \eqref{eqn:dr:deconcat}, the deconcatenation term \( \ttw^\lmot(k_1,\ldots,k_j) \otimes \ttw^\mot(k_{j+1}, \ldots,k_d) \), we see that for \( \vec{k}_{j+1,d} \) to have level \( \ell - 1 \), a single 1 must have been removed.  Therefore if \( \vec{k}_{1,j} = (\{2\}^a, 1, \{2\}^b) \), we know via \autoref{thm:mott2212}, that \( \ttw^\lmot(\vec{k}_{1,j}) \in \zeta^\lmot(\overline{2r+1}) \Q \).
		
		Then we turn to the contribution from \eqref{eqn:dr:0eps}, and recall the considerations in the proof of \autoref{lem:hoffmanmotivic}.  Namely, \( \abs{\vec{k}_{i,j}} - r = 2 \), \( k_i = 2 \), which forces certain behaviour onto \( \vec{k}_{i,j} \).  If \( \abs{\vec{k}_{i,j}} - r = 2 \), then \( \vec{k}_{i,j} \) must contain an odd number of 1's.  But for level-grading reasons, it actually must contain exactly one such.  We have the following case.
		\begin{center}
			\begin{tabular}{c|c|c}
				$\vec{k}_{i,j}-r$ & $ \vec{k}_{i,j} $ & Contribution to \( D_r \) \\ \hline
				2 & $(2,\{2\}^a,1,\{2\}^b)$ & $ \zeta_{2-2}^\lmot(\{2\}^a,1,\{2\}^b) \otimes \ttw^\mot(\vec{k}_{1,i-1}, 2, \vec{k}_{j+1,d})$ 
			\end{tabular}		
		\end{center}
		So via \eqref{eqn:coeffzl2232}, we have \( \zeta^\lmot_\alpha(\vec{k}_{i+1,j}) \in \zeta^\lmot(\overline{2r+1}) \Q \).
		
		Likewise, from \eqref{eqn:dr:eps0}, we have \( \abs{\vec{k}_{i,j}} - r = 2 \), with \( k_j = 2 \) so \( \abs{\vec{k}_{i,j}} \) contains an odd number of 1's.  In the level-graded, it therefore must contain exactly one 1.  Be aware that we must reverse \( \vec{k}_{i,j} \) when inserting it into \( \zeta^\lmot \) in term \eqref{eqn:dr:eps0}.  We have the following case.
		\begin{center}
			\begin{tabular}{c|c|c}
				$\vec{k}_{i,j}-r$ & $ \vec{k}_{i,j} $ & Contribution to \( D_r \) \\ \hline
				2 & $(\{2\}^a,1,\{2\}^b,2)$ & $ -\zeta^\lmot_{2-2}(\{2\}^b,1,\{2\}^a) \otimes \ttw^\mot(\vec{k}_{1,i-1}, 2, \vec{k}_{j+1,d})$ 
			\end{tabular}		
		\end{center}
		Once again, the term \( \zeta^\lmot_\alpha(\vec{k}_{j-1,i}) \in \zeta^\lmot(\overline{2r+1}) \Q \) via \eqref{eqn:coeffzl2232}.
	\end{proof}
\end{Lem}

	We now look at the action of these derivations on elements of a given level, and package them together into the following linear map.

\begin{Def}\label{def:partialH}
	For all \( N ,\ell \geq 1 \), let \( \partial^H_{N,\ell} \) be the linear map
	\[
	\partial_{N,\ell}^H \colon \gr_\ell^H \mathcal{H}_N^H \to \bigoplus_{1 \leq 2r+1 \leq N} \gr_{\ell-1}^H \mathcal{H}_{N-2r-1}^H \,,
	\]
	defined by first applying \( \bigoplus_{1\leq2r+1\leq N} \gr_\ell^H D_{2r+1} \big|_{\gr_\ell^H \mathcal{H}_N^H} \), and then sending all \( \log^\mot(2) \mapsto \tfrac{1}{2}, \zeta^\lmot(2r+1) \mapsto 2^{2r-1} \), \( r > 0 \) to by the projection \begin{align*}
	\widetilde{\pi}_{2r+1} & \colon \Q \zeta^\lmot(\overline{2r+1}) \to \Q  \\
	& \begin{cases} 
	\log^\mot(2) \mapsto \frac{1}{2} \,, & \text{if \( r = 0 \)\,,}\\
	\zeta^\lmot(2r+1) \mapsto 2^{2r-1} \,, & \text {if \( r > 0 \)}
	\end{cases}
	\end{align*}
\end{Def}

The goal is to show that the maps \( \partial_{N,\ell}^H \) are injective for \( \ell \geq 1 \).  Then by recursion, we will establish the elements of level \( \ell \) are linearly independent (otherwise \(\partial_{N,\ell}^H \) would construct a non-trivial relation of strictly smaller level).

\begin{Def}[Matrix basis]\label{def:Mbas}
	Let \( \ell,N \geq 1 \), with \( N \equiv \ell \pmod*{2} \).  Define the following sets
	\begin{align*}
	B_{H,N,\ell} &\coloneqq \{ w \in \{1,2\}^\times  \mid \deg_1 w , \abs{w} = N \} \\
	B'_{H,N,\ell} &\coloneqq \{ w \in \{1,2\}^\times \mid \deg_1 w  = \ell-1, \abs{w} < N \}
	\end{align*}
	In the case \( \ell = 1 \), the set \( B'_{H,N,\ell} \) also includes the empty word (of weight 0 and level 0).  Sort both sides in reverse colexicographic order (i.e. reading right to left, largest first), with \( 1 < 2 \).
\end{Def}

A counting argument shows that \( \# B_{H,N,\ell} = \# B'_{H,N,\ell} \) for all such choices of \( \ell, N \).  These basis elements will be used to define the matrix form of the linear map \(	\partial_{N,\ell}^H \), and the claim of injectivity corresponds to non-zero determinant.

\begin{Eg}
	For \( N = 8, \ell = 2 \), we have
	\begin{align*}
	& B_{H,N,\ell} = \{ 11222 \,, 12122 \,, 21122 \,, 12212 \,, 21212 \,, 22112 \,, 12221 \,, 21221 \,, 22121 \,, 22211 \} \, \\ 
	& B'_{H,N,\ell} = \{ 1222 \,, 2122 \,, 122 \,, 2212 \,, 212 \,, 12 \,, 2221 \,, 221 \,, 21, 1 \}
	\end{align*}
\end{Eg}

\begin{Def}[Matrix of \( \partial^H_{N,\ell} \)]\label{def:Mpar}
	For \( \ell \geq 1, N \geq 1 \), with \( N \equiv \ell \pmod*{2} \), let 
	\[
	M_{H,N,\ell} \coloneqq \big(f_{w'}^w\big)_{w\in B_{H,N,\ell}, w' \in B_{H,N,\ell}'}
	\]
	be the matrix of \( \partial^H_{N,\ell} \) with respect to the bases \( B_{H,N,\ell} \) and \( B'_{H,N,\ell} \).  Here \( f_w'^w \) denotes the coefficient of \( \ttw^\mot(w') \) in \( \partial^H_{N,\ell} \ttw^\mot(w) \), and in the matrix \( w \) corresponds to rows, and \( w' \) to columns.
\end{Def}
\begin{Eg}
	For \( N = 8, \ell = 2 \), the matrix \( M_{H,8,2} \) is as follows; the first row and column label the elements of \( B'_{H,8,2} \) and \( B_{H,8,2} \) respectively.
	\begin{center}
		\small
		\begin{adjustwidth}{-0.4cm}{}
		\begin{tabular}{c||c|c|c|c|c|c||c|c|c|c}
			& \!\!1222\!\! &  \!\!2122\!\!  & 122 & \!\!2212\!\! & 212 & 12 & \!\!2221\!\!\! & 221 & 21 & 1 \\ \hline\hline
\!\!11222\!\! & $ 1 $ & $ 0 $ & $ {-}2 c_{21} $ & $ 0 $ & $ 0 $ & $ {-}8 c_{221} $ & $ 0 $ & $ 0 $ & $ 0 $ & $ 0 $ \\
\!\!12122\!\! & $ 0 $ & $ 1 $ & $ \!\!2 d_{12}{-}2 c_{21}\!\!\! $ & $ 0 $ & $ 0 $ & $ 8 c_{23}{-}8 c_{32} $ & $ 0 $ & $ 0 $ & $ 0 $ & $ 0 $ \\
\!\!21122\!\! & $ 0 $ & $ 0 $ & $ 2 d_{21} $ & $ 0 $ & $ {-}2 c_{21} $ & $ 0 $ & $ 0 $ & $ 0 $ & $ 0 $ & $ 0 $ \\
\!\!12212\!\! & $ 0 $ & $ 0 $ & $ 2 c_{21} $ & $ 1 $ & $ \!\!2 d_{12}{-}2 c_{21}\!\!\! $ & $ \!\!{-}8 c_{23}+8 c_{32}{+}8 d_{122}\!\!\! $ & $ 0 $ & $ 0 $ & $ 0 $ & $ 0 $ \\
\!\!21212\!\! & $ 0 $ & $ 0 $ & $ 0 $ & $ 0 $ & $ 2 d_{21} $ & $ 8 d_{212} $ & $ 0 $ & $ 0 $ & $ 0 $ & $ 0 $ \\
\!\!22112\!\! & $ 0 $ & $ 0 $ & $ 0 $ & $ 0 $ & $ 2 c_{21} $ & $ 8 d_{221} $ & $ 0 $ & $ 0 $ & $ 0 $ & $ 0 $ \\ \hline \hline
\!\!12221\!\! & $ {-}\frac{1}{2} $ & $ 0 $ & $ 2 c_{21} $ & $ 0 $ & $ 0 $ & $ 8 c_{221} $ & $ 1 $ & $ \!\!2 d_{12}{-}2 c_{21}\!\! $ & $ 8 d_{122}{-}8 c_{221} $ & $ \!\!32 d_{1222}\!\! $ \\
\!\!21221\!\! & $ 0 $ & $ {-}\frac{1}{2} $ & $ 0 $ & $ 0 $ & $ 2 c_{21} $ & $ 0 $ & $ 0 $ & $ \!\!2 d_{21}{-}2 c_{21}\!\! $ & $ \!\! 8 c_{23}{-}8 c_{32}{+}8 d_{212} \!\!\! $ & $ \!\!32 d_{2122}\!\! $ \\
\!\!22121\!\! & $ 0 $ & $ 0 $ & $ 0 $ & $ {-}\frac{1}{2} $ & $ 0 $ & $ 0 $ & $ 0 $ & $ 2 c_{21} $ & $\!\! {-}8 c_{23}{+}8 c_{32}{+}8 d_{221} \!\!\!$ & $ \!\!32 d_{2212}\!\! $ \\
\!\!22211\!\! & $ 0 $ & $ 0 $ & $ 0 $ & $ 0 $ & $ 0 $ & $ 0 $ & $ {-}\frac{1}{2} $ & $ 2 c_{21} $ & $ 8 c_{221} $ & $ \!\!32 d_{2221}\!\! $ 
			\end{tabular}
		\end{adjustwidth}
	\end{center}
	The entries \( 1 \) in the matrix arise from both the deconcatenation term \( 2 \ttw^\lmot(1) \) removing a leading 1 which appears in \( D_{1} \) as per \autoref{prop:d1}, the entries \( -\tfrac{1}{2} \) correspond to the deconcatenation term \( -\ttw^\lmot(1) \) removing a trailing 1 which appear in \( D_{1} \).  With the projection \( \log^\lmot(2) \to \frac{1}{2} \), these combinations give 1 and \( -\frac{1}{2} \) respectively.
	
	After substituting the values for \( c_\bullet \) and \( d_\bullet \) using the formulae above, we obtain the matrix
	\begin{center}
		\begin{tabular}{c||c|c|c|c|c|c||c|c|c|c}
		& \!\!1222\!\! &  \!\!2122\!\!  & 122 & \!\!2212\!\! & 212 & 12 & \!\!2221\!\!\! & 221 & 21 & 1 \\ \hline\hline
		\!\!11222\!\! &	$ 1 $ & $ 0 $ & $ 4 $ & $ 0 $ & $ 0 $ & $ -16 $ & $ 0 $ & $ 0 $ & $ 0 $ & $ 0 $ \\
		\!\!12122\!\! & $ 0 $ & $ 1 $ & $ -3 $ & $ 0 $ & $ 0 $ & $ -80 $ & $ 0 $ & $ 0 $ & $ 0 $ & $ 0 $ \\
		\!\!21122\!\! &	$ 0 $ & $ 0 $ & $ -7 $ & $ 0 $ & $ 4 $ & $ 0 $ & $ 0 $ & $ 0 $ & $ 0 $ & $ 0 $ \\
		\!\!12212\!\! &	$ 0 $ & $ 0 $ & $ -4 $ & $ 1 $ & $ -3 $ & $ 111 $ & $ 0 $ & $ 0 $ & $ 0 $ & $ 0 $ \\
		\!\!21212\!\! & $ 0 $ & $ 0 $ & $ 0 $ & $ 0 $ & $ -7 $ & $ 186 $ & $ 0 $ & $ 0 $ & $ 0 $ & $ 0 $ \\
		\!\!22112\!\! & $ 0 $ & $ 0 $ & $ 0 $ & $ 0 $ & $ -4 $ & $ 31 $ & $ 0 $ & $ 0 $ & $ 0 $ & $ 0 $ \\ \hline\hline
		\!\!12221\!\! &	$ -\frac{1}{2} $ & $ 0 $ & $ -4 $ & $ 0 $ & $ 0 $ & $ 16 $ & $ 1 $ & $ -3 $ & $ 15 $ & $ -127 $ \\
		\!\!21221\!\! &	$ 0 $ & $ -\frac{1}{2} $ & $ 0 $ & $ 0 $ & $ -4 $ & $ 0 $ & $ 0 $ & $ -3 $ & $ 106 $ & $ -1905 $ \\
		\!\!22121\!\! &	$ 0 $ & $ 0 $ & $ 0 $ & $ -\frac{1}{2} $ & $ 0 $ & $ 0 $ & $ 0 $ & $ -4 $ & $ 111 $ & $ -1905 $ \\
		\!\!22211\!\! & $ 0 $ & $ 0 $ & $ 0 $ & $ 0 $ & $ 0 $ & $ 0 $ & $ -\frac{1}{2} $ & $ -4 $ & $ 16 $ & $ -127 $ \\
		\end{tabular}
	\end{center}
	We notice already that the matrix is block \emph{lower} triangular; the blocks correspond to the number of trailing ones in the quotients, which is also the number of trailing 1's when deconcatenating the maximal string \( 2^a1 \) from the start of the basis words.  Each diagonal block except the last (i.e. here only the first, but in general all further intermediate ones too) is \emph{upper} triangular modulo 2, so has determinant \( \neq 0 \).  Expanding the determinant of the last block about its first column produces two matrices with integer entries, which are equivalent to triangular matrices mod 2, so the first block has determinant \( \frac{\text{odd}}{2} \).  Overall the full matrix has the same property: the determinant is in \( \frac{1}{2} + \Z \).
	
	We aim to show this is a general phenomenon for level \( \ell \geq 1 \).  In fact, we shall show that \( \partial^H_{N,\ell} \) never increases the number of trailing 1's, explaining the block triangular appearance; moreover, for a fixed number of trailing 1's we show that \( \partial^H_{N,\ell} \) acts by deconcatenation, modulo 2, explaining the upper triangular appearance of each block.  (Special care must be given for the first row, where an extra \( -\tfrac{1}{2} \) is produced.)
\end{Eg}

	We formally introduce a partition of the basis set \( B'_{H,N,\ell} \) by the number of trailing 1's.

	\begin{Def}[Trailing 1's]
	Write
	\[
		T'_{\alpha,N,\ell} \coloneqq \{ w' \in B'_{H,N,\ell} \mid \text{\( w' \) has exactly \( \alpha \) trailing 1's} \} 
	\]
	More precisely, one can define ``exactly \( \alpha \) trailing 1's'' as a word of the form \( 1^\alpha \) or \( w21^\alpha \), where \( w \in \{1,2\}^\times \), to obtain
	\[
		T'_{\alpha,N,\ell} = B'_{S,N,\ell} \cap \Big( \{ 1^\alpha \} \cup \{ w21^\alpha \mid w \in \{ 1\,2\}^\times \} \Big) \,.
	\]
	\end{Def}

	We note that \( T'_{\alpha, N, \ell} = \emptyset \) if \( \alpha \geq \ell \) as a word of level \( < \ell \) cannot contain \( \ell \) trailing 1's.  Whereas \( T'_{\ell-1, N, \ell} = \{2^{\frac{1}{2}(N-\ell)}1^{\ell-1},\ldots,21^{\ell-1},1^{\ell-1}\} \).  So we certainly have as a disjoint union that
	\[
		B'_{S,N,\ell} = \bigcup_{0 \leq \alpha M \ell} T'_{\alpha,N,\ell} \,.
	\]
	
	Now consider the bijection
	\begin{align*}
		\phi \colon B'_{S,N,\ell} &\to B_{H,N,\ell} \\
		u &\mapsto 2^a1u \,,
	\end{align*}
	where \( a \) is the unique value such that \( 2^a1u \) has weight \( N \).  We pull back the partition \( T'_{\alpha,N,\ell} \) to define 
	\[
		T_{\alpha,N,\ell} = \{ w \in B'_{H,N,\ell} \mid \phi^{-1}(w) \in T'_{\alpha,N,\ell} \} \,.
	\]
	Note that the inverse \( \phi^{-1}(w) \) is obtained by taking the suffice when deconcatenating \( w \) after the first 1.
		
	\begin{Lem}\label{lem:trailing1:agrees}
		For \( w \in B_{H,N,\ell} \), with \( w \neq 2^{\frac{1}{2}(N-\ell)} 1^\ell \) then the following holds.  The word \( w \) has \( \alpha \) trailing 1's if and only if 
		\[
			w \in T_{\alpha,N,\ell} \,.
		\]
		However, the word \( w = 2^{\frac{1}{2}(N-\ell)}1^\ell \) lies in \( T_{\ell-1,N,\ell} \).
			
		\begin{proof}
			Firstly, we check the case \( w = 2^{\frac{1}{2}(N-\ell)}1^\ell \).  Deconcatenating at the first 1 tells us that \( \phi(w) = 1^{\ell-1} \) which has exactly \( \ell-1 \) trailing 1's.  So \( w \in T_{\ell-1,N,\ell} \) as claimed.
				
			Now take any other word \( v \) of level \( \ell \).  It cannot have \( 1^\ell \) trailing 1's, so is of necessarily of the form \( v'21^{\alpha} \), with \( \alpha < \ell \), for some \( v' \in \{1,2\}^\times \), where \( \deg_1 v' \geq \ell - \alpha \geq 1 \).  This means the first 1 in \( v \) occurs somewhere in \( v' \).  Deconcatenating after this, gives a suffice of the form \( v''21^{\alpha} \), so that \( \phi(v) \in T'_{\alpha,N,\ell} \), meaning \( v \in T_{\alpha,N,\ell} \) as claimed.
			
			Conversely, given \( v \in T_{\alpha,N,\ell} \), we know that \( \phi^{-1}(v) \in T'_{\alpha,N,\ell} \), so that \( \phi^{-1}(v) = v'21^\alpha \) or \( \phi^{-1}(v) = 1^\alpha \), with \( \alpha = \ell-1 \).  The former case leads to \( v = 2^a1v'21^\alpha \) which ends in exactly \( \alpha \) trailing 1's.  The latter case leads to \( v = 2^{\frac{1}{2}(N-\ell)}1^\ell \), which we already excluded.
		\end{proof}
	\end{Lem}

	Now we claim that the map \( \partial^H_{N,\ell} \) never increases the number of trailing 1's.
	
	\begin{Lem}\label{lem:mat:blocktri}
		For all words \( w \in T_{\alpha,N,\ell} \), the image under \( \partial^H_{N,\ell} \) satisfies
		\[
			\partial^H_{N,\ell} w = \sum_{\substack{w' \in T_{\beta, N, \ell}  \\ \beta \leq \alpha }} f_{w'}^w \ttw^\mot(w') \,.
		\]
		That is to say, \( \partial^H_{N,\ell} w \) only involves words with \( \leq \alpha \) trailing 1's.
		
		\begin{proof}
			By \autoref{lem:trailing1:agrees}, we know that the set \( T_{\alpha,N,\ell} \) is characterised as the words \( w \in B_{N,\ell} \) ending with \( \alpha \) many trailing 1's, except that when \( \alpha = \ell-1 \), where also include the word \( 2^{\frac{1}{2}(N-\ell)}1^\ell \in T_{\ell-1,N,\ell} \).  (Note that since \( T'_{\ell,N,\ell} = \emptyset \), the set \( T_{\ell,N,\ell} \) would also be empty.)
			
			Purely for level-filtration reasons (see \autoref{lem:hoffmanmotivic}), the word \( w = 2^{\frac{1}{2}(N-\ell)}1^\ell \) must map to a sum of words with \( < \ell \) trailing 1's.  We may therefore assume \( \alpha \leq \ell-1 \), and \( w \in T_{\alpha,N,\ell} \) genuinely ends in \( 1^\alpha \). \medskip
			
			We consider the terms which arise in \( \partial_{N,\ell}^H \ttw^\mot(w) \) via the cases in \autoref{lem:hoffmot:form} for \( \gr_\ell D_{2r'+1} \ttw^\mot(w) \), with \( w \) viewed as a tuple.  Take take the deconcatenation term
			\[
				\ttw^\mot(\vec{k}_{1,j}) \otimes \ttw^\mot(\vec{k}_{j+1,d})
			\]
			Since \( \vec{k}_{j+1,d} \) is a suffix of \( \vec{k} \) it clearly has \( \leq \alpha \) trailing 1's.  (Either the cut \( \vec{k}_{1,j} \) ends before the first trailing 1, in which case we have exactly as many trailing 1's.  Or the cut ends after this point, and we have even reduced the number of trailing 1's as \( \vec{k}_{j+1,d} = (\{1\}^{\beta}) \) with \( \beta < \alpha \)).
			
			On the other hand, if \( r = 1 \), and we take the deconcatenation term
			\[
				\delta_{k_d=1} \ttw^\lmot(1) \otimes \ttw^\mot(k_1,\ldots,k_{d-1})
			\]
			which occurs in \( D_1 \), then we have certainly removed a trailing 1 if this term is non-zero. \medskip
			
			Now consider the cases as is the tables in \autoref{lem:hoffmot:form}, which come from replacing a subindex \( \vec{k}_{i,j} \)  by \( \abs{\vec{k}_{i,j}} - r\).  In both cases we see the replacement is by a 2.  So if \( \vec{k}_{i,j} \) ends before the trailing 1's, we do not increase their number.  Otherwise \( \vec{k}_{i,j} \) ends within the string of trailing 1's, and some set \( 1^\beta \), \( \beta > 0 \) of them are replaced by a 2, leaving strictly fewer trailing 1's, namely \( 1^{\alpha - \beta} \). \medskip
			
			The block (lower) triangularity corresponds to the ordering, wherein we have the reverse colexicographic order (reading right to left, largest first), with \( 1 < 2 \).  So \( 1^\gamma > w_021^\alpha > w_0'21^\beta \), for any \( \beta < \alpha \leq \gamma\).
		\end{proof}
	\end{Lem}

	This lemma has established that the matrix \( M_{H,N,\ell} \) of \( \partial^H_{N,\ell} \) is block (lower) triangular, and that the diagonal blocks are square.  (We constructed \( T_{\alpha,N,\ell} \) as the preimage of \( T'_{\alpha,N,\ell} \) under a bijection of the bases.)  We therefore reduction the question of injectivity of \( \partial^H_{N,\ell} \), equivalently the determinant of \( M_{H,N,\ell} \) being non-zero, to a question of understanding the determinants of these diagonal blocks.
	
	\begin{Def}
		For \( \ell \geq 1, N \geq 1, \alpha \leq \ell-1 \), with \( N \equiv \ell \pmod*{2} \), let 
		\[
		M_{\alpha,H,N,\ell} \coloneqq \big(f_{w'}^w\big)_{w\in T_{\alpha,N,\ell}, w' \in T_{\alpha,N,\ell}'}
		\]
		be the diagonal block of \( M_{H,N,\ell} \) corresponding to \( \alpha \) trailing 1's (after deconcatenating \( 2^a1 \) for \( T_{\alpha,N,\ell} \), or immediately for \( T'_{\alpha,N,\ell} \)).
	\end{Def}

	\begin{Lem}\label{lem:block:isdeconcat}
		For \( \alpha < \ell-1 \), the restriction of \( \partial_{N,\ell} \) to \( T_{\alpha,N,\ell} \), and projected to the \( T_{'\alpha,N,\ell} \) satisfies the following.  For \( w \in T_{\alpha,N,\ell} \), the coefficient of every word \( \ttw^\mot(u) \) with \( u \in T'_{\alpha,N,\ell} \), in
		\[
			\partial_{N,\ell}^H \ttw^\mot(w) - \sum_{\substack{w = uv \\ \deg_1 u = 1}} 2^{\abs{u}-2} d_u \ttw^\mot(v)
		\]
		is an even integer.
		
		\begin{proof}
			Since \( \alpha < \ell - 1 \) we know \( w \neq 2^{\frac{1}{2}(N-\ell)}1^\ell \), so we do not have to worry about the deconcatenation term \( \ttw^\mot(k_1,\ldots,k_{d-1}) \) contributing: it has strictly fewer than \( \ell - 1 \) trailing 1's.
			
			The terms in 
			\[
				\sum_{\substack{w = uv \\ \deg_1 u = 1}} 2^{\abs{u}-2} d_u \ttw^\mot(v)
			\]
			are the deconcatenation terms from \autoref{lem:hoffmot:form}, after using \( \widetilde{\pi}_{\abs{u}} = \widetilde{\pi}_{2a+2b+1}\) to project the factor \( \ttw^\lmot(u) = \ttw^\lmot(\{2\}^{a},1,\{2\}^b) = d_u \zeta^\lmot(2a+2b+1) \).
			
			The remaining terms (whether or not they have fewer trailing 1's), arise from the the \eqref{eqn:dr:0eps} and \eqref{eqn:dr:eps0} terms (or rather their images in \( \gr^H_\ell D_{2r+1} \)).  They are categorised by the cases listed in \autoref{lem:hoffmot:form}, namely
			\begin{center}
			\begin{tabular}{c|c|c}
				$\vec{k}_{i,j}-r$ & $ \vec{k}_{i,j} $ & Contribution to \( D_r \) \\ \hline
				2 & $(2,\{2\}^a,1,\{2\}^b)$ & $ \zeta_{2-2}^\lmot(\{2\}^a,1,\{2\}^b) \otimes \ttw^\mot(\vec{k}_{1,i-1}, 2, \vec{k}_{j+1,d})$  \\
				2 & $(\{2\}^a,1,\{2\}^b,2)$ & $ -\zeta^\lmot_{2-2}(\{2\}^b,1,\{2\}^a) \otimes \ttw^\mot(\vec{k}_{1,i-1}, 2, \vec{k}_{j+1,d})$ 
			\end{tabular}		
		\end{center}
		In each case, the coefficient of \( \ttw^\mot(v) \) after projecting via \( \widetilde{\pi}_{2r+1} \) has one of the following forms (depending on whether \( b = 0 \) or \( b > 0 \) in case 1, likewise case 2)
		\begin{align*}
		2^{2a-1} c_{2^a1} &= \begin{cases}
		0 & a = 0 \\
		2^{2a} (-1)^a & a  > 0
		\end{cases}\\
		2^{2a+2b+1} c_{2^a32^b} &=  (-1)^{a+b}\bigg({-} 2^{2a+2b+2} \binom{2a+2b+2}{2a+2} + (2^{2a+2b+2} - 1)\binom{2a+2b+2}{2b+1} \bigg) \,, 
		\end{align*}
		This is exactly the same claim as in \eqref{eqn:even:221} and \eqref{eqn:even:232}, so as before these coefficients are indeed even integers.
		\end{proof}
	\end{Lem}

	In the same manner as \autoref{thm:matS:inj}, it follows that the diagonal block \( M_{\alpha,H,N,\ell} \) corresponding to \( T_{\alpha, N,\ell} \) and \( T'_{\alpha, N,\ell} \) is invertible, for \( \alpha < \ell-1 \).

	\begin{Prop}\label{prop:diagalpha:inv}
		The diagonal block \( M_{\alpha,H,N,\ell} \) is invertible for \( \alpha < \ell-1 \), in particular it has non-zero determinant.
		
		\begin{proof}
			Recall the bijection \( \phi \) between the bases, sending \( u \in B'_{H,N,\ell} \) to \( 2^a1u \), for the unique \( a \) such that \( 2^a1u \) has weight \( N \).  This map defined the set \( T_{\alpha,N,\ell} \) as the preimage of \( T'_{\alpha,N,\ell} \).  
			
			By the previous lemma \autoref{lem:block:isdeconcat}, we know that \( M_{\alpha,H,N,\ell} \) is upper triangular modulo 2, the terms above the diagonal arising from the deconcatenation terms.  That is to say: The diagonal entries of \( M_{\alpha,H,N,\ell} \) are of the form \( f_{u}^{2^r1u} = 2^{2a-1}d_{2^a1} + 2n \), \( n \Z\).  The only other non-even entries in the column indexed by \( u \in B'_{\alpha,H,N,\ell} \) occur for rows indexed by \( w = 2^a12^bu \), however since \( 2^r1u < 2^a12^bu \), these occur above the diagonal.
			
			As the diagonal terms are given by 
			\[
			2^{2a-1}d_{2^a1} = (-1)^{a}(2^{2a+1} - 1) \equiv 1 \pmod*{2} \,,
			\]
			we see that \( f_u^{2^r1u} \equiv 1 \pmod*{2} \).  Therefore \( M_{\alpha,H,N,\ell} \) is upper triangular, modulo 2, and has 1's on the diagonal.  It therefore has determinant \( \equiv 1 \pmod*{2} \), and so \( M_{\alpha,H,N,\ell} \) has non-zero determinant.
		\end{proof}
	\end{Prop}

	We now turn to the case \( \alpha = \ell - 1 \).  In this case the following modification of \autoref{lem:block:isdeconcat} holds.

	\begin{Lem}\label{lem:block:isdeconat1}
				For \( \alpha = \ell-1 \), the restriction of \( \partial_{N,\ell} \) to \( T_{\alpha,N,\ell} \), and projected to the \( T_{'\alpha,N,\ell} \) satisfies the following.  For \( w \in T_{\alpha,N,\ell} \), the coefficient of every word \( \ttw^\mot(u) \) with \( u \in T'_{\alpha,N,\ell} \), in
		\[
		\partial_{N,\ell}^H \ttw^\mot(w) - \sum_{\substack{w = uv \\ \deg_1 u = 1}} 2^{\abs{u}-2} d_u \ttw^\mot(v) + \frac{1}{2} \ttw^\mot(2^{\frac{1}{2}(N-\ell)}1^{\ell-1})
		\]
		is an even integer.
		
		\begin{proof}
			The proof of \autoref{lem:block:isdeconcat} goes through unchanged, except that we must also consider the case \( w \neq 2^{\frac{1}{2}(N-\ell)}1^\ell \).  Even for this, the argument about the deconcatenation \eqref{eqn:dr:deconcat} and other terms in \autoref{lem:hoffmot:form} goes through unchanged.
			
			The only additional term we must consider is the term arising from deconcatenating a trailing 1, namely \( {-}\log^\lmot(1) \otimes \ttw^\mot(k_1,\ldots,k_{d-1}) \) which appears in \autoref{prop:d1}.  This is the additional term above, and so the proof is complete.
		\end{proof}
	\end{Lem}

	We note now that this additional term occurs in the first column, last row of the matrix \( M_{\alpha,H,N,\ell} \) because the word indexing the column is \( 2^{\frac{1}{2}(N-\ell)}1^{\ell-1} > \cdots  21^{\ell-1} > 1^{\ell-1} \) while the word indexing the row is \( 2^{\frac{1}{2}(N-\ell)}1^{\ell} < w21^{\ell-1} \), for any \( w \).

	Finally we can show that the diagonal block \( M_{\ell-1, H, N, \ell} \) is also invertible, or equivalently has non-zero determinant.
	
	\begin{Prop}\label{prop:diagell:inv}
		The diagonal block \( M_{\alpha,H,N,\ell} \) is invertible for \( \alpha = \ell-1 \), in particular it has non-zero determinant.
		
		\begin{proof}
			The above observation tells us that the first column of the matrix \( M_{\alpha,H,N,\ell} \), \( \alpha = \ell - 1 \), consists of a single entry \( \frac{1}{2} \) at the bottom, a single entry \( \frac{1}{2} d_{1} = 1 \) at the top, and (potentially) a number of even entries.  However, since this column is indexed by \( 2^{\frac{1}{2}(N-\ell)}1^{\ell-1} \), this column corresponds to the computation of \( \gr_\ell D_{1} \).  Therefore there are no other entries in this column since \( \zeta_0^\lmot(1) = 0 \). Now expand out the determinant about this column.  \medskip
				
			The minor \( A_{1,1} \) corresponding to the \( (1,1) \) entry of \( M_{\alpha,H,N,\ell} \) is again an upper triangular matrix modulo 2, as it arises from deleting the first row and column of \( M_{\alpha, H, N, \ell} \).  That matrix is  itself integral and upper triangular modulo 2, after removing the entry \( -\frac{1}{2} \) in the first column (compare the argument in \autoref{lem:block:isdeconcat} and \autoref{lem:block:isdeconat1}).  So as entry \( -\frac{1}{2} \) plays no role in the \( (1,1) \) cofactor, the integrality and upper triangularity modulo 2 holds.  Likewise the diagonal entries are equal to 1, and so we find \( C_{1,1} = \det A_{1,1} \equiv 1 \pmod*{2} \).  This means \( C_{1,1} = 2x + 1 \in \Z \) is an odd integer.\medskip
			
			The minor \( A_{1,\frac{1}{2}(N-\ell)+1} \) corresponding to the bottom entry of the first column, is given by the an explicit formula modulo 2.  We note the rows are indexed by \( 2^{a-1}12^{\frac{1}{2}(N-\ell)+1-a}1^{\ell-1} \), for \( 1 \leq a \leq  \frac{1}{2}(N-\ell) \), so that \( 2^{\frac{1}{2}(N-\ell)}1^\ell \) is avoided.  Likewise the columns are indexed by \( 2^{\frac{1}{2}(N-\ell) - c}1^{\ell-1} \), for \( 1 \leq c \leq \frac{1}{2}(N-\ell) \), so that \( 2^{\frac{1}{2}(N-\ell)}1^{\ell-1} \) is avoided.  This means that modulo 2, the minor is given by
			\begin{align*}
				\big(A_{1,\frac{1}{2}(N-\ell)+1} \big)_{a,c=1}^{\frac{1}{2}(N-\ell)} &= 2^{2c - 1} d_{2^{a-1}12^{c+1-a}} \\
				&= (-1)^{c} (2^{2c+1} - 1) \binom{2c}{2a-2} \\
				&\equiv \binom{2c}{2a-2} \pmod{2}
			\end{align*}
			We notice the following: when \( a = c+1 \), 
			\[
				\binom{2c}{2a-2} = 1 \,
			\]
			so the minor has 1's on the subdiagonal.  When \( a = 1 \)
			\[
				\binom{2c}{2a-2} = 1 \,,
			\]
			so the minor has 1's in the first row.  Summing down each column we note that
			\[
				\sum_{a=2}^{\frac{1}{2}(N-\ell)} \binom{2c}{2a-2} = \begin{cases}
					2^{2c-1} - 1 & 1 \leq c < \frac{1}{2}(N-\ell) \\
					2^{2c-1} - 2 & c = \frac{1}{2}(N-\ell) \,,
				\end{cases}
			\]
			since in the latter case the term \( \binom{2c}{2c} \) on the subdiagonal is not part of the matrix.  This means that if we subtract the sum of the remaining rows of the minor from the first row, we obtain a single 1 in the final column.  This establishes that the modulo 2 the minor is equivalent a permutation of an upper triangular matrix with 1's on the diagonal (move the last column to the start).  Hence \( C_{1,\frac{1}{2}(N-\ell)+1} = \det(A_{1,\frac{1}{2}(N-\ell)+1}) \equiv 1 \pmod*{2} \), and since \( A_{1,\frac{1}{2}(N-\ell)+1} \) actually has integer entries, we have \( C_{1,\frac{1}{2}(N-\ell)+1} = 2y  +1 \in \Z \) is an odd integer. \medskip
			
			Finally, we assemble the determinant of \( M_{\alpha,H,N,\ell} \) to be
			\begin{align*}
				& (-1)^{1 + 1} \cdot C_{1,1} + (-1)^{(\frac{1}{2}(N-\ell)+1) + 1} C_{1,\frac{1}{2}(N-\ell)+1} \\
				&= (2x + 1) + \frac{1}{2} (2y + 1) \\
				&= (2x + y + 1) + \frac{1}{2} \,.
			\end{align*}
			In particular it is in \( \frac{1}{2} + \Z \), and so cannot be 0.
		\end{proof}		
	\end{Prop}
	
	From these two propositions follows immediately the invertibility of the whole matrix \( M_{H,N,\ell} \).
	
	\begin{Cor}\label{thm:matH:inj}
		The matrix \( M_{H,N,\ell} \) is invertible.
		
		\begin{proof}
			The matrix \( M_{H,N,\ell} \) is block upper triangular by \autoref{lem:mat:blocktri} and the discussion thereafter.  The diagonal blocks are invertible square matrices by \autoref{prop:diagalpha:inv} and \autoref{prop:diagell:inv}, hence \( M_{H,N,\ell} \) itself is invertible.
		\end{proof}
	\end{Cor}

	\begin{Cor}\label{cor:hoff:indep}
		The Hoffman one-two elements
		\[
			\{ \ttw^\mot(k_1,\ldots,k_d) \mid k_i \in \{1,2\} \} 
		\]
		are linearly independent.

		\begin{proof}
			We proceed by induction on the level, as in \cite[Theorem 7.4]{brown12}, \cite[Corollary 38]{murakami21}.  The elements of level \( \ell = 0 \) are of the form \( \ttw^\mot(\{2\}^n) \), which are linearly independent because weight is a grading on \( \mathcal{H}^{(2)} \).  Now suppose the elements
			\[
			\{ \ttw^\mot(w) \mid w \in \{1,2\}^\times \,, \deg_1 w \leq \ell -1  \} \,,
			\]
			of level \( \leq \ell - 1 \) are linearly independent.  Since weight is a grading on \( \mathcal{H}^{(2)} \), any non-trivial linear relation between elements of level \( \ell \) can be assumed as homogeneous of some weight \( N \).  By \autoref{thm:matH:inj}, the map \( \partial^H_{N,\ell} \) is injective as the matrix of the map is invertible.  Application of \( \partial_H^{N,\ell} \) to a non-trivial linear relation between level \( \ell \) elements produces a non-trivial linear relation of strictly smaller level, which does not exist by the induction assumption.  So the elements of level \( \ell \) are also linearly independent, which completes the proof by induction.
		\end{proof}
	\end{Cor}

	\begin{Cor}\label{cor:text:eq:H2}
		The elements 
		\[
		\{ \ttw^\mot(w) \mid w \in \{1,2\}^\times \} \,,
		\]
		form a basis for the space of:
		\begin{itemize} 
			\item[i)] motivic extended shuffle-regularised multiple \( t \) values,
			\item[ii)] alternating (shuffle-regularised) motivic multiple zeta values
		\end{itemize}
		In particular these spaces agree, and extended shuffle-regularised motivic multiple \( t \) values have dimension \( F_{N+1} \) in weight \( N \), where \( F_{k} = F_{k-1} + F_{k-2} \) with \( F_0 = F_1 = 1 \) is the sequence of Fibonacci numbers.
		
		\begin{proof}
			From their definition as sums of alternating motivic MZV's, we know the following inclusion holds
			\[
				\mathcal{T}^\mathrm{ext}_N \subset \mathcal{H}_{N}^{(2)} \,,
			\]
			where \( \mathcal{T}_N^\mathrm{ext} \) denote the space of all shuffle-regularised motivic MtV's of weight \( N \).  However the upper bound \( \dim_\Q 	\mathcal{H}^{(2)} \leq F_{N+1} \) is established in \cite{deligneGoncharov} (in fact already \( = \)), and the lower bound \(  F_{N+1} \leq \mathcal{T}_N^\mathrm{ext}  \) from the explicit collection of independent elements shows that all of these inclusions are equalities and the dimensions is exactly \( F_{N+1} \) in weight \( N \).
		\end{proof}
	\end{Cor}

	\subsection{Stuffle regularised Hoffman one-two elements}

	We now wish to extend the independence result on the Hoffman one-two elements from the case of shuffle regularised MtV's to the more natural case of stuffle regularised MtV's.  On the motivic level, we shall do this by viewing \autoref{prop:tstVtotshandz111} as a definition.
	
	\begin{Def}[Motivic \( \ttw^{\mot,\ast,V} \)]\label{def:tst:mot}
		Let \( \vec{k} = (k_1,\ldots,k_d) \), such that \( k_d \neq 1 \).  Then the stuffle regularised motivic MtV with \( \ttw^{\mot,\ast,V}(1) = 2V \in \mathcal{H}^{(2)} \) is defined by
		\begin{equation}
		\ttw^{\mot,\ast,V}(\vec{k}, \{1\}^\alpha) \coloneqq \sum_{i=0}^\alpha \ttw^{\mot}(\vec{k},\{1\}^{\alpha-i}) \cdot \zeta^{\mot,\ast,2V-\log^\mot(2)}(\{1\}^{i}) \,,
		\end{equation}
		where
		\(
			\zeta^{\mot,\ast,U}(\{1\}^i) 
		\)
		is given by the coefficient of \( u^i \) in
		\[
			\exp\Big(U u - \sum_{n=2}^\infty \frac{(-1)^n}{n} \zeta^\mot(n) u^n \Big) \,.
		\]
	\end{Def}

	Regardless then of the technicalities of defining a stuffle-regularisation on the motivic level, one knows that
	\[
		\per 	\ttw^{\mot,\ast,V}(\vec{k}, \{1\}^\alpha) = 	\ttw^{\ast,V}(\vec{k}, \{1\}^\alpha) \,.
	\]
	In particular
	\begin{align*}
		\per \ttw^{\mot,\ast,V}(1) &= \ttw^{\ast,V}(1) + \zeta^{\ast,2V-\log(2)}(1)  = \log(2) + (2V - \log(2)) \\
		& = 2 V = 2 t^{V,\ast}(1) = \ttw^{V,\ast}(1) \,,
	\end{align*}
	as per the definition of \( \ttw \).  So \( \ttw^{\mot,\ast,V} \) corresponds to the regularisation of \( t^{\ast,V}(1) = V \).  Therefore any linear independence and basis results will successfully translate over to the classical real valued versions as spanning set results, along with whatever identities we establish motivically. \medskip

	Naturally the question of how to compute \( D_{2r+1} \ttw^{\mot,\ast,V}(\vec{k}, \{1\}^\alpha) \) now arises, but for this we appeal again to the derivation property of \( D_{2r+1} \), namely
	\[
		D_{2r+1} (X Y) = (1 \otimes Y) D_{2r+1} X + (1 \otimes X) D_{2r+1} Y \,.
	\]

	We first give a lemma about the action of \( D_{2r+1} \) on \( \zeta^{\mot,\ast,U}(\{1\}^i) \).
	
	\begin{Lem}
		The action of \( D_{2r+1} \) on \( \zeta^{\mot,\ast,U}(\{1\}^i) \), with \( U \in \mathcal{H}^{(2)} \), is given by
		\[
			D_{2r+1} \zeta^{\mot,\ast,U}(\{1\}^i) = \begin{cases}
				\zeta^{\lmot,\ast,U}(\{1\}^{2r+1}) \otimes \zeta^{\mot,\ast,U}(\{1\}^{i - (2r+1)}) & \text{if \( i \geq 2r+1 \) \,, } \\
				0 & \text{otherwise} \,.
			\end{cases}
		\]
		
		\begin{proof}
			The case \( 2r+1 > i \) is clear, as in this case there can be no (non-zero) weight \( i - (2r+1) \) factor in the right hand factor of \( D_{2r+1} \).  So we assume \( 2r+1 \leq i \).
			
			We know that \( D_{2r+1} \zeta^\mot(s) = \delta_{s = 2r+1} \zeta^\lmot(2r+1) \otimes 1 \) since \( \zeta^\mot(s) \) is primitive for the coaction \( \Delta \).  More generally, if \( k_i \neq 2r+1 \), for any \( 1 \leq i \leq n \), then
			\begin{align*}
				& D_{2r+1} \zeta^\mot(2r+1)^\ell \zeta^\mot(k_1) \cdots \zeta^\mot(k_n) \\
				& {}  =  \zeta^\lmot(2r+1) \otimes \ell \zeta^\mot(2r+1)^{\ell-1}  \zeta^\mot(k_1) \cdots \zeta^\mot(k_n) \,.
			\end{align*}
			So, when acting on a polynomial \( p(\zeta^\mot(2r+1), \zeta^\mot(k_1),\ldots,\zeta^\mot(k_n)) \) in single motivic zeta values, the right hand tensor factor is (formally) the derivative of \( p(\zeta^\mot(2r+1), \zeta^\mot(k_1),\ldots,\zeta^\mot(k_n)) \) with respect \( \zeta^\mot(2r+1) \), and the left hand tensor factor is simply \( \zeta^\lmot(2r+1) \).
			
			More rigorously, the right hand factor of action of \( D_{2r+1} \) mimics the action of \( \frac{\mathrm{d}}{\mathrm{d}z_{2r+1}} \) on the polynomial \( p(z_{2r+1}, z_{k_1}, \ldots, z_{k_n}) \), under the correspondence \( \zeta^\mot(m) \leftrightarrow z_m \).  So we are justified now in proceeding via this formal derivative with respect to \( \zeta^\mot(2r+1) \). \medskip
			
			If \( 2r+1 > 1 \), applying this formal differentiation operation to
			\begin{equation}\label{eqn:z111:genseries}
				\sum_{i=0}^\infty \zeta^{\mot,\ast,U}(\{1\}^i) u^i = \exp\Big(U u - \sum_{n=2}^\infty \frac{(-1)^n}{n} \zeta^\mot(n) u^n \Big) \,,
			\end{equation}
			viewed as a generating series, leads to the following (extending \( D_{2r+1} \) by linearity to the coefficients of a power series):
			\begin{align*}
				& D_{2r+1} \sum_{i=0}^\infty \zeta^{\mot,\ast,U}(\{1\}^i) u^i \\
				& = \zeta^\lmot(2r+1) \otimes \frac{\mathrm{d}}{\mathrm{d}\zeta^\mot(2r+1)} \exp\Big(U u - \sum_{n=2}^\infty \frac{(-1)^n}{n} \zeta^\mot(n) u^n \Big) \\
				& = \zeta^\lmot(2r+1) \otimes \frac{(-1)^{2r}}{2r+1} u^{2r+1} \exp\Big(U u - \sum_{n=2}^\infty \frac{(-1)^n}{n} \zeta^\mot(n) u^n \Big)
			\end{align*}
			So by comparing the coefficient of \( u^i \) on both sides, we obtain
			\[
				D_{2r+1} \zeta^{\mot,\ast,U}(\{1\}^i) = \frac{1}{2r+1}\zeta^\lmot(2r+1) \otimes \zeta^{\mot,\ast,U}(\{1\}^{i-(2r+1)}) \,.
			\]
			It remains to note that
			\[
				\zeta^{\lmot,\ast,U}(\{1\}^{2r+1}) = \frac{1}{2r+1} \zeta^\lmot(2r+1) \,,
			\]
			by extracting the irreducible contribution in \eqref{eqn:z111:genseries}. \medskip
			
			The corresponding result holds for \( 2r+1 = 1 \), mutatis mutandis, by the view that \( \zeta^{\mot,\ast,U}(1) = U \in \mathcal{H}^{(2)} \).  So in particular \( \zeta^{\mot,\ast,U}(1) \) is some rational multiple \( \lambda \zeta^\mot(\overline{1}) \) of \( \zeta^\mot(\overline{1}) = \log^\mot(2) \), and so primitive for the coaction.  Namely \( D_1 U = D_1 \lambda \log^\mot(2) = \lambda (\log^\lmot(2) \otimes 1) = (U)^\lmot \).  Likewise, \( \zeta^{\lmot,\ast,U}(1) = (U)^\lmot \), so the left hand tensor factor is also just \( \zeta^{\lmot,\ast,U}(1) \) in this case.
		\end{proof}
	\end{Lem}

	Then we compute the derivation \( D_{2r+1} \) on the stuffle-regularised motivic MtV's as follows.  We claim it is given by the essentially same formula as in \autoref{prop:dkt}, with \( \ttw^\bullet \) replaced by \( \ttw^{\bullet,\ast,V} \), and a (potential) additional term deconcatenating 1's from the end.
	
	\begin{Prop}[Derivation \( D_r \) on \( \ttw^{\mot,\ast,V} \)]\label{prop:dktstuffle}
	Let \( \vec{k} = (k_1,\ldots,k_d) \in (\Z_{\geq1})^d \) be an index.  Write \( \vec{k}_{i,j} = (k_i, \ldots, k_j) \) for a subindex of \( \vec{k} \) and \( \abs{(a_1,\ldots,a_r)} = a_1 + \cdots + a_r \) for the total (weight) of an index.  Then the derivation \( D_r \), \( r \) odd, is computed on the stuffle regularised \( \ttw^{\mot,\ast,V} \) as follows
	\begin{align}
	& D_r \big( \ttw^{\mot,\ast,V}(k_1,\ldots,k_d) \big) = \notag \\
	& \label{eqn:drst:deconcat} \sum_{1 \leq j \leq d} \delta_{\abs{\vec{k}_{1,j}}=r} \ttw^{\lmot}(k_1,\ldots,k_j) \otimes \ttw^{\mot,\ast,V}(k_{j+1},\ldots,k_d) \\[1ex]
	& \label{eqn:drst:0eps} + \!\! \sum_{1 \leq i < j \leq d} \begin{aligned}[t] 
	\delta_{\abs{\vec{k}_{i+1,j}} \leq r < \abs{\vec{k}_{i,j}}-1} \Big( \zeta^\lmot_{r-\abs{\vec{k}_{i+1,j}}}&(k_{i+1},\ldots,k_j) - \delta_{r=1} \log^\lmot(2) \Big) \\
	&  {} \otimes \ttw^{\mot,\ast,V}(k_1,\ldots,k_{i-1},\abs{\vec{k}_{i,j}} - r, k_{j+1},\ldots,k_d) \end{aligned} \\[1ex]
	& \label{eqn:drst:eps0} - \!\! \sum_{1 \leq i < j \leq d} \begin{aligned}[t]
	\delta_{\abs{\vec{k}_{i,j-1}} \leq r < \abs{\vec{k}_{i,j}}-1} \Big( \zeta^\lmot_{r-\abs{\vec{k}_{i,j-1}}}&(k_{j-1},\ldots,k_i) - \delta_{r=1} \log^\lmot(2) \Big) \\
	& {} \otimes \ttw^{\mot,\ast,V}(k_1,\ldots,k_{i-1},\abs{\vec{k}_{i,j}} - r, k_{j+1},\ldots,k_d) \end{aligned} \\[0.5ex]
	& \label{eqn:drst:deconcatend} {} + \delta_{(k_{d-r},\ldots,k_d) = (1,\ldots,1)} \,\cdot\, \zeta^{\lmot,\ast,2V-\log^\mot(2)}(\{1\}^{r}) \otimes \ttw^{\mot,\ast,V}(k_1,\ldots,k_{d-r})
	\end{align}		
	
	\begin{proof}
		We treat this based on the number of trailing 1's in the \( \vec{k} \).  Write \( \vec{k} = (k_1,\ldots,k_{d-\alpha}, \{1\}^\alpha) \), with \( k_{d-\alpha} \neq 1 \), and apply the derivation property to
		\begin{align*}
			D_{r} \ttw^{\mot,\ast,V}(k_1,\ldots,k_{d-\alpha}, \{1\}^\alpha) 
			& {} =  \sum_{\ell=0}^\alpha D_r \ttw^{\mot}(k_1,\ldots,k_{d-\alpha},\{1\}^{\alpha-\ell}) \cdot \zeta^{\mot,\ast,2V-\log^\mot(2)}(\{1\}^{\ell}) \\
			& {} =  \sum_{\ell=0}^\alpha \begin{aligned}[t] \Big\{
				 & \big( 1 \otimes \zeta^{\mot,\ast,2V-\log^\mot(2)}(\{1\}^{\ell}) \big) \cdot D_r \ttw^{\mot}(k_1,\ldots,k_{d-\alpha},\{1\}^{\alpha-\ell}) \\[-1ex]
				& {} + \big( 1 \otimes \ttw^{\mot}(k_1,\ldots,k_{d-\alpha},\{1\}^{\alpha-\ell}) \big) \cdot D_r \zeta^{\mot,\ast,2V-\log^\mot(2)}(\{1\}^{\ell}) \Big\} \,. \end{aligned}
		\end{align*}
		
		We compute the second term of the sum to be
		\begin{align*}
			& (1 \otimes \ttw^{\mot}(k_1,\ldots,k_{d-\alpha},\{1\}^{\alpha-\ell})) \cdot D_r \zeta^{\mot,\ast,2V-\log^\mot(2)}(\{1\}^{\ell}) \\
			& {} = \big( 1 \otimes \ttw^{\mot}(k_1,\ldots,k_{d-\alpha},\{1\}^{\alpha-\ell}) \big) \cdot \big( \delta_{r \leq \ell} \zeta^{\mot,\ast,2V-\log^\mot(2)}(\{1\}^{r}) \otimes \zeta^{\mot,\ast,2V-\log^\mot(2)}(\{1\}^{\ell-r})\big)  \\
			& {} = \zeta^{\mot,\ast,2V-\log^\mot(2)}(\{1\}^{r}) \otimes \big( \delta_{r \leq \ell}\ttw^{\mot}(k_1,\ldots,k_{d-\alpha},\{1\}^{\alpha-\ell}) \zeta^{\mot,\ast,2V-\log^\mot(2)}(\{1\}^{\ell-r}) \big) \,.
		\end{align*}
		The sum \( \sum_{\ell=0}^\alpha \) then restricts to \( \sum_{\ell=r}^\alpha \) because of the Kronecker delta, so we find
		\begin{align*}
		& \sum_{\ell=0}^\alpha (1 \otimes \ttw^{\mot}(k_1,\ldots,k_{d-\alpha},\{1\}^{\alpha-\ell})) \cdot D_r \zeta^{\mot,\ast,2V-\log^\mot(2)}(\{1\}^{\ell}) \Big) \\
		& = \zeta^{\mot,\ast,2V-\log^\mot(2)}(\{1\}^{r}) \otimes \sum_{\ell=r}^\alpha \ttw^{\mot}(k_1,\ldots,k_{d-\alpha},\{1\}^{\alpha-\ell}) \zeta^{\mot,\ast,2V-\log^\mot(2)}(\{1\}^{\ell-r}) \big) \\
		&= \delta_{r \leq \alpha} \zeta^{\mot,\ast,2V-\log^\mot(2)}(\{1\}^{r}) \otimes \ttw^{\mot,V,\ast}(k_1,\ldots,k_{d-\alpha},\{1\}^{\alpha-r})
		\end{align*}
		This gives the last term \eqref{eqn:drst:deconcatend}.
		
		Now consider the first term of the sum.  We need to apply the previous formula from \autoref{prop:dkt} for \( D_r \ttw^\mot(k_1,\ldots,k_{d-\alpha},\{1\}^\alpha) \).  We obtain
		\begin{equation}\label{eqn:drst:aslsum}
		\begin{aligned}
		 & \sum_{\ell=0}^\alpha  \big( 1 \otimes \zeta^{\mot,\ast,2V-\log^\mot(2)}(\{1\}^{\ell}) \big) \cdot D_r \ttw^{\mot}(k_1,\ldots,k_{d-\alpha},\{1\}^{\alpha-\ell}) = {} \\
			& \sum_{\ell=0}^\alpha \bigg\{  
			\begin{aligned}[t] & \sum_{1 \leq j \leq d-\ell} \delta_{\abs{\vec{k}_{1,j}}=r} \ttw^{\lmot}(k_1,\ldots,k_j) \otimes \ttw^{\mot}(k_{j+1},\ldots,k_{d-\ell}) \zeta^{\mot,\ast,2V-\log^\mot(2)}(\{1\}^{\ell}) \\[1ex]
			&  + \!\! \sum_{1 \leq i < j \leq d-\ell} \begin{aligned}[t] 
			& \delta_{\abs{\vec{k}_{i+1,j}} \leq r < \abs{\vec{k}_{i,j}}-1} \Big( \zeta^\lmot_{r-\abs{\vec{k}_{i+1,j}}}(k_{i+1},\ldots,k_j) - \delta_{r=1} \log^\lmot(2) \Big) \\
			&  \hspace{1em} {} \otimes \ttw^{\mot}(k_1,\ldots,k_{i-1},\abs{\vec{k}_{i,j}} - r, k_{j+1},\ldots,k_{d-\ell})  \zeta^{\mot,\ast,2V-\log^\mot(2)}(\{1\}^{\ell}) \end{aligned} \\[1ex]
			&  - \!\! \sum_{1 \leq i < j \leq d-\ell} \begin{aligned}[t]
			& \delta_{\abs{\vec{k}_{i,j-1}} \leq r < \abs{\vec{k}_{i,j}}-1}  \Big( \zeta^\lmot_{r-\abs{\vec{k}_{i,j-1}}}(k_{j-1},\ldots,k_i) - \delta_{r=1} \log^\lmot(2) \Big) \\
			& \hspace{1em} {} \otimes \ttw^{\mot}(k_1,\ldots,k_{i-1},\abs{\vec{k}_{i,j}} - r, k_{j+1},\ldots,k_{d-\ell}) \zeta^{\mot,\ast,2V-\log^\mot(2)}(\{1\}^{\ell}) \bigg\} \end{aligned}
			\end{aligned}
		\end{aligned}
		\end{equation}
		The sum over \( \ell \) and over \( i < j \) (respectively \( j \)) interchange as follows
		\begin{align*}
			\sum_{\ell = 0}^\alpha \sum_{1 \leq i < j \leq d-\ell} & = \sum_{1 \leq i < j \leq d} \sum_{\ell = 0}^{\min(\alpha,d-j)} \\
			\sum_{\ell = 0}^\alpha \sum_{1 \leq j \leq d-\ell} & = \sum_{1 \leq j \leq d} \sum_{\ell = 0}^{\min(\alpha,d-j)} \,.
		\end{align*}
		We than note that the upper bound of the \( \ell \)-summation is given exactly by the total number of trailing 1's contained across the \( \ttw^\mot \) and \( \zeta^{\mot,\ast,2V-\log^\mot(2)} \) arguments.  No further 1's can be introduced, as \( \abs{\vec{k}_{i,j}} - r > 1 \) which follows immediately from the inequality \( r < \abs{\vec{k}_{i,j}} - 1 \) in the Kronecker deltas.  Then if, for example, \( j = d - \alpha \), then \( k_{j+1} = 1 \) while \( k_{j} = k_{d-\alpha} \neq 1 \), so the subindex \( (k_{j+1}, \ldots, k_{d-\ell}) = (k_{d-\alpha+1}, \ldots, k_{d-\ell}) \) consists of \( \alpha - \ell \)  many 1's.  And indeed \( \alpha - \ell \) from \( \ttw^\mot \) and \( \ell \) from \( \zeta^{\mot,\ast,2V-\log^\mot(2)} \) give \( \alpha \) overall, equal to \( \min(\alpha, d-j) = \alpha \).  Whereas if \( j = d-\alpha+1 \), we already remove the first \( 1 = k_{j} \) from the subindex, leaving \( \alpha - 1 - \ell \)  many 1's in \( \ttw^\mot \) and \( \alpha-1 \) overall, agreeing with \( \min(\alpha, d-j) = \alpha - 1 \).
		
		This means that after summing over \( \ell \) we obtain the corresponding \( \ttw^{\mot,\ast,V} \) value in each case, namely
		\[
			\sum_{\ell=0}^{\min(\alpha,d-j)} \ttw^{\mot}(k_{j+1},\ldots,k_{d-\ell}) \zeta^{\mot,\ast,2V-\log^\mot(2)}(\{1\}^{\ell}) = \ttw^{\mot,\ast,V}(k_{j+1},\ldots,k_{d-\ell}) \,,
		\]
		and likewise for the other terms, as per \autoref{def:tst:mot}.  The sum over \( \ell \) only affects the right hand tensor factors, as the left hand ones are independent of \( \ell \), so we readily obtain the remaining terms \eqref{eqn:drst:deconcat}, \eqref{eqn:drst:0eps} and \eqref{eqn:drst:eps0} from the three summands in \eqref{eqn:drst:aslsum}.  This completes the proof.
	\end{proof}
	\end{Prop}

	Now fix \( V = \lambda \log^\mot(2) \), \( \lambda \in \Q \).  One can then proceed in the same way as \autoref{lem:hoffmanmotivic} and \autoref{lem:hoffmot:form} to conclude that the Hoffman-stuffle filtration
	\begin{align*}
		\mathcal{H}^{H,\ast} &\coloneqq \langle t^{\mot,\ast,\lambda \log^\mot(2)}(w) \mid w \in \{ 1,2 \}^\times \rangle_\Q \\
		H_{\ell,\ast} \mathcal{H}^{H,\ast} &\coloneqq \langle t^{\mot,\ast,\lambda \log^\mot(2)}(w) \mid w \in \{ 1,2 \}^\times \,, \text{ s.t. \( \deg_1 w \leq \ell \) } \rangle_\Q \,.
	\end{align*}
	is motivic, of a particular form.  Namely 
	\[
		\gr_\ell^{H,\ast} D_{2r'+1}( \gr_\ell^{H,\ast} \mathcal{H}^{H,\ast} ) \subseteq \zeta^\lmot(\overline{2r'+1}) \Q \otimes_\Q \gr_{\ell-1}^{H,\ast} \mathcal{H}^{H,\ast} 
	\]
	In particular: terms \eqref{eqn:drst:deconcat},\eqref{eqn:drst:0eps},\eqref{eqn:drst:eps0} give exactly the same contributions as previously, except for replacing \( \ttw^\mot \) with \( \ttw^{\mot,\ast,V} \) in the \emph{right} hand tensor factor.  Finally
	\[
		\zeta^{\lmot,\ast,2V-\log^\mot(2)}(\{1\}^{2r'+1}) = \begin{cases}
			(2\lambda - 1) \log^\lmot(2)  & \text{ if \( 2r'+1 = 1 \)} \\
			\frac{1}{2r'+1} \zeta^\lmot(2r'+1) & \text{ if \( 2r'+1 > 1 \) \,,} \\
		\end{cases}
	\]
	and the right hand factor \( \delta_{(k_{d-r},\ldots,k_d) = (1,\ldots,1)}  \ttw^{\mot,\ast,V}(k_1,\ldots,k_{d-r}) \) is obtained by removing \( 2r'+1 \) ones from the end of the \( t \) value.
	So the contribution from \eqref{eqn:drst:deconcatend} lands in the space \( \zeta^\lmot(\overline{2r'+1}) \Q \otimes_{\Q} \gr_{\ell-1}^{H,\ast} \mathcal{H}^{H,\ast} \).  In fact, if \( 2r'+1 > 1 \), we remove at least 3 ones, and so reduce the level by 3, which vanishes in \( \gr_{\ell-1}^{H,\ast} \mathcal{H}^{H,\ast} \).  Whereas, if \( 2r'+1 = 1 \), this term contributes
	\[
		\delta_{k_d=1} (2\lambda - 1) \log^\lmot(2)  \otimes \ttw^{\mot,\ast,V}(k_1,\ldots,k_{d-1}) \,.
	\]
	Note also, this is the only place the regularisation parameter \( \lambda \) enters the calculation.
	
	In particular, the additional \( \delta_{k_d=1}  (2\lambda - 1) \widetilde{\pi}^{1} \big( \log^\lmot(2) \big) \cdot \ttw^{\mot,\ast,V}(k_1,\ldots,k_{d-1}) \) term  combines with the original term
	\[
		- \delta_{k_d = 1} \widetilde{\pi}_1 \big( \log^\lmot(2)  \big)  \ttw^{\mot,\ast,V}(k_1,\ldots,k_{d-1}) 
	\]
	coming from deconcatenating at the end.  (This arises in \eqref{eqn:drst:eps0}; by the same argument as in \autoref{prop:d1}, one knows that only the two extremal terms, removing initial or terminal 1's, actually contribute to \( D_1 \).)  One then introduces the linear map \( \partial^{H,\ast}_{N,\ell} \) as
	\[
	\partial_{N,\ell}^{H,\ast} \colon \gr_\ell^{H,\ast} \mathcal{H}_N^{H,\ast} \to \bigoplus_{1 \leq 2r+1 \leq N} \gr_{\ell-1}^{H,\ast} \mathcal{H}_{N-2r-1}^{H,\ast} \,,
	\]
	by first applying \( \bigoplus_{1\leq2r+1\leq N} \gr_\ell^{H,\ast} D_{2r+1} \big|_{\gr_\ell^{H,\ast} \mathcal{H}_N^{H,\ast}} \), then projecting \( \zeta^\lmot(2r'+1) \,, \log^\lmot(2) \) to \( \Q \) via \( \widetilde{\pi}_{2r'+1} \) as in \autoref{def:partialH}.
	
	Then define the matrix \( M_{H,\ast,N,\ell} \) as the matrix of \( \partial^{H,\ast}_{N,\ell} \) as in \autoref{def:Mpar}, with respect to the bases \( B_{H,N,\ell}, B'_{H,N,\ell} \) given in \autoref{def:Mbas}, after replacing \( \ttw^\mot \) with \( \ttw^{\mot,\ast,V} \).
	
	Now, one notes that the to compute the matrix \( M_{H,\ast,N,\ell} \) one replaces each term \( -\frac{1}{2} \) in the matrix \( M_{H,N,\ell} \) arising from \( \widetilde{\pi}_1 \big( \log^\lmot(2) \big) = \frac{1}{2} \)  by the coefficient \( \frac{1}{2} ( \lambda - 1) - \frac{1}{2} = \lambda - 1 \).
	
	\begin{Eg}
		For \( N = 8, \ell = 2 \), the matrix \( M_{H,\ast,8,2} \) is as follows; the first row and column label the elements of \( B'_{H,8,2} \) and \( B_{H,8,2} \) respectively.
		\begin{center}
			\small
			\begin{adjustwidth}{-0.4cm}{}
				\begin{tabular}{c||c|c|c|c|c|c||c|c|c|c}
					& \!\!1222\!\! &  \!\!2122\!\!  & 122 & \!\!2212\!\! & 212 & 12 & \!\!2221\!\!\! & 221 & 21 & 1 \\ \hline\hline
					\!\!11222\!\! & $ 1 $ & $ 0 $ & $ {-}2 c_{21} $ & $ 0 $ & $ 0 $ & $ {-}8 c_{221} $ & $ 0 $ & $ 0 $ & $ 0 $ & $ 0 $ \\
					\!\!12122\!\! & $ 0 $ & $ 1 $ & $ \!\!2 d_{12}{-}2 c_{21}\!\!\! $ & $ 0 $ & $ 0 $ & $ 8 c_{23}{-}8 c_{32} $ & $ 0 $ & $ 0 $ & $ 0 $ & $ 0 $ \\
					\!\!21122\!\! & $ 0 $ & $ 0 $ & $ 2 d_{21} $ & $ 0 $ & $ {-}2 c_{21} $ & $ 0 $ & $ 0 $ & $ 0 $ & $ 0 $ & $ 0 $ \\
					\!\!12212\!\! & $ 0 $ & $ 0 $ & $ 2 c_{21} $ & $ 1 $ & $ \!\!2 d_{12}{-}2 c_{21}\!\!\! $ & $ \!\!{-}8 c_{23}+8 c_{32}{+}8 d_{122}\!\!\! $ & $ 0 $ & $ 0 $ & $ 0 $ & $ 0 $ \\
					\!\!21212\!\! & $ 0 $ & $ 0 $ & $ 0 $ & $ 0 $ & $ 2 d_{21} $ & $ 8 d_{212} $ & $ 0 $ & $ 0 $ & $ 0 $ & $ 0 $ \\
					\!\!22112\!\! & $ 0 $ & $ 0 $ & $ 0 $ & $ 0 $ & $ 2 c_{21} $ & $ 8 d_{221} $ & $ 0 $ & $ 0 $ & $ 0 $ & $ 0 $ \\ \hline \hline
					\!\!12221\!\! & $ \!\!\lambda {-} 1\!\! $ & $ 0 $ & $ 2 c_{21} $ & $ 0 $ & $ 0 $ & $ 8 c_{221} $ & $ 1 $ & $ \!\!2 d_{12}{-}2 c_{21}\!\! $ & $ 8 d_{122}{-}8 c_{221} $ & $ \!\!32 d_{1222}\!\! $ \\
					\!\!21221\!\! & $ 0 $ & $ \!\!\lambda {-} 1\!\! $ & $ 0 $ & $ 0 $ & $ 2 c_{21} $ & $ 0 $ & $ 0 $ & $ \!\!2 d_{21}{-}2 c_{21}\!\! $ & $ \!\! 8 c_{23}{-}8 c_{32}{+}8 d_{212} \!\!\! $ & $ \!\!32 d_{2122}\!\! $ \\
					\!\!22121\!\! & $ 0 $ & $ 0 $ & $ 0 $ & $ \!\!\lambda {-} 1\!\! $ & $ 0 $ & $ 0 $ & $ 0 $ & $ 2 c_{21} $ & $\!\! {-}8 c_{23}{+}8 c_{32}{+}8 d_{221} \!\!\!$ & $ \!\!32 d_{2212}\!\! $ \\
					\!\!22211\!\! & $ 0 $ & $ 0 $ & $ 0 $ & $ 0 $ & $ 0 $ & $ 0 $ & $ \!\!\lambda {-} 1\!\! $ & $ 2 c_{21} $ & $ 8 c_{221} $ & $ \!\!32 d_{2221}\!\! $ \\
				\end{tabular}
			\end{adjustwidth}
		\end{center}
	\end{Eg}
	
	We note now that when \( \lambda = \frac{1}{2} \), the matrices \( M_{H,\ast,N,\ell} \) and \( M_{H,N,\ell} \) are identical, and therefore the stuffle-regularised matrix is also invertible.  Moreover, when \( \lambda = 1 \) the last diagonal block (corresponding to the original block \( M_{\ell-1,H,N,\ell} \) ending in \( \ell-1 \) trailing 1's) is now upper triangular modulo 2 (compare \autoref{lem:block:isdeconat1}), and so also again establishes that \( M_{H,\ast,N,\ell} \) is an invertible matrix.  More generally we have the following.
	
	\begin{Prop}\label{prop:mats:inv}
		Suppose \( \lambda \) has the form \( \frac{2a+1}{b} \in \Q \), with \( a, b \in \Z \).  Then the matrix \( M_{H,\ast,N,\ell} \) is invertible.
		
		\begin{proof}
			The previous result \autoref{lem:mat:blocktri} carries through to show the matrix is block triangular.  The result \autoref{lem:block:isdeconcat} also carries over to show the diagonal blocks corresponding to \( < \ell-1 \) trailing 1's are upper triangular modulo 2, and are therefore invertible.  The proof of \autoref{lem:block:isdeconat1} is adapted to show that the determinant of the last block has the form
			\[
				(2x + 1) + (\lambda - 1) (2y + 1) \,,
			\]
			with \( x, y \in \Z \).  For \( \lambda \) of the above form, this is
			\[
				2(x-y) + \frac{(2a+1) (2y+1)}{b} \,,
			\]
			which cannot be 0, as the numerator of the fraction is odd.
		\end{proof}
	\end{Prop}
	
	The proofs of \autoref{cor:hoff:indep} and \autoref{cor:text:eq:H2} now directly generalise to this case, giving
	\begin{Cor}\label{cor:hoff:st:indep}
		Let \( V = \lambda \log^\mot(2) \), with \( \lambda = \frac{2a+1}{b} \in \mathbb{Q} \) and \( a, b \in \Z \).  Then the elements 
		\[
		\{ \ttw^{\mot,\ast,V}(w) \mid w \in \{1,2\}^\times  \} \,,
		\]
		are linearly independent.  Moreover, they form a basis for the space of:
		\begin{itemize} 
			\item[i)] motivic extended stuffle-regularised multiple \( t \) values with \( \ttw^{\mot,V,\ast}(1) = 2V \),
			\item[i$'$)] motivic extended shuffle-regularised multiple \( t \) values,
			\item[ii)] alternating (shuffle-regularised) motivic multiple zeta values
		\end{itemize}
		In particular all of these spaces agree, and extended stuffle-regularised motivic multiple \( t \) values with \( \ttw^{\mot,\ast,V}(1) = 2V\) have dimension \( F_{N+1} \) in weight \( N \), where \( F_{k} = F_{k-1} + F_{k-2} \) with \( F_0 = F_1 = 1 \) is the sequence of Fibonacci numbers.
	\end{Cor}

	\subsection{Singular regularisation parameters} \label{rem:singular:reg}
	
		The proof of \autoref{prop:mats:inv} breaks down irrevocably in certain cases, in a way that is unavoidable.  For example, for \( N = 8, \ell = 2 \) as above, one can see that \( \lambda = \frac{242}{91} \) leads to determinant 0 in the last diagonal block.  This corresponds to the a linear dependence between regularised elements of level \( \ell \leq 2 \) in weight \( 8 \).
		
		For \( V = \frac{242}{91} \log(2) \), one has the following identity between stuffle-regularised MtV's (and a corresponding identity of stuffle-regularised motivic MtV's), as verified via the Data Mine \cite{mzvDM}
		\begin{align*}
			 t^{\ast,V}(2,2,2,1,1) = {} 
			& \tfrac{345998 }{24843}t^{\ast,V}(2,2,2,2) \\
			& -\tfrac{22801 }{8281}t^{\ast,V}(1,1,2,2,2)
			-\tfrac{11023 }{8281}t^{\ast,V}(1,2,1,2,2)
			+\tfrac{1661}{1183}  t^{\ast,V}(1,2,2,1,2) \\
			&-\tfrac{22801}{8281}  t^{\ast,V}(2,1,1,2,2)
			-\tfrac{919}{637} t^{\ast,V}(2,1,2,1,2)
			-\tfrac{17257 }{8281}t^{\ast,V}(2,2,1,1,2) \\
			& +\tfrac{151}{91} t^{\ast,V}(1,2,2,2,1) 
			+\tfrac{73}{91}	t^{\ast,V}(2,1,2,2,1)
			-\tfrac{11}{13} t^{\ast,V}(2,2,1,2,1) \,.
		\end{align*}
		
		Such a regularisation parameter should be termed \emph{singular}, as the matrix of \( \partial^{H,\ast}_{N,\ell} \) is singular.  The following \( V = \lambda \log^\mot(2) \) are singular regularisation parameters, first appearing at the indicated weight. \medskip
		\begin{center}
			\begin{tabular}{c|cccccccccc}
				$\!\!N\!\!$ & 1 & 3 & 5 & 7 & 9 & 11 & 13 & 15 & 17 & 19\\ \hline
				$ \!\!\lambda\!\!$ $\phantom{\mathllap{\displaystyle\frac{1}{1}}}$
				 & 0 & 2 
				& \!\!$\frac{28}{11}$\!\! 
				& \!\!$\frac{242}{91}$\!\! 
				& \!\!$\frac{64472}{23479}$\!\! 
				& \!$\frac{712586}{252913}$\!
				& \!$\frac{8156772916}{2873825507}$\!
				& \!$\frac{1002618956134}{348754372637}$\!
				& \!$\frac{6597362406922672}{2270331930729959}$\!
				& \!$\frac{91024278619403627042}{31196250956544565801}$\!
		\end{tabular}
		\end{center}\medskip
		A (potential) new parameter \( \lambda \) appears in level 1, odd weight, corresponding to last diagonal block with 0 trailing \( 1 \)'s, and a reduction of 
		\[
		\ttw^{\mot,\ast,V}(\{2\}^n,1) = \sum_{i=0}^{n+1} c_i \ttw^{\mot}(\{2\}^i, 1, \{2\}^{n-i}) \,,
		\]
		for some \( c_i \in \Q \).  In weight \( 2N+1 \) this reduction can also be obtained more directly from the identity in \autoref{thm:mott2212}, when written in matrix form with rows indexed by \( c_i \) and columns by \( \zeta^\mot(\overline{2r'+1}) \ttw^\mot(\{2\}^{N-r}) \), which essentially encodes (the last block of) such \( M_{H,\ast,2N+1,\ell=1} \).
		
		Once such a parameter appears, it renders nonsensical the matrices \( M_{H,\ast,N,\ell} \) of higher weight in that level, as the basis of lower level elements \( B'_{H,N,\ell} \)  is no longer linearly independent.  One can strip trailing 1's from the last diagonal block, without changing the combinatorial of the matrix entries, to see that every singular regularisation parameter arises from the level 1 relation. \medskip
	
		The sequence \( (\lambda_i)_{i=1}^\infty \) of singular regularisation parameters appears to satisfy a number of properties.  We end with the following conjecture.
	
		\begin{Conj}
			The sequence \( (\lambda_i)_{i=1}^\infty = (0, 2, \frac{28}{11}, \frac{242}{91}, \frac{64472}{23479}, \ldots) \) of singular regularisation parameters satisfies the following:
			\begin{itemize}
				\item[i)] the sequence is increasing \( \lambda_{i+1} > \lambda_i \),
				\item[ii)] the sequence is bounded \( \lambda_i < 3 \), for all \( i \),
				\item[ii$'$)] the sequence has limit \( \lim_{i\to\infty} \lambda_i = 3 \)
			\end{itemize}
		\end{Conj}
	
	\bibliographystyle{habbrv}
	\bibliography{mot_t2212}

\begin{thebibliography}{10}
\expandafter\ifx\csname url\endcsname\relax
  \def\url#1{\texttt{#1}}\fi
\expandafter\ifx\csname doi\endcsname\relax
  \def\doi#1{\burlalt{doi:#1}{http://dx.doi.org/#1}}\fi
\expandafter\ifx\csname urlprefix\endcsname\relax\def\urlprefix{URL }\fi
\expandafter\ifx\csname href\endcsname\relax
  \def\href#1#2{#2}\fi
\expandafter\ifx\csname burlalt\endcsname\relax
  \def\burlalt#1#2{\href{#2}{#1}}\fi

\bibitem{bailey64}
W.~N. Bailey.
\newblock {\em Generalized hypergeometric series}.
\newblock Cambridge Tracts in Mathematics and Mathematical Physics, No. 32.
  Stechert-Hafner, Inc., New York, 1964.

\bibitem{mzvDM}
J.~Bl\"{u}mlein, D.~J. Broadhurst, and J.~A.~M. Vermaseren.
\newblock The multiple zeta value data mine.
\newblock {\em Comput. Phys. Comm.}, 181(3):582--625, 2010,
  arXiv:\burlalt{0907.2557}{http://arxiv.org/abs/0907.2557}.
\newblock \doi{10.1016/j.cpc.2009.11.007}.

\bibitem{broadhurstKreimer97}
D.~J. Broadhurst and D.~Kreimer.
\newblock Association of multiple zeta values with positive knots via {F}eynman
  diagrams up to {$9$} loops.
\newblock {\em Phys. Lett. B}, 393(3-4):403--412, 1997,
  arXiv:\burlalt{hep-th/9609128}{http://arxiv.org/abs/hep-th/9609128}.
\newblock \doi{10.1016/S0370-2693(96)01623-1}.

\bibitem{brown12}
F.~Brown.
\newblock Mixed {T}ate motives over {$\mathbb{Z}$}.
\newblock {\em Ann. of Math. (2)}, 175(2):949--976, 2012.
\newblock \doi{10.4007/annals.2012.175.2.10}.

\bibitem{brown17}
F.~Brown.
\newblock Notes on motivic periods.
\newblock {\em Commun. Number Theory Phys.}, 11(3):557--655, 2017,
  arXiv:\burlalt{1512.06410}{http://arxiv.org/abs/1512.06410}.
\newblock \doi{10.4310/CNTP.2017.v11.n3.a2}.

\bibitem{chavan21}
S.~Chavan, M.~Kobayashi, and J.~Layja.
\newblock Integral evaluation of odd {E}uler sums, multiple $t$-value
  $t\left(3,2,\ldots,2\right)$ and multiple zeta value $\zeta(3,2,\ldots,2)$,
  2021, arXiv:\burlalt{2111.07097}{http://arxiv.org/abs/2111.07097}.

\bibitem{deligneGoncharov}
P.~Deligne and A.~B. Goncharov.
\newblock Groupes fondamentaux motiviques de {T}ate mixte.
\newblock {\em Ann. Sci. \'{E}cole Norm. Sup. (4)}, 38(1):1--56, 2005,
  arXiv:\burlalt{math/0302267}{http://arxiv.org/abs/math/0302267}.
\newblock \doi{10.1016/j.ansens.2004.11.001}.

\bibitem{evans84}
R.~J. Evans and D.~Stanton.
\newblock Asymptotic formulas for zero-balanced hypergeometric series.
\newblock {\em SIAM J. Math. Anal.}, 15(5):1010--1020, 1984.
\newblock \doi{10.1137/0515078}.

\bibitem{fresanBook}
J.~B. Gil and J.~Fres{\'a}n.
\newblock {\em Multiple zeta values: from numbers to motives}.
\newblock Clay Mathematics Proceedings, to appear, 2017.
\newblock \urlprefix\url{http://javier.fresan.perso.math.cnrs.fr/mzv.pdf}.

\bibitem{glanoisTh}
C.~Glanois.
\newblock {\em Periods of the motivic fundamental groupoid of \( \mathbb{P}^1
  \setminus \{ 0, \mu_N, \infty \} \)}.
\newblock PhD thesis, Pierre and Marie Curie University (Paris 6), 2016,
  arXiv:\burlalt{1603.05155}{http://arxiv.org/abs/1603.05155}.
\newblock \urlprefix\url{https://www.theses.fr/2016PA066013}.

\bibitem{glanoisUnramified}
C.~Glanois.
\newblock Unramified {E}uler sums and {H}offman $\star$ basis, 2016,
  arXiv:\burlalt{1603.05178}{http://arxiv.org/abs/1603.05178}.

\bibitem{goncharov01}
A.~B. Goncharov.
\newblock Multiple polylogarithms and mixed {T}ate motives, 2001,
  arXiv:\burlalt{math/0103059}{http://arxiv.org/abs/math/0103059}.

\bibitem{goncharov05}
A.~B. Goncharov.
\newblock Galois symmetries of fundamental groupoids and noncommutative
  geometry.
\newblock {\em Duke Math. J.}, 128(2):209--284, 2005,
  arXiv:\burlalt{math/0208144}{http://arxiv.org/abs/math/0208144}.
\newblock \doi{10.1215/S0012-7094-04-12822-2}.

\bibitem{hoffman92}
M.~E. Hoffman.
\newblock Multiple harmonic series.
\newblock {\em Pacific J. Math.}, 152(2):275--290, 1992.
\newblock \urlprefix\url{http://projecteuclid.org/euclid.pjm/1102636166}.

\bibitem{hoffman19}
M.~E. Hoffman.
\newblock An odd variant of multiple zeta values.
\newblock {\em Commun. Number Theory Phys.}, 13(3):529--567, 2019,
  arXiv:\burlalt{1612.05232}{http://arxiv.org/abs/1612.05232}.

\bibitem{ikz06}
K.~Ihara, M.~Kaneko, and D.~Zagier.
\newblock Derivation and double shuffle relations for multiple zeta values.
\newblock {\em Compos. Math.}, 142(2):307--338, 2006.
\newblock \doi{10.1112/S0010437X0500182X}.

\bibitem{LewinBook}
L.~Lewin.
\newblock {\em {Polylogarithms and associated functions}}.
\newblock North Holland, New York, 1982.

\bibitem{murakami21}
T.~Murakami.
\newblock On {H}offman’s $t$-values of maximal height and generators of
  multiple zeta values.
\newblock {\em Mathematische Annalen}, pages 1--38, 2021.
\newblock \doi{10.1007/s00208-021-02209-3}.

\bibitem{saha17}
B.~Saha.
\newblock A conjecture about multiple $t$-values, 2017,
  arXiv:\burlalt{1712.06325}{http://arxiv.org/abs/1712.06325}.

\bibitem{zagier94}
D.~Zagier.
\newblock Values of zeta functions and their applications.
\newblock In {\em First {E}uropean {C}ongress of {M}athematics, {V}ol. {II}
  ({P}aris, 1992)}, volume 120 of {\em Progr. Math.}, pages 497--512.
  Birkh\"{a}user, Basel, 1994.

\bibitem{zagier2232}
D.~Zagier.
\newblock Evaluation of the multiple zeta values
  {$\zeta(2,\ldots,2,3,2,\ldots,2)$}.
\newblock {\em Ann. of Math. (2)}, 175(2):977--1000, 2012.
\newblock \doi{10.4007/annals.2012.175.2.11}.

\bibitem{zhaoBook}
J.~Zhao.
\newblock {\em Multiple zeta functions, multiple polylogarithms and their
  special values}, volume~12 of {\em Series on Number Theory and its
  Applications}.
\newblock World Scientific Publishing Co. Pte. Ltd., Hackensack, NJ, 2016.
\newblock \doi{10.1142/9634}.

\end{thebibliography}
	
\end{document}